\documentclass[10pt, reqno]{amsart}
\usepackage[utf8x]{inputenc}
\usepackage[margin=1.0in]{geometry}
\usepackage{amsfonts}
\usepackage{amsmath}
\usepackage{amssymb}
\usepackage{amsthm}
\usepackage{hyperref}
\usepackage{verbatim}
\usepackage[T1]{fontenc}
\usepackage{enumitem}
\usepackage{palatino}
\usepackage{tikz}
\usepackage{mathrsfs}
\usetikzlibrary{trees}
\usepackage{dsfont}
\usepackage{upgreek}
\usepackage[titletoc,title]{appendix}
\usepackage{pgfplots}
\usepackage{float}

\newcommand{\R}{\mathbb{R}}
\newcommand{\N}{\mathbb{N}}

\newcommand{\pr}{\mathbb{P}}
\newcommand{\ex}{\mathbb{E}}

\newcommand{\h}{\mathcal{H}}
\newcommand{\e}{\mathcal{E}}

\newcommand{\blangle}{\big\langle}
\newcommand{\brangle}{\big\rangle}

\newtheorem{lem}{Lemma}

\newtheorem{prop}{Proposition}
\newtheorem{thm}{Theorem}
\newtheorem{dfn}{Definition}
\newtheorem{innercustomthm}{Hypothesis}
\newenvironment{customthm}[1]
{\renewcommand\theinnercustomthm{#1}\innercustomthm}
{\endinnercustomthm}

\theoremstyle{definition}
\newtheorem{rem}{Remark}
\makeatletter
\newcommand{\leqnomode}{\tagsleft@true\let\veqno\@@leqno}
\newcommand{\reqnomode}{\tagsleft@false\let\veqno\@@eqno}
\makeatother

\numberwithin{lem}{section}
\numberwithin{cor}{section}
\numberwithin{prop}{section}
\numberwithin{thm}{section}
\numberwithin{dfn}{section}
\title{Importance sampling for stochastic reaction-diffusion equations in the moderate deviation regime}
\author{Ioannis Gasteratos}
\address[Ioannis Gasteratos]{Imperial College London, Department of Mathematics}
\email{i.gasteratos@imperial.ac.uk}

 \author{Michael Salins}
 \address[Michael Salins]{Boston University, Department of Mathematics and Statistics}
 \email{msalins@bu.edu}
 \author{Konstantinos Spiliopoulos}
 \address[Konstantinos Spiliopoulos]{Boston University, Department of Mathematics and Statistics}
 \email{kspiliop@bu.edu}
	\thanks{\noindent\textit{2010 Mathematics Subject Classification}. 65C05, 60G99, 60F9.\\  \textit{Key words and phrases}. rare event simulation, importance sampling, Monte Carlo methods, moderate deviations, stochastic reaction-diffusion equations, optimal control.\\
	 This work was partially supported by the National Science Foundation (DMS 1550918, DMS 2107856) and Simons Foundation Award 672441}
\begin{document}
\maketitle
\begin{abstract}
	We develop a provably efficient importance sampling scheme that estimates exit probabilities of solutions to small-noise stochastic reaction-diffusion equations from scaled neighborhoods of a stable equilibrium. The moderate deviation scaling allows for a local approximation of the nonlinear dynamics by their linearized version. In addition, we identify a finite-dimensional subspace where exits take place with high probability. Using stochastic control and variational methods we show that our scheme performs well both in the zero noise limit and pre-asymptotically. Simulation studies for stochastically perturbed bistable dynamics illustrate the theoretical results.
\end{abstract}

\section{Introduction}\label{Sec1}
    \noindent In this paper we are concerned with the problem of rare event simulation for the stochastic reaction-diffusion equation (SRDE)
	          \begin{equation}\label{model}
	          \left\{\begin{aligned}
	          &\partial_tX^\epsilon(t,\xi)=\mathcal{A}X^{\epsilon}(t,\xi)+f\big( X^\epsilon(t,\xi)\big)+\sqrt{\epsilon}\dot{W}(t,\xi)\;,\;\;(t,\xi)\in [0,\infty)\times(0,\ell)
	          \\&X^{\epsilon}(0,\xi)=x(\xi), \xi\in(0,\ell),\;\mathcal{N}X^{\epsilon}(t,\xi)=0,\; (t,\xi)\in[0,\infty)\times\{0,\ell\},
	          \end{aligned}\right.
	          \end{equation}
	         where $\epsilon\ll 1,$ $\mathcal{A}$ is a uniformly elliptic second-order differential operator, $f:\R\rightarrow\R$ is a dissipative nonlinearity with polynomial growth and  $\dot{W}$ is a stochastic forcing term of intensity $\sqrt{\epsilon}$ modeled by space-time white noise. The mixed boundary conditions are given by the linear operator $\mathcal{N}$ which acts on functions defined on the boundary  $\partial(0,\ell)$ (see Section \ref{Sec2} for more details), and the initial datum $x:(0,\ell)\rightarrow\R$ is a continuous function in the kernel of $\mathcal{N}.$
	
	             Systems like \eqref{model} are of interest because they exhibit metastable behavior. Assuming that the associated noiseless dynamics are non-trivial and $\epsilon>0$, the stochastic forcing can induce transitions between neighborhoods of metastable states. As $\epsilon\to 0$, transitions and exits from domains of attraction occur with very small probabilities and rigorous asymptotic analysis of exit times and places is possible within the framework of large deviations or potential theory (see e.g. \cite{day1990large,freidlin1998random,gayrard2004metastability,gayrard2005metastability} and \cite{berglund2019introduction,da1991minimum,faris1982large,Hgele2011MetastabilityOT,salins2021metastability}, as well as references within, for results in metastability theory in finite and infinite dimensions respectively).
	
	                  In practice, efficient simulation of such events is challenging. On the one hand, Large Deviation Principles (LDPs) characterize the exponential decay rates of probabilities in the limit as $\epsilon\to0$ but ignore the effect of prefactors which can be significant (see \cite{DupuisSpilZhou}). On the other hand, as $\epsilon$ decreases, standard Monte-Carlo schemes require an increasingly large sample size in order to maintain a small relative error per sample. For this reason, accelerated and adaptive methods such as importance sampling or multi-level splitting become essential when it comes to rare events. For more details on the general theory and applications of such methods in a number of different models, the interested reader is referred to the book \cite{budhiraja2019analysis}.
	
	                    In the present work, we aim to develop a provably efficient importance sampling scheme that computes exit probabilities of $X^\epsilon$ from scaled neighborhoods of a stable equilibrium point $x^*$. In particular, let $X^\epsilon_x$ denote the unique (mild) solution of \eqref{model} with initial condition $x,$ $D\subset L^2(0,\ell)$ and
	                    \begin{equation*}
	                      \tau_{x^*}^\epsilon=\inf\{ t>0:  X_{x^*}^{\epsilon}(t)\notin D\}.
	                    \end{equation*}
	                    For $T, L>0,$ we focus on the estimation of probabilities $\pr[\tau_{x^*}^\epsilon\leq T  ]$ in the case where $D=D_\epsilon$ with
	                    \begin{equation}\label{Domain}
	                    D_\epsilon=\big\{ x\in L^2(0,\ell): \|x-x^*\|_{L^2}< L\sqrt{\epsilon}h(\epsilon) \big\}.
	                    \end{equation}
	                    The scaling $h(\epsilon)$ is chosen so that $h(\epsilon)\to\infty$  and $\sqrt{\epsilon}h(\epsilon)\to 0,$ as $\epsilon\to 0.$ As $\epsilon\to 0$, exit probabilities from such domains lie in an asymptotic regime that interpolates between the Central Limit Theorem (CLT) and LDP. To be precise, let $X^0_x$ denote the (deterministic) solution of \eqref{model} with $\epsilon=0$ and define a family of centered and re-scaled processes
	                    \begin{equation}\label{etadef}
	                       \eta_x^\epsilon:=\frac{X_x^\epsilon-X^0_x}{\sqrt{\epsilon}h(\epsilon)}\;,\;\; \epsilon>0.
	                    \end{equation}
	                     As $\epsilon\to 0,$ the choices  $h(\epsilon)=1/\sqrt{\epsilon}$ and $h(\epsilon)=1$ correspond to large and Gaussian deviations of $X^{\epsilon}$ respectively. 
	                      Exits of $X^\epsilon$ from $D$ are then equivalent to exits of $\eta_x^\epsilon$ from an $L^2-$ball of radius $L$ around $0$ and large deviations of the family $\{\eta_x^\epsilon\}_{\epsilon\in (0,1)}$ are called moderate deviations of $\{X_x^\epsilon\}_{\epsilon\in (0,1)}.$  Moderate Deviation Principles (MDPs) have been studied in many different contexts such as multiscale and interacting particle systems, Markov processes with jumps, small-noise stochastic dynamics, statistical estimation, option pricing and stochastic recursive algorithms see e.g. \cite{gasteratos2022moderate, wang2015moderate} for SRDEs as well as \cite{bezemek2022moderate,budhiraja2016moderate,dupuis2015moderate,gao2011delta,grama1997moderate,guillin2003averaging,jacquier2020pathwise,kallenberg1983moderate,morse2017moderate}.
	
	                    Importance sampling is a variance-reduction accelerated Monte-Carlo method and its objective is to minimize the variance of the estimator by carefully chosen changes of measure. Such changes of measure "push" the dynamics towards trajectories that realize the rare event of interest. This procedure transforms tail events to more typical events, thus allowing for more efficient sampling. The simulation outcomes are then weighted by likelihood ratios so that the importance sampling estimators remain unbiased under the new probability measures. Importance sampling schemes for events in the large and moderate deviation regimes have been developed for finite-dimensional systems in \cite{dupuis2017moderate,DupuisSpilZhou,spiliopoulos2017importance,spiliopoulos2020importance,vanden2012rare}. In \cite{dupuis2017moderate,spiliopoulos2020importance}, the authors observed that moderate-deviation based schemes provide a viable and simpler alternative to their large-deviation based counterparts, in cases where both are applicable. This is due to the fact that the MDP action functional, which characterizes exponential decay rates of probabilities, takes a much simpler form. In turn, this allows for more tractable and straightforward design of optimal changes of measure. 
	
	                      Importance sampling for SRDEs presents new challenges due to infinite dimensionality combined with the nonlinearity of the dynamics. Our work is close to \cite{salins2017rare} where a large deviation based scheme was developed for linear equations (i.e. when $f=0$). In there, the authors show that efficient changes of measure need to accomplish both variance and dimension reduction. For example, changes of measure that force infinitely many modes of the dynamics lead to estimators with very large variance when $\epsilon$ is small. A possible workaround is to show that exits from $D$ take place in a finite-dimensional submanifold of $\partial D$ with high probability. This was achieved in the linear case of \cite{salins2017rare} where it was proved that, under a sufficiently large spectral gap, exit from $D$ happens in the direction of the eigenvector $e_1$ of $-\mathcal{A}$ corresponding to the smallest non-zero eigenvalue. Similar results regarding the exit direction for  (finite-dimensional) SDEs with a linear drift have been proved in \cite{zabczyk1985exit} (see also Remark \ref{Zabcszykrem} below).
	
	                       To the best of our knowledge, importance sampling for nonlinear SRDEs is rigorously studied here for the first time. The main difficulty in designing large deviation-based schemes for such equations lies in the task of identifying a finite-dimensional exit submanifold (if any). We are able to overcome this obstacle by working in the moderate deviation regime. As we show in the sequel, the latter is equivalent to linearizing the dynamics in a neighborhood of the equilibrium $x^*$.  Consequently, the results of \cite{salins2017rare} can be applied locally at the cost of a linearization error which is, however, negligible as $\epsilon\to 0$. In cases where both LDP and MDP-based schemes are available, one may think of the tradeoff between the two as follows: Moderate deviations cover the regime between central limit theorem and large deviations, so they are appropriate to characterize rare events, but not so rare that they would be in the large deviations regime. On the other hand, moderate deviations schemes are in general more tractable due to the  asymptotic linearization of the dynamics that takes place. In our setting, this tradeoff is reflected in the fact that we only consider exit domains \eqref{Domain} in which the radius shrinks to zero as $\epsilon\to0.$ Furthermore, the probability of exiting from a ball of radius $\sqrt{\epsilon}h(\epsilon)$ is strictly smaller than the probability of exiting a ball of radius $1$. The MDP importance sampling schemes described in this paper can provide a quantitative upper bound for the much more difficult to characterize LDP exit probabilities.

	                       The design of an importance sampling scheme and proof of its good asymptotic and pre-asymptotic performance is the main contribution of this paper. In the course of our analysis, we prove an MDP for additive-noise SRDEs with a non-Lipschitz nonlinearity which cannot be found in the literature (see Theorem \ref{MDPthm} and Remark \ref{MDPrem}). Furthermore, our theory is applied to the stochastic Allen-Cahn (also known as real Ginzburg-Landau or Chafee-Infante) equation and supplemented by simulation studies. In contrast to the linear case, there is a number of interesting cases where the aforementioned spectral gap is not satisfied. Another novel feature of this work is the construction of changes of measure that perform well asymptotically (i.e. as $\epsilon\to0$) in the absence of this condition (see Hypothesis \ref{A3c'} below).
	
	                            The rest of this paper is organized as follows: In Section \ref{Sec2} we fix the notation and state our assumptions. In the first part of Section \ref{Sec3} we introduce moderate deviations and subsolution-based importance sampling and then state and prove our results on the asymptotic theory of the scheme. Section \ref{Sec4} is devoted to the implementation and pre-asymptotic performance analysis of our scheme. In Section \ref{Sec5} we apply the developed theory to the
	                            case where $f$ is, up to a sign, the derivative of a double-well potential. Our examples include the stochastic Allen-Cahn equation  (which features a cubic nonlinearity) with different boundary conditions as well as SRDEs with higher order polynomial nonlinearities. The results of simulation studies are then presented in Section \ref{Sec6}. Finally, Appendix \ref{AppA} collects the proofs of some useful lemmas.

\section{Notation and Assumptions}\label{Sec2}

 Let $\ell>0.$   The Hilbert space $L^2(0,\ell)$ endowed with its usual inner product will be denoted by $(\h,\langle \cdot,\cdot\rangle_\h)$. The Banach space
 $C[0,\ell],$ endowed with the supremum norm, is denoted by $\mathcal{E}.$
 The norm of a Banach space $\mathcal{X}$ will be denoted by $ \|\cdot\|_{\mathcal{X}}$ and the closed ball of radius $R>0$ and center $x_0\in\mathcal{X}$, i.e. the set $\{x\in\mathcal{X} : \|x-x_0\|_{\mathcal{X}}\leq R\}$, by $B_\mathcal{X}(x_0,R)$. We use $\mathring{D},\bar{D},\partial D $ to denote interior, closure and boundary of a set $D\subset\mathcal{X}$ respectively. The lattice notation $\wedge, \vee$ is used to indicate minimum and maximum respectively.

For $\theta> 0,$ $p\in [1,\infty),$ we denote by $W^{p,\theta}(0,\ell)$ the fractional Sobolev space of $x\in L^p(0,\ell)$ such that
\begin{equation*}
[x]^p_{p,\theta}:=\iint_{[0,\ell]^2}\frac{|x(\xi_2)-x(\xi_1)|^p}{|\xi_2-\xi_1|^{p\theta+1}}d\xi_1 d\xi_2<\infty.
\end{equation*}
$W^{p,\theta}(0,\ell),$  endowed with the norm $\|\cdot\|_{p,\theta}:=\|\cdot\|_{L^p(0,\ell)}+[\cdot]_{p,\theta},$ is a Banach space.  $W^{2,\theta}(0,\ell)$ is a Hilbert space and is denoted by $H^{\theta}(0,\ell).$
Moreover, for $T>0$ and $\beta\in[0,1)$, we denote by $C^\beta([0,T];\mathcal{X})$ the space of $\beta$-H\"older continuous $\mathcal{X}$-valued functions defined on the interval $[0,T]$. $C^\beta([0,T];\mathcal{X}),$ endowed with the norm \begin{equation*}
\|X\|_{C^\beta([0,T];\mathcal{X})}:= \|X\|_{C([0,T];\mathcal{X})}+[X]_{C^\beta([0,T];\mathcal{X})}:=\sup_{t\in[0,T]}\|X(t)\|_{\mathcal{X}}+\sup_{\overset{s,t\in[0, T]}{ t\neq s}}\frac{\|X(t)-X(s)\|_\mathcal{X}}{|t-s|^\beta}\;,
\end{equation*}
is a Banach space.

 For any two Banach spaces $\mathcal{X}, \mathcal{Y}$ we denote the space of linear bounded operators $B: \mathcal{X}\rightarrow\mathcal{Y}$ by $\mathscr{L}(\mathcal{X}; \mathcal{Y})$. The latter is a Banach space when endowed with the norm $\|B\|_{\mathscr{L}(\mathcal{X}; \mathcal{Y})}:=\sup_{x\in B_\mathcal{X}(0,1)}\|Bx\|_{\mathcal{Y}}$.
When the domain coincides with the co-domain, we use the simpler notation $\mathscr{L}(\mathcal{X}).$ 
The spaces of trace-class and Hilbert-Schmidt linear operators $B:\h\rightarrow\h$ are denoted by $\mathscr{L}_1(\h)$ and $\mathscr{L}_2(\h)$ respectively. The former is a Banach space
when endowed with the norm $\|B\|_{\mathscr{L}_1(\h)}:=\text{tr}(\sqrt{B^*B})$
while the latter is a Hilbert space when endowed with the inner product
$\langle B_1, B_2\rangle_{\mathscr{L}_2(\h)}:= \text{tr}( B_2^* B_1)$.


The operator $\mathcal{A}$ in \eqref{model} is a uniformly elliptic second-order differential operator in divergence form. In particular:
      \begin{equation}
      \label{scriptA}
      \mathcal{A}\phi(\xi)=\frac{d}{d\xi}\bigg(a(\xi)\frac{d\phi(\xi)}{d\xi}\bigg)\;, \xi\in(0,\ell)
      \end{equation}
with $a\in C^1(0,\ell)$ and $\inf_{\xi\in(0,\ell)}a(\xi)>0$. The operator $\mathcal{N}$ acts on the boundary $\{0,\ell\}$ and can be either the identity operator (corresponding to Dirichlet boundary conditions), first-order differential operators of the type
\begin{equation*}
\mathcal{N}u(\xi)=b(\xi)u'(\xi)+c(\xi)u(\xi)\;,\;\xi\in\{0, \ell\}
\end{equation*}
for some $b,c\in C^1[0,\ell]$ such that $b\neq 0$ on $\{0,\ell\}$ (corresponding to Neumann or Robin boundary conditions) or
\begin{equation*}
\mathcal{N}u=\big(u(\ell)-u(0), u'(\ell)-u'(0)\big)
\end{equation*}
for periodic boundary conditions.
We denote by $A$ the realization of the differential operator $\mathcal{A}$ in $\h$, endowed with the boundary condition $\mathcal{N}$. It is defined on a dense subspace $Dom(A)\subset\h$ that contains
\begin{equation*}
\{ u\in H^2(0,\ell): \mathcal{N}u=0  \}
\end{equation*}
and it generates a $C_0$ semigroup of operators $S=\{S(t)\}_{t\geq 0}\subset\mathscr{L}(\h)$. Moreover, the part of $A$ in $\overline{Dom(A)}\subset\e,$ where the closure is taken in the topology of $\e,$ generates either a $C_0$ or an analytic semigroup for which we use the same notation (see e.g. A.27 in \cite{da2014stochastic} for a definition).  Regarding the spectral properties of $A$, we make the following assumptions:
\begin{customthm}{1(a)}\label{A1a} In view of \eqref{scriptA}, the operator $-A$ is self-adjoint. As a result,  there exists a countable complete orthonormal basis $\{e_{n}\}_{n\in\N}\subset\h$ of eigenvectors of $-A$. The corresponding  sequence of nonnegative eigenvalues is denoted by $\{a_{n}\}_{n\in\N}$.	
\end{customthm}

\begin{customthm}{1(b)} \label{A1b} The eigenvectors satisfy \begin{equation*}\label{eigenbound}
	\sup_{n\in\N}\|e_{n}\|_{\mathcal{E}}<\infty.
	\end{equation*}
\end{customthm}

\begin{rem} Without loss of generality, we can replace the operator $A$ by $\tilde{A}=A-cI$ for some $c>0$ and the reaction term $f$ in \eqref{model}, by $\tilde{f}(x(\xi)):=f(x(\xi))+cx(\xi)$. The model is invariant under this  transformation and, in light of Hypothesis \ref{A1a}, it follows that
	$\|\tilde{S}(t)\|_{\mathscr{L}(\h)}\leq e^{-c t}$.
	Throughout the rest of this work we will be using $\tilde{A},\tilde{S}$ and $ \tilde{f}$ with no further distinction in notation.
	
\end{rem}
Let $\theta\geq 0$.  In view of Hypotheses \ref{A1a} along with the previous remark, $-A$, restricted to its image, has a densely defined bounded inverse  $(-A)^{-1}$ which can then be uniquely extended to all of $\h$. The fractional power $(-A)^{-\theta}$ is defined via interpolation and is also injective. Letting $(-A)^{\frac{\theta}{2}}:= ((-A)^{-\frac{\theta}{2}})^{-1}$ we define
$\h^\theta:= Dom((-A)^\frac{\theta}{2})= Range((-A)^{-\frac{\theta}{2}})\subset\h$. The norm
$\|x\|_{\h^\theta}:=\big\|(-A)^\frac{\theta}{2}x\big\|_\h$ turns  $\h^\theta$ into a Banach space and is equivalent to the graph norm (see \cite{lunardi2012analytic}, Chapter 2.2).

\begin{rem} For $\theta\in(0,\frac{1}{2})$ the spaces $H^\theta(0,\ell)$ and $\h^\theta$ coincide via the identification
	$$H^\theta(0,\ell)=\h^\theta=\big\{x\in\h : \sup_{t\in (0,1] }t^{-\theta/2} \|S(t)x-x\|_\h<\infty      \big\}$$
	which holds with equivalence of norms. The latter implies that for each $t\geq 0$, the linear operator $S(t)-I\in\mathscr{L}(H^\theta;\h)$ and there exists a constant $C>0$ such that
	\begin{equation}\label{sobcont}
	\big\|S(t)-I\big\|_{\mathscr{L}(H^\theta;\h)}\leq Ct^{\theta/2}.
	\end{equation}
\end{rem}

\noindent The analytic semigroup $S$ possesses the following regularizing properties (see e.g. section 4.1.1 in \cite{cerrai2001second})     :

(i) For $0\leq s\leq r\leq \frac{1}{2}$ and $t>0$, $S$ maps $H^{s}(0,\ell)$ to $H^{r}(0,\ell)$ and
\begin{equation}\label{Sobosmoothing}
\|S(t)x\|_{H^{r}}\leq C_{r,s}(t\wedge 1)^{-\frac{r-s}{2}}e^{c_{r,s}t}\|x\|_{H^{s}}\;\;,\;x\in H^{s}(0,\ell),
\end{equation}
for some positive constants $c_{r,s}, C_{r,s}$.

(ii) $S$ is \textit{ultracontractive}, i.e. for $t>0,$ $S(t)$ maps $\h$ to $L^{\infty}(0,\ell)$ and furthermore, for any $1\leq p\leq r\leq\infty$,
\begin{equation*}\label{Lpsmoothing}
\|S(t)x\|_{L^r(0,\ell)}\leq C(t\wedge 1)^{-\frac{r-p}{2pr}}\|x\|_{L^p(0,\ell)}\;\;,\;x\in L^p(0,\ell).
\end{equation*}


\noindent The next set of assumptions concerns the nonlinear reaction term in \eqref{model}.
\begin{customthm}{2(a)} \label{A2a}
	$f:\R\rightarrow\R$ is twice continuously differentiable and
	\begin{equation*}
	f=f_1+f_2
	\end{equation*}
	where $f_1:\R\rightarrow\R$ is globally Lipschitz continuous and $f_2:\R\rightarrow\R$ is a non-increasing function.
\end{customthm}
\begin{customthm}{2(b)} \label{A2b}
There exists $C_f>0$ and $p_0\geq3$ such that for all $ x\in\R$ and $i\in\{0,1,2 \}$
\begin{equation}
\label{fgrowth}
|\partial^{(i)}_{x}f(x)|\leq C_f\big(1+|x|^{p_0-i}\big).
\end{equation}
\end{customthm}
\noindent For $p\geq1$, $f$ induces a superposition (or Nemytskii) operator $F:\mathcal{E}\rightarrow L^p(0,\ell)$ defined by $F(x)(\xi):=f(x(\xi)),$ $\xi\in(0,\ell).$ In view of Hypotheses \ref{A2a} and \ref{A2b}, $F$ is twice G\^ateaux differentiable along any direction in $\mathcal{E}$ and (with some abuse of notation) its G\^ateaux differentials are given by $D^{i}F=\partial^{i}_xf$, $i=1,2$.

The last set of assumptions concerns the stability properties of the deterministic and linearized dynamics governed by \eqref{model}, after setting $\epsilon=0.$

\begin{customthm}{3(a)} \label{A3a}
	There exists at least one asymptotically stable equilibrium $x^*\in Dom(A)$ of \eqref{model} solving the elliptic Sturm-Liouville problem $
	 Ax+F(x)=0.$
\end{customthm}
\begin{customthm}{3(b)} \label{A3b}
	The linear self-adjoint operator $-A-DF(x^*)$ has a countable, non-decreasing sequence of nonnegative eigenvalues $\{a_n^f\}_{n\in\N}$ corresponding to a complete orthonormal set of eigenvectors $\{e^f_n\}_{n\in\N}\subset\mathcal{E}.$ Therefore, the equilibrium $x^*$ is asymptotically stable.
\end{customthm}

\begin{customthm}{3(c)} \label{A3c} The first two eigenvalues of the self-adjoint operator $-A-DF(x^*)$  satisfy $3a_1^f<a_2^f.$
\end{customthm}
This spectral gap provides a sufficient condition that allows us to identify a one-dimensional exit direction for limiting trajectories (see Lemma \ref{exitdirectionlem} below). A weaker condition under which our results continue to hold is $2a_1^f<a_2^f$ (see Remark \ref{smallergaprem}). In fact, our asymptotic results continue to hold under the following relaxed spectral gap:
\begin{customthm}{3(c')} \label{A3c'} There exists $k_0\geq 1$ such that $3a_1^f<a_{k_0+1}^f$ and $a_1^f<a_2^f.$ 
\end{customthm}
Note that Hypothesis \ref{A3c} trivially implies Hypothesis \ref{A3c'} with $k_0=1$. The latter will be used throughout Section \ref{Sec3} to prove asymptotic results. In Section \ref{Sec4} we restrict the pre-asymptotic analysis to schemes that work under Hypothesis \ref{A3c}.

 Turning to the stochastic forcing, let $(\Omega,\mathscr{F}, \mathscr{F}_{t\geq 0}, \pr)$ be a complete filtered probability space. The space-time white noise $\dot{W}$ is understood as the time-derivative of a cylindrical Wiener process $W:[0,\infty)\times\h\rightarrow L^2(\Omega)$ in the sense of distributions. The latter is a Gaussian family of random variables with covariance given by
$$\ex[W(t_1,\chi_1)W(t_2,\chi_2)         ]=t_1\wedge t_2\langle \chi_1, \chi_2\rangle_\h,$$
for $(t_i,\chi_i)\in[0,\infty)\times\h, i=1,2.$ Given a separable Hilbert space $(\h_1, \langle .\;,.\rangle_{\h_1})$ such that $\h$ is a linear subspace of $\h_1$ and the inclusion map $\h\overset{i}{\rightarrow}\h_1$ is Hilbert-Schmidt, $W$ can be identified with the $\h_1-$valued  Wiener process $$W(t)=\sum_{n=1}^{\infty} W(t,e_{n})i(e_{n})\;\;,t\geq 0$$
with covariance operator $Q=ii^*\in\mathscr{L}_1(\h)$.
This identification is assumed throughout the rest of this paper without further distinction in notation.

Having introduced the necessary notation, we can recast \eqref{model} as a stochastic evolution equation on $\mathcal{E}$ given by

 \begin{equation}\label{evoeq}
\left\{\begin{aligned}
&dX^\epsilon(t)=[AX^{\epsilon}(t)+F(X^\epsilon(t))]dt+\sqrt{\epsilon}dW(t)\\&X^{\epsilon}(0)=x.
\end{aligned}\right.
\end{equation}
A mild solution to the latter is defined as a process $X^\epsilon$ satisfying for each $\epsilon$ and all $t\in[0,T],$
 \begin{equation}\label{etamild0}
X^\epsilon(t)=S(t)x+\int_{0}^{t}S(t-s)F(X^\epsilon(s))ds+\sqrt{\epsilon}\int_{0}^{t}S(t-s)dW(t)
\end{equation}
with probability $1$. The last term is known as a stochastic convolution and will be frequently denoted by $W_A.$ Our assumptions guarantee that the $\mathcal{E}$-valued paths of $W_A$ are continuous with probability $1$ and
\begin{equation}\label{convolutionbnd}
\ex\sup_{t\in[0,T]}\big\|  W_A(t)\big\|^p_{\e}<\infty.
\end{equation}
This can be proved by the stochastic factorization method of Da Prato-Zabczyk \cite{da2014stochastic} (see also Theorem B.6 in \cite{salins2020systems}). Moreover, for each $\epsilon>0,$ \eqref{evoeq} has a unique mild solution taking values in $C([0,T];\mathcal{E})$ with probability $1$ (see e.g. Theorem 2.2 in \cite{cerrai2004large}).
\section{Moderate Deviations, Importance Sampling and Asymptotic Theory}\label{Sec3}
\subsection{General theory and main results}
In this section we present some theoretical aspects of subsolution-based importance sampling in the moderate deviation regime, applied to our problem of interest.
   First, we recall the notion of a Moderate Deviation Principle (MDP).
   \begin{dfn}\label{MDPdef} Let $T>0,$ $\mathcal{X}=\h$ or $\mathcal{E}, x\in\mathcal{X}$ and a functional $\mathcal{S}_{x,T}:C([0,T];\mathcal{X})\rightarrow[0,\infty]$ with compact sub-level sets.\\
   \noindent(i)	We say that the collection of $C([0,T];\mathcal{X})$-valued random elements $\{X^\epsilon\}_{\epsilon\ll1}$ satisfies an MDP with action functional $\mathcal{S}_{x,T}$ if, for all continuous and bounded $g: C([0,T];\mathcal{X})\rightarrow\R$ and all scalings $h(\epsilon)$ such that $h(\epsilon)\to\infty$ and $\sqrt{\epsilon}h(\epsilon)\to 0$ as $\epsilon\to 0$
   	\begin{equation}\label{MDPLP}
   	\lim_{\epsilon\to 0}\frac{1}{h^2(\epsilon)}\log\ex e^{-h^2(\epsilon)g(\eta_x^{\epsilon})}=-\inf_{\{\phi\in C([0,T];\mathcal{X}): \phi(0)=0\}} \big[ \mathcal{S}_{x,T}(\phi)+g(\phi)   \big],
   	\end{equation}
   	where $\eta_x^{\epsilon}$ is defined as in \eqref{etadef}. \\
   	\noindent (ii) A Borel set $E\subset C([0,T];\mathcal{X})$ will be called an $\mathcal{S}_{x,T}-$continuity set if $$\inf_{\phi\in \bar{E}}\mathcal{S}_{x,T}(\phi)=\inf_{\phi\in \mathring{E}}\mathcal{S}_{x,T}(\phi).$$
   \end{dfn}

As mentioned in Section \ref{Sec1} we aim to compute probabilities of the form
\begin{equation}\label{exitprob}
P(\epsilon)=\pr[ \tau_{x^*}^\epsilon\leq T]
\end{equation}
for $\epsilon\ll1, T>0,$ where $
\tau_{x^*}^\epsilon=\inf\{ t>0:  X_{x^*}^{\epsilon}(t)\notin D\}$ and
\begin{equation}
\label{exitnhood}
D=D_\epsilon=\mathring{B}_\h(x^*, L\sqrt\epsilon h(\epsilon))
\end{equation}
 for some $L>0.$ Passing to the moderate deviation process $\eta^\epsilon_x$ and recalling that $x^*$ is a (stable) equilibrium of $X^0_x$ we see that
$$\eta_{x^*}^{\epsilon}=\frac{ X_{x^*}^{\epsilon}-x^*}{\sqrt{\epsilon}h(\epsilon)}$$ and \begin{equation}
\label{taueps}
\tau_{x^*}^\epsilon=\inf\{ t>0:  \eta_{x^*}^{\epsilon}(t)\notin \mathring{B}_\h(0, L)\}.
\end{equation}
 As will be shown in Section \ref{limsec}, $\eta^\epsilon_{x^*}$ converges, as $\epsilon\to0$, to the solution of a linear deterministic PDE with zero initial condition. Since $0$ is the unique fixed point of this PDE, the limit process is bound to stay at $0$ and $\lim_{\epsilon\to 0} P(\epsilon)=0.$ This is why accelerated methods that estimate $P(\epsilon)$ when $\epsilon$ is small are useful.

In this paper, we will only work with unbiased estimators. Hence, minimizing the variance of the estimator is equivalent to minimizing the second moment. As we show below,
an upper bound for the exponential decay rate of the second moment of any unbiased estimator can be determined in terms of the action functional $\mathcal{S}_{x,T}.$
\begin{lem}\label{Gupperlem} Let $P(\epsilon)$ as in \eqref{exitprob} and $\hat{P}(\epsilon)$ be an unbiased estimator of $P(\epsilon)$ with respect to a probability measure $\bar{\mathbb{P}}$ defined on $(\Omega, \mathscr{F}).$   For any $\phi\in C([0,T];\mathcal{X}),$ let $\tau_\phi=\inf\{t>0: \phi(t)\notin \mathring{B}_{\h}(0,L) \}$  and
	\begin{equation}\label{Gdef}
	G_{T}(0,0):=\inf_{\{\phi\in C([0,T];\mathcal{X}): \phi(0)=0, \tau_\phi= T\}} \mathcal{S}_{x^*,T}(\phi).
	\end{equation}
If $\{X^\epsilon\}$ satisfies an MDP with action functional $\mathcal{S}_{x,T}$ and $E=\{  \phi\in C([0,T];\h):  \tau_{\phi}\leq T\}$ is a $\mathcal{S}_{x,T}-$continuity set then
	\begin{equation*}
	\limsup_{\epsilon\to 0} -\frac{1}{h^2(\epsilon)}\log \bar{\ex}[ (\hat{P}(\epsilon))^2]\leq 2G_T(0,0),
	\end{equation*}
where $\bar{\ex}$ denotes expectation with respect to the measure $\bar{\mathbb{P}}.$
\end{lem}
\begin{proof} We have
	\begin{equation*}
	P(\epsilon)=\pr[ \tau_{x^*}^\epsilon\leq T]=\pr[ \sup_{t\in[0,T]}\|\eta_{x^*}^{\epsilon}(t) \|_{\h}\geq L    ]= \pr[ \eta^\epsilon_{x^*}\in E ].
	\end{equation*}
	Now, for any unbiased estimator $\hat{P}(\epsilon),$
	\begin{equation*}
	\bar{\ex}[  \hat{P}(\epsilon)^2  ]\geq \bar{\ex}[  \hat{P}(\epsilon)  ]^2=P(\epsilon)^2 ,
	\end{equation*}
	where we used Jensen's inequality. Thus
	\begin{equation*}
	\begin{aligned}
	\limsup_{\epsilon\to 0}-\frac{1}{h^2(\epsilon)}\log\bar{\ex}[  \hat{P}(\epsilon)^2  ]&\leq 2\limsup_{\epsilon\to 0}-\frac{1}{h^2(\epsilon)}\log P(\epsilon)\\&
	=-2\liminf_{\epsilon\to 0} \frac{1}{h^2(\epsilon)}\log P(\epsilon)\\&
	 =2\inf_{\{\phi\in C([0,T];\mathcal{X}): \phi(0)=0,\phi\in E\}}\mathcal{S}_{x^*,T}\leq 2G_T(0,0)
		\end{aligned}
			\end{equation*}
	where we used the continuity property of $E$ in the last equality.
\end{proof}
\noindent As in finite dimensions (see e.g. the discussion in Section 2.2 in \cite{DupuisSpilZhou}) , the previous lemma shows that $2G_T(0,0)$ is the best possible exponential decay rate for any unbiased estimator. In turn, this motivates the following criterion for asymptotic optimality.
\begin{dfn}\label{optimalitydef}
	An unbiased estimator $\hat{P}(\epsilon)$ of $P(\epsilon)$ defined on a probability space $(\Omega, \mathscr{F}, \bar{\pr})$ will be called asymptotically optimal if
	\begin{equation*}
	\liminf_{\epsilon\to 0} -\frac{1}{h^2(\epsilon)}\log \bar{\ex}[ (\hat{P}(\epsilon))^2]\geq 2G_T(0,0).
	\end{equation*}
\end{dfn}
\noindent In other words, an estimator is asymptotically optimal if its second moment achieves the best possible exponential decay rate in the limit as $\epsilon\to 0$.

   Importance sampling involves changes of measure chosen to guarantee that the corresponding estimators achieve optimal (or nearly optimal) asymptotic behavior. Given a measurable feedback control (or change of measure) $u:[0,T]\times\h\rightarrow\h$ that is bounded on bounded subsets of $\h,$  we define a family of probability measures $\{\pr^\epsilon\}_{\epsilon>0}$ on $(\Omega, \mathscr{F})$ such that, for all $\epsilon$, $\pr^\epsilon<<\pr$ on $\mathscr{F}_T$ and
   \begin{equation*}
   \frac{d\pr^\epsilon}{d\pr}\bigg|_{\mathscr{F}_T}=\exp\bigg( h(\epsilon)\int_{0}^{T}\blangle  u\big(s,\eta^{\epsilon}_{x^*}(s)\big ), dW(s)\brangle_\h -\frac{h^2(\epsilon)}{2}\int_{0}^{T	} \|u\big(s,\eta^{\epsilon}_{x^*}(s)\big )\|^2_\h ds \bigg).
   \end{equation*}
   Using these new measures, it is straightforward to verify that
   \begin{equation*}
   \hat{P}(\epsilon,u):=\frac{d\pr}{d\pr^\epsilon}\mathds{1}_{\{\tau^{\epsilon}_{x^*}\leq T\}},
   \end{equation*}
   defined on $( \Omega, \mathscr{F}_T, \pr^{\epsilon}  )$, is an unbiased estimator of $\pr[\tau^{\epsilon}_{x^*}\leq T]$. Its second moment is given by \begin{equation}\label{Qdef}
   \begin{aligned}
   Q^{\epsilon}(u)&:=\ex^\epsilon\big[\hat{P}(\epsilon,u)^2\big]
 \\&
   = \ex^\epsilon\bigg[\exp\bigg( -2h(\epsilon)\int_{0}^{\tau^{\epsilon}_{x^*}}\blangle  u\big(s,\eta^{\epsilon}_{x^*}(s)\big ), dW(s)\brangle_\h +h^2(\epsilon)\int_{0}^{\tau^{\epsilon}_{x^*}	} \|u\big(s,\eta^{\epsilon}_{x^*}(s)\big )\|^2_\h ds \bigg)\mathds{1}_{\{\tau^{\epsilon}_{x^*}\leq T\}}\bigg].
   \end{aligned}
   \end{equation}
  As we show in the next lemma $Q^{\epsilon}(u)$ admits a variational stochastic control representation which will be useful for studying its asymptotic behavior. A similar variational formula can be found in (2.5) of \cite{salins2017rare}.
   \begin{lem}\label{varreplem}  Let $u:[0,T]\times\h\rightarrow\h$ be a measurable feedback control that is bounded on bounded subsets of $\h$, uniformly in $t\in[0,T]$.
   	Then for all $\epsilon>0$
   	  \begin{equation}
   	 \label{varrep}
   	 \begin{aligned}
   	 -\frac{1}{h^2(\epsilon)}\log Q^{\epsilon}(u)=\inf_{v\in\mathcal{A}}\ex\bigg[\frac{1}{2}\int_{0}^{\hat{\tau}^{\epsilon,v}_{x^*}}\|v(s)\|^2_\h ds-\int_{0}^{\hat{\tau}^{\epsilon,v}_{x^*}}\|u\big(s,\hat{\eta}^{\epsilon,v}_{x^*}(s)\big )\|^2_\h ds\bigg],
   	 \end{aligned}
   	 \end{equation}
   	where $\mathcal{A}$ is the collection of all $\h$-valued, $\mathscr{F}_{t\geq 0}$-adapted processes  $v$ defined on $[0, T]$ such that $$\hat{\tau}^{\epsilon,v}_{x^*}=\inf\{ t>0 : \hat{\eta}^{\epsilon,v}_{x^*}(t)\notin \mathring{B}_\h(0, L)     \}\leq T$$ with probability $1,$ $\hat{\eta}^{\epsilon,v}_{x^*}$ solves
   	\begin{equation}\label{controleq}
   	\left\{\begin{aligned}
   	&d\hat{\eta}^{\epsilon,v}_{x^*}(t)=A\hat{\eta}^{\epsilon,v}_{x^*}(t)+\frac{1}{\sqrt{\epsilon}h(\epsilon)}\big[ F\big(x^*+\sqrt{\epsilon}h(\epsilon)\hat{\eta}^{\epsilon,v}_{x^*}(t)  \big)-F\big(x^*  \big) \big]dt+\big[v(t)
   	-u\big(t,\hat{\eta}^{\epsilon,v}_{x^*}(t)\big )\big]dt+ \frac{1}{h(\epsilon)}dW(t)\\& \hat{\eta}^{\epsilon,v}_{x^*}(0)=0_\h
   	\end{aligned}\right.
   	\end{equation}
   	and
   	$$\ex\int_{0}^{\hat{\tau}^{\epsilon,v}_{x^*}}\|v(s)\|^2_\h ds<\infty.$$ 	
   \end{lem}
   \begin{proof} Let $\epsilon>0.$
   	  From the Cameron-Martin-Girsanov theorem,
   	\begin{equation*}
   	\begin{aligned}
   	Q^{\epsilon}(u)&=\ex^\epsilon\bigg[\exp\bigg( -2h(\epsilon)\int_{0}^{\tau^{\epsilon}_{x^*}}\blangle  u\big(s,\eta^{\epsilon}_{x^*}(s)\big ), dW^\epsilon(s)\brangle_\h -h^2(\epsilon)\int_{0}^{\tau^{\epsilon}_{x^*}	} \|u\big(s,\eta^{\epsilon}_{x^*}(s)\big )\|^2_\h ds \bigg)\mathds{1}_{\{\tau^{\epsilon}_{x^*}\leq T\}}\bigg],
   	\end{aligned}
   	\end{equation*}
   	where
   	$$W^\epsilon(t):=W(t)-h(\epsilon)\int_{0}^{t}u\big(s,\eta^{\epsilon}_{x^*}(s)\big )ds\;, t\in[0,T]$$
   	is a cylindrical Wiener process under $\pr^\epsilon$. Using yet another change of measure with
   	
   	\begin{equation*}
   	\frac{d\tilde{\pr}^\epsilon}{d\pr^\epsilon}\bigg|_{\mathscr{F}_T}=\exp\bigg(-2 h(\epsilon)\int_{0}^{T}\blangle  u\big(s,\eta^{\epsilon}_{x^*}(s)\big ), dW^{\epsilon}(s)\brangle_\h -2h^2(\epsilon)\int_{0}^{T	} \|u\big(s,\eta^{\epsilon}_{x^*}(s)\big )\|^2_\h ds \bigg),
   	\end{equation*}
   	we can write
   	\begin{equation}\label{prevar}
   	\begin{aligned}
   	Q^{\epsilon}(u)&
   	= \tilde{\ex}^\epsilon\bigg[\exp\bigg(  h^2(\epsilon)\int_{0}^{\tau^{\epsilon}_{x^*}	} \|u\big(s,\eta^{\epsilon}_{x^*}(s)\big )\|^2_\h ds \bigg)\mathds{1}_{\{\tau^{\epsilon}_{x^*}\leq T\}}\bigg]
   	=\ex\bigg[\exp\bigg(  h^2(\epsilon)\int_{0}^{\hat{\tau}^{\epsilon}_{x^*}	} \|u\big(s,\hat{\eta}^{\epsilon}_{x^*}(s)\big )\|^2_\h ds \bigg)\mathds{1}_{\{\hat{\tau}^{\epsilon}_{x^*}\leq T\}}\bigg],
   	\end{aligned}
   	\end{equation}
   	where $\hat{\eta}^{\epsilon}_{x^*}$ solves
   	\begin{equation*}\label{girsanoveta}
   	\begin{aligned}
   	\big\{d\hat{\eta}^{\epsilon}_{x^*}(t)&=A\hat{\eta}^{\epsilon}_{x^*}(t)+\frac{1}{\sqrt{\epsilon}h(\epsilon)}\big[ F\big(x^*+\sqrt{\epsilon}h(\epsilon)\hat{\eta}^{\epsilon}_{x^*}(t)  \big)-F\big(x^*  \big) \big]dt
   	-u\big(t,\hat{\eta}^{\epsilon}_{x^*}(t)\big )dt+ \frac{1}{h(\epsilon)}dW(t)\;,\;\; \hat{\eta}^{\epsilon}_{x^*}(0)=0_\h\big\}
   	\end{aligned}
   	\end{equation*}
   	and $\hat{\tau}^{\epsilon}_{x^*}$ denotes the corresponding exit time for $\hat{\eta}^{\epsilon}_{x^*}$.   This follows, once again, from the Cameron-Martin-Girsanov theorem, as
   	\begin{equation*}
   	\tilde{W}^{\epsilon}(t):=W(t)+h(\epsilon)\int_{0}^{t}u\big(s,\eta^{\epsilon}_{x^*}(s)\big )ds\;, t\in[0,T]
   	\end{equation*}
   	is a cylindrical Wiener process under the measure $\tilde{\pr}^\epsilon$.
   	From \eqref{prevar} we see that the second moment of the estimator can be written as an exponential functional of the driving noise and, as such, it admits the variational representation \eqref{varrep} (see (2.5) in \cite{salins2017rare} as well as (14) in \cite{spiliopoulos2020importance} for the finite-dimensional case).
   \end{proof}

The form of the MDP action functional provides essential information for choosing changes of measure $u$ that perform well asymptotically. In particular, if for all $\phi$ with $\mathcal{S}_{x,T}(\phi)<\infty$ there exists a (local) Lagrangian $\mathcal{L}_x$ defined on a subset of $\mathcal{X}\times\h,$ such that
                	\begin{equation}\label{localaction}
                \mathcal{S}_{x,T}(\phi)=\int_{0}^{T}\mathcal{L}_x(\phi(t),\dot{\phi}(t)) dt\;,
                \end{equation}
  then "good" changes of measure are connected to subsolutions of the PDE
 \begin{equation}\label{HJB}
 \left\{
 \begin{aligned}
  &\partial_tU(t,\eta)+\mathbb{H}_x\big(\eta,D_\eta U(t,\eta)\big)=0\;,\;(t,\eta)\in[0,T)\times \mathcal{K}\\&
     U(T,\eta)=\bar{g}(\eta)\;,\;\eta\in\mathcal{K}\subset\h,
 \end{aligned}\right.
 \end{equation}
 with
 	\begin{equation*}
 \bar{g}(\eta)=\begin{cases}
 &0, \;\;\eta: \|\eta\|_\mathcal{\h}\geq L\\&
 \infty, \;\; \eta: \|\eta\|_\mathcal{\h}< L.
 \end{cases}
 \end{equation*}
 Here, $\mathbb{H}_x$ denotes the Hamiltonian corresponding to $\mathcal{L}_x$ via Legendre transform (up to a sign). In the problems we consider, the latter are not well-defined on the whole space but rather on a subset $\mathcal{K}\times\h\subset\h\times\h,$ see e.g.  \eqref{Hamiltonian2} below. The notion of subsolution is meant in the sense of the following definition.
 \begin{dfn}\label{subsol} A subsolution of \eqref{HJB} is any $U:[0,T]\times\mathcal{K}\rightarrow\R$ such that for all $(t,\eta),$ $U(\cdot,\eta)\in C^1(0,T)$, $U(t,\cdot)\in C^1(\mathcal{K})$ in the sense of Fr\'echet differentiation and satisfies
 	\begin{equation*}\label{HJBsub}
 	\left\{
 	\begin{aligned}
 	&\partial_tU(t,\eta)+\mathbb{H}_x\big(\eta,D_\eta U(t,\eta)\big)\geq 0\;,\;(t,\eta)\in[0,T)\times \mathcal{K}\\&
 	U(T,\eta)\leq \bar{g}(\eta)\;,\;\eta\in\mathcal{K}\subset\h.
 	\end{aligned}\right.
 	\end{equation*}
 \end{dfn}
  The interested reader is referred to  \cite{dupuis2007subsolutions} for the original development of subsolution-based importance sampling.
    As we will show below (Theorem \ref{MDPthm} and Remark \ref{MDPrem}), when $x=x^*$, the MDP action functional takes the form \eqref{localaction} with
\begin{equation}\label{lagrangian2}
\mathcal{L}_{x^*}(\eta,v)=\frac{1}{2}\|v-[A+DF(x^*)]\eta\|_\h^2,\;\; (\eta,v)\in Dom(A)\cap\mathcal{E}\times\h
\end{equation}
and the corresponding Hamiltonian is given by
\begin{equation}\label{Hamiltonian2}
\mathbb{H}_{x^*}(\eta,p)=\blangle [A+DF(x^*)]\eta, p\brangle_\h-\frac{1}{2}\|p\|_\h^2\;, (\eta,p)\in Dom(A)\cap\mathcal{E}\times\h.
\end{equation}
A direct consequence of \eqref{localaction} is that we can construct an explicit stationary subsolution in terms of the corresponding \textit{quasipotential}. The latter is given by
\begin{equation*}\label{quasipot}
\begin{aligned}
V_{x^*}(\eta)&=\inf\{ \mathcal{S}_{x^*,T}(\phi) :\phi\in C([0,T];\mathcal{X}): \phi(0)=0, \phi(T)=\eta, T\in(0,\infty)\}\\&
=\|(-A)^\frac{1}{2}\eta\|_\h^2-\blangle DF(x^*)\eta, \eta \brangle\\&
=-\blangle [A+DF(x^*)]\eta, \eta \brangle\;,\; \eta\in Dom(A)
\end{aligned}
\end{equation*}
and $V_{x^*}(\eta)=\infty$ otherwise. A physical interpretation of $V_{x^*}(\eta)$ is that of the minimal "energy" required to push a path from $0$ to the state $\eta$ and its explicit form is a consequence of the fact that \eqref{evoeq} is, in our setting, a gradient system (see e.g. \cite{da1991minimum}, \cite{da2014stochastic} Section 12.2.3 for SRDEs). In view of Hypotheses \ref{A1a}, \ref{A1b}, \ref{A3a}, \ref{A3b} it follows that
\begin{equation}\label{subsoleq}
U(t,\eta)=a_1^fL^2-V_{x^*}(\eta)
\end{equation}
is a subsolution of \eqref{HJB} on $\mathcal{K}=Dom(A)$. The final condition is satisfied since $a_1^f=\inf_{n\in\N}a_n^f.$
\begin{rem}\label{infdimrem}
	In finite-dimensional systems, feedback controls (or changes of measure) defined by $u(t,\eta)=-D_\eta U(t,\eta)$ lead to nearly optimal asymptotic behavior (see \cite{dupuis2012importance} Section 2.3, \cite{DupuisSpilZhou} Theorem 2.4 for large-deviation and \cite{spiliopoulos2020importance} Theorem 3.1 for moderate deviation-based schemes). A first issue that appears in infinite dimensions is that $u(t, \hat{\eta}^{\epsilon,v}_{x^*}(t))$ is not well-defined since with probability $1$ and for all $t,$ $\hat{\eta}^{\epsilon,v}_{x^*}(t)\notin Dom(A)$. The latter is a consequence of the spatial irregularity of the noise. 
\end{rem}

\noindent Throughout the rest of this paper, $P^f_n:\h\rightarrow\h$ denotes  an orthogonal projection to the $n-$dimensional eigenspace $\text{span}\{ e^f_j\}_{j=1}^{n}$ and we consider the "projected" quasipotential $V_{x^*}(P^f_n\eta)=V_{x^*}(\langle \eta, e_1^f\rangle_\h e_1^f),$ the subsolution $U(t,P^f_n\eta)$ of \eqref{HJB} (with $\mathcal{K}=P_n^f\h$). The changes of measure we will use are given by
\begin{equation}\label{uchoice}
u_{k_0}(t,\eta):=-D_\eta U(t,P^f_{k_0}\eta):=2\sum_{i=1}^{k_0}a_i^f\langle\eta, e_i^f\rangle_\h e_i^f,
\end{equation}
with $k_0$ as in Hypothesis \ref{A3c'}.
For implementation purposes, $u_{k_0}$ is replaced by a sequence $u_{k_0}^\epsilon$ that converges to $u_{k_0}$ as $\epsilon\to 0$. For more details on the choice of $u_{1}^\epsilon$ see \eqref{uepsilonchoice} and the discussion in Section \ref{Sec4} below.
\noindent We can now present our main results on the asymptotic behavior of the scheme.
\begin{thm}(Moderate Deviations)\label{MDPthm} Let $T>0, L>0$ as in \eqref{exitnhood}, $k_0$ as in Hypothesis \ref{A3c'}, $u_{k_0}$ as in \eqref{uchoice}, $Q^\epsilon$ as in \eqref{Qdef} and $B_\h(0,L)\subset\h$ denote the closed ball of radius $L$ centered at the origin. Moreover let $u_{k_0}^\epsilon:[0,T]\times\h\rightarrow\h$ be a sequence that converges pointwise and uniformly over bounded subsets of $\h$ to $u_{k_0}$, \begin{equation}\label{scriptT}
	\mathcal{T}=\big\{y\in C([0,T];\h): y(0)=0, \exists \tau\in(0,T] :y(\tau)\in\partial B_\h(0,L),\; y(t)\in B_\h(0,L)\; \forall t\in[0,\tau)\big\}
	\end{equation}
and
	\begin{equation*}
\mathcal{C}_{y,x^*}=\big\{v\in L^2([0,T];\h): \dot{y}(t) = Ay(t) +DF(x^*)y(t)-u_{k_0}(t,y(t))+v(t)\big\}.
\end{equation*}
Under Hypotheses \ref{A1a}-(c), \ref{A2a},(b), \ref{A3a},(b),(c') we have
	\begin{equation}\label{MDPeq1}
	\lim_{\epsilon\to0} -\frac{1}{h^2(\epsilon)}\log Q^{\epsilon}(u_{k_0}^\epsilon)=  \inf_{y\in\mathcal{T}}\inf_{v\in \mathcal{C}_{y,x^*}}\int_{0}^{\tau}\bigg(\frac{1}{2}\|v(t)\|^2_\h-\|u_{k_0}(t,y(t))\|_\h^2\bigg) dt,
	\end{equation}
with the convention that the infimum over the empty set is $\infty.$
\end{thm}

\begin{rem}\label{refrem1} A few comments on \eqref{MDPeq1} are in order: \textbf{1)} If $y\in H^1((0,T);\h)\cap L^2([0,T];Dom(A)),$ the set $\mathcal{C}_{y,x^*}$ reduces to the singleton $\{\bar{v}(t):=  \dot{y}(t)-Ay(t)-DF(x^*)y(t)-u(t,y(t))\}$  and for any $y\notin H^1((0,T);\h)\cap L^2([0,T];Dom(A)),$ $\mathcal{C}_{y,x^*}$ is empty. \textbf{2)} Using the same notation, it follows that the right-hand side of \eqref{MDPeq1} can be expressed as
	$$\inf_{y\in\mathcal{T}}\int_{0}^{\tau}\bigg(\frac{1}{2}\|\bar{v}(t)\|^2_\h-\|u_{k_0}(t,y(t))\|_\h^2\bigg) dt.$$
	\textbf{3)} Since the functional on the right-hand side involves only the values of $y$ on $[0,\tau]$ it is straightforward to see that the infimum can in fact be taken over paths $y\in C([0,\tau];\h)$ that satisfy the constraints in \eqref{scriptT}.
\end{rem}
\noindent Using the moderate deviation asymptotics of Theorem \ref{MDPthm} we can then prove the following:
\begin{thm}(Near asymptotic optimality)\label{Asymptoticthm} Let $L,T>0$, $k_0,u_{k_0},u_{k_0}^\epsilon:[0,T]\times\h\rightarrow\h$ as in Theorem \ref{MDPthm}, $\mathcal{A}$ as in Lemma \ref{varreplem}, $U$ as in \eqref{subsoleq} and $G_T$ as in \eqref{Gdef}. For any sequence  $\{v^\epsilon\}\subset\mathcal{A}$ such that
		\begin{equation}\label{approxmin}
 -\frac{1}{h^2(\epsilon)}\log Q^{\epsilon}(u_{k_0}^\epsilon)\geq \ex\bigg[\frac{1}{2}\int_{0}^{\hat{\tau}^{\epsilon,v^\epsilon}_{x^*}}\|v^\epsilon(s)\|^2_\h ds-\int_{0}^{\hat{\tau}^{\epsilon,v^\epsilon}_{x^*}}\|u_{k_0}^\epsilon\big(s,\hat{\eta}^{\epsilon,v^\epsilon}_{x^*}(s)\big )\|^2_\h ds\bigg]-\epsilon^2
	\end{equation}
	we have
\begin{equation}\label{asymptoticexitplace}
\lim_{\epsilon\to 0}\ex\blangle   \hat{\eta}^{\epsilon,v^\epsilon}_{x^*}\big(\hat{\tau}^{\epsilon,v^\epsilon}_{x^*}\big), e_1^f    \brangle^2_\h= L^2.
\end{equation}	
Moreover, we have the second moment bounds
\begin{equation}\label{varbounds}
G_T(0,0)+U(0,0) \leq \lim_{\epsilon\to 0}-\frac{1}{h^2(\epsilon)}\log Q^{\epsilon}(u_{k_0}^\epsilon)\leq 2G_T(0,0),
\end{equation}
where $U(0,0)\leq G_T(0,0)$ and $G_T(0,0)\longrightarrow U(0,0)$ as $T\to\infty.$
\end{thm}
\noindent The first statement above asserts that the limiting controlled trajectories exit the domain $D$ through the boundary near the direction of the eigenvector $e_1^f$ (see Hypotheses \ref{A3c}, (c')). Finally, \eqref{varbounds} shows that, for any finite time horizon $T$, our scheme is close to asymptotic optimality, according to Definition \ref{optimalitydef}, and achieves optimal behavior in the limit $\epsilon\to 0, T\to\infty$. Near asymptotic optimality is a common feature of importance sampling schemes for  continuous-time dynamics even in finite dimensions. This is mainly a consequence of using subsolutions of \eqref{HJB} instead of exact solutions which are rarely given in explicit form. Our numerical studies indicate that near optimality leads to provably superior performance in comparison to standard Monte Carlo.

\begin{rem}\label{Zabcszykrem}The moderate deviation regime allows us to work with the exit problem of a linear equation instead of that of the initial nonlinear SRDE \eqref{model}. The "drift" of this linear equation is given by $A+DF(x^*)$ and thus the dominant eigenpairs of this operator govern the exit time and exit place asymptotics. As mentioned in the introduction, similar statements have been proved for finite-dimensional linear equations in \cite{zabczyk1985exit} (see e.g. Theorem 6).	
\end{rem}

\subsection{On the asymptotic exit direction}\label{epsilon0sec} In this section we study the limiting variational problem appearing on the right-hand side of \eqref{MDPeq1}. In particular, we will show that, under Hypothesis \ref{A3c'}, changes of measure that force the dynamics in the $e_1^f$ direction lead to minimal paths that exit from the ball $\mathring{B}_\h(0,L)$ through the same direction. From this point on we will only use the notation
$\mathcal{S}_{x,T}$ to denote the explicit action functional
\begin{equation}\label{MDPaction}
\mathcal{S}_{x,T}(\phi)=\frac{1}{2}\int_{0}^{T}\big\|\dot{\phi}(t)-[A+DF\big(X^0_x(t)\big)]\phi(t)\big\|^2_{\h}dt.
\end{equation}
Moving on to the variational problem in \eqref{MDPeq1}, we let
$I^{k_0}:\mathcal{T}\subset C([0,T];\h)\rightarrow \R,$
\begin{equation}\label{functional}
I^{k_0}(y):=\inf_{v\in \mathcal{C}_{y,x^*}}\int_{0}^{\tau}\bigg(\frac{1}{2}\|v(t)\|^2_\h-\|u_{k_0}(t,y(t))\|_\h^2\bigg) dt
\end{equation}
and seek to characterize $\arg\min_{y\in\mathcal{T}}I^{k_0}(y).$ For the first part of this section we consider the case $k_0=1$ covered by Hypothesis \ref{A3c}. The more general setting of Hypothesis \ref{A3c'} will be studied in Proposition \ref{kprop} below. For the sake of simplicity we will drop the superscript $k_0$ and write $I\equiv I^{1}$ and $u\equiv u_{1}$ unless otherwise stated.

A first observation is that
$I(y)<\infty$ if and only if $y\in H^1((0,T);\h)\cap L^2([0,T];Dom(A))$ and for all such $y$ the infimum above is uniquely attained by
\begin{equation*}
\bar{v}(t)=  \dot{y}(t)-Ay(t)-DF(x^*)y(t)-u(t,y(t)),\; t\in[0,T]
\end{equation*}
(see also Remark \ref{refrem1} above). Therefore, in view of \eqref{uchoice}, we can re-express $I$ as follows:
\begin{equation}\label{Icomp}
\begin{aligned}
I(y)&=\int_{0}^{\tau}\frac{1}{2}\|\bar{v}(t)\|^2_\h-\|u(t,y(t))\|_\h^2 dt\\&=\int_{0}^{\tau}\bigg(\frac{1}{2}\|\dot{y}(t)-Ay(t) -DF(x^*)y(t)+u(t,y(t))\|^2_\h-\|u(t,y(t))\|_\h^2\bigg)\ dt\\&
=\int_{0}^{\tau}\bigg(\frac{1}{2}\|\dot{y}(t)-Ay(t) -DF(x^*)y(t)\|^2_\h-\frac{1}{2}\|u(t,y(t))\|_\h^2 \bigg)dt\\&
+\int_{0}^{\tau}\blangle \dot{y}(t), u(y(t))\brangle_\h dt- \int_{0}^{\tau}\blangle   Ay(t) +DF(x^*)y(t), u(y(t))\brangle_\h dt\\&
=\mathcal{S}_{x^*,\tau}(y)-\int_{0}^{\tau}\frac{1}{2}\|u(t,y(t))\|_\h^2 dt+2a_1^f\int_{0}^{\tau}\langle y(t), e_1^f\rangle_\h\langle \dot{y}(t), e_1^f\rangle_\h dt\\&
- 2a_1^f\int_{0}^{\tau}\blangle Ay(t) +DF(x^*)y(t), e_1^f\brangle_\h \langle y(t), e_1^f\rangle_\h dt\\&
=\mathcal{S}_{x^*,\tau}(y)+2a_1^f\int_{0}^{\tau}\langle y(t), e_1^f\rangle_\h\langle \dot{y}(t), e_1^f\rangle_\h dt
+2(a_1^{f})^2\int_{0}^{\tau} \langle y(t), e_1^f\rangle^2_\h dt-\int_{0}^{\tau}\frac{1}{2}\|u(t,y(t))\|_\h^2dt.
\end{aligned}
\end{equation}
The last two terms in the last display are equal due to \eqref{uchoice}. Thus,
\begin{equation}\label{Ifinalform}
\begin{aligned}
&I(y)=\mathcal{S}_{x^*,\tau}(y)+a_1^f\int_{0}^{\tau}\frac{d}{dt}\bigg(\langle y(t), e_1^f\rangle^2_\h\bigg)dt
=\mathcal{S}_{x^*,\tau}(y)+a_1^f\big(\langle y(\tau), e_1^f\rangle^2_\h-\langle y(0), e_1^f\rangle^2_\h\big).
\end{aligned}
\end{equation}
It is straightforward to verify that $\arg\min_{y\in\mathcal{T}}I(y)\neq\varnothing,$ i.e. the minimum value of $I$ over the set $\mathcal{T}\subset C([0,T];\h)$ is attained in $\mathcal{T}$. Indeed,  $\mathcal{S}_{x^*,\cdot}:[0,T]\times C([0,T];\h)\rightarrow[0, \infty]$ is lower-semicontinuous and the second summand in \eqref{Ifinalform} defines a continuous functional on the same set. Thus, $I$ is itself lower-semicontinuous and furthermore $\mathcal{T}$ is closed in the topology of $C([0,T];\h)$ (recall that $B_\h(0,L)$ in \eqref{scriptT} is a closed ball in $\h)$.

\begin{rem}\label{minrem}
	We shall proceed to the characterization of minimizers in three steps. First we minimize over paths $y$ with $y(0)=0$ and $y(\tau)=z\in\partial B_{\h}(0,L)$. Then we minimize over the exit place $z$ and finally over the time $\tau$ in which the path $y$ hits the boundary $\partial B_{\h}(0,L)$ of the closed ball $B_{\h}(0,L)$. At this point, we emphasize that, in contrast to $\tau^\epsilon_{x^*}$ \eqref{taueps}, $\tau_\phi$ (Lemma \ref{Gupperlem}),$\hat{\tau}^{\epsilon, v}_{x^*}$ (Lemma \ref{varreplem}) and $\hat{\tau}^{\epsilon, v^\epsilon}_{x^*}$ \eqref{approxmin}, it is not known a priori whether the time $\tau$ is the first exit time of $y$ from the open ball $\mathring{B}_\h(0,L)$. We will show that the latter is true for minimizing paths in Lemma \ref{exitdirectionlem} and Proposition \ref{kprop} below.
\end{rem}

\begin{lem}\label{Eulalem} Let $y^*\in\arg\min\{ I(y): y\in C([0,\tau];\h), y(0)=0, y(\tau)=z\}.$ Then
	\begin{equation*}
	y^*(t)= y_{z,\tau}^*(t) =  \sum_{k=1}^{\infty}\frac{\sinh(a_k^{f}t) }{\sinh(a_k^{f}\tau)}\langle z,e_k^f\rangle_\h e_k^f, \;\;t\in[0,\tau].
	\end{equation*}
	
\end{lem}
\begin{proof}
	\noindent The fact that we minimize over $y\in C([0,\tau];\h)$ instead of $C([0,T];\h)$ is justified by Remark \ref{refrem1}-3). Next notice that $y^*_{z,\tau}\in C([0,\tau];\h)$ since $\sinh$ is increasing and continuous. In particular,
	$$\|y^*_{z,\tau}\|_{\h}\leq \sum_{k=1}^{\infty}\langle z,e_k^f\rangle^2_\h=\|z\|^2_{\h}.$$
	 Proceeding to the proof we have, in view of \eqref{Ifinalform},
	\begin{equation*}
	I(y)=\int_{0}^{\tau}\mathcal{L}_{x^*}(y(t),\dot{y}(t))dt+a_1^f\langle z, e_1^f\rangle^2_\h
	\end{equation*}
	with $\mathcal{L}_{x^*}$ as in \eqref{lagrangian2}. Minimizers are then governed by the Euler-Lagrange equation
	\begin{equation}\label{EL}
	\partial_tD_v\mathcal{L}_{x^*}(y(t),\dot{y}(t))=D_\eta\mathcal{L}_{x^*}(y(t),\dot{y}(t))
	\end{equation}
 which boils down to
	\begin{equation*}
	\big\{y''(t)=[A+DF(x^*)]^2y(t)\;, y(0)=0\;,\;y(\tau)=z\big\}.
	\end{equation*}
Projecting to the eigenbasis $\{e_k^f\}_{k\in\N}$ of $A+DF(x^*)$ we obtain
\begin{equation*}\label{Eula3}
\frac{d^2}{dt^2}\langle  y(t), e_k\rangle =(a_k^{f})^2\langle y(t),  e_k \rangle\; ,k\in\N_{0}, t\in[0,T].
\end{equation*}
Letting $y_k=\langle y, e_k^f\rangle_\h,z_k=\langle z, e_k^f\rangle_\h,$ the general solution of the latter has the form
$
y_k(t)=c_1e^{a_k^ft}+c_2e^{-a_k^{f}t}
$
and taking into account the initial and terminal conditions we obtain
\begin{equation*}
c_1+c_2=0, z_k=c_1e^{a_k^{f}\tau}+c_2e^{-a_k^{f}\tau}\implies c_1=\frac{z_k}{e^{a_k\tau}-e^{-a_k^{f}\tau}}.
\end{equation*}
Thus,
\begin{equation*}
y_k(t)=\frac{z_k(e^{a_k^{f}t}-e^{-a_k^{f}t})} {e^{a_k^f\tau}-e^{-a_k^{f}T}}=z_k\frac{\sinh(a_k^{f}t) }{\sinh(a_k^{f}\tau)}.
\end{equation*}
\end{proof}
\noindent  The next lemma is concerned with the exit direction when Hypothesis \ref{A3c} holds.
\begin{lem}\label{exitdirectionlem} Let $T>0,$ $I$ as in \eqref{functional} and $u,\mathcal{T}, C_{y,x^*}$ as in Theorem \ref{MDPthm}. Under Hypothesis \ref{A3c}, any  $y^*\in\arg\min\{I(y); y\in \mathcal{T} \} $ $y^*$ first exits $\mathring{B}_\h(0,L)$ at $\tau=T$ in the direction of the eigenvector $e_1^f$ (recall Remark \ref{minrem}) i.e. for all $k\geq 2$, $$\langle y^*(\tau), e_k^f\rangle_\h=\langle y^*(T), e_k^f\rangle_\h=0,$$
$\|y^*(t)\|_{\h}<L$ for all $t<T$  and $\|y^*(T)\|_{\h}=L.$ 
	
\end{lem}
\begin{proof} Let $\phi^*=\phi^*_{z,\tau}$ be a minimizer provided by Lemma \ref{Eulalem}. Notice that, since the Euler-Lagrange equations provide necessary conditions for minimality, any $\phi^*\in\arg\min\{ I(y): y\in C([0,\tau];\h), y(0)=0, y(\tau)=z\}$ will be of this form. After straightforward algebra we obtain
	\begin{equation*}
	\begin{aligned}
	I(\phi^*_{z,\tau})&= \mathcal{S}_{x^*,\tau}(\phi^*)+a_1^f\langle z, e_1^f\rangle^2_\h\\&
	=  \sum_{k=1}^{\infty}\int_{0}^{\tau}\langle \dot{\phi}^*(t)-[A+DF(x^*)]\phi^*(t),e^f_k\rangle_\h^2 dt+a_1^f\langle z, e^f_1\rangle^2_\h\\&
	=\sum_{k=1}^{\infty}\int_{0}^{\tau}\big(\dot{\phi}_k^*(t)-a_k^{f}\phi_k^*(t)\big)^2dt+a_1^f\langle z, e^f_1\rangle^2_\h\\&
	=a_1^f z^2_1 +\sum_{k=1}^{\infty}\frac{a^{f}_kz_k^2}{1-e^{-2a^{f}_k\tau}}.
	\end{aligned}
	\end{equation*}
Now for each fixed $\tau$, Hypothesis \ref{A3c} guarantees that this quadratic form is minimized for $z^*\in\partial B_{\h}(0,L)$ such that $z^*_k=0$ for all $k\geq 2$ and $z_1^*=\pm L$ (see e.g. Theorem 3.4 in \cite{salins2017rare}). Then,
	\begin{equation}\label{minvalue}
\begin{aligned}
I(\phi^*_{z^*,\tau})=a_1^fL^2\bigg(1+\frac{1}{1-e^{-2a^{f}_1\tau}}\bigg)
	\end{aligned}
\end{equation}
is minimized for the largest possible $\tau$ i.e. for $\tau=T.$ Hence, since the order with which the variables are being minimized does not change  the value of the minimum, we have
$\min_{y\in\mathcal{T}}I(y)= I(\phi^*_{z^*,T})$ and the minimizers $y^*=\phi^*_{z^*,T}$ enjoy the desired properties.  Finally,  note that any element $y^*\in\arg\min\{I(y); y\in \mathcal{T} \} $ is of the form $\phi^*_{z^*, T}.$ Indeed, fix the initial and terminal values $y^*(0)=0, y^*(\tau)=z\in\partial B_\h(0, L)$ and assume that $y^*$ does not satisfy the Euler-Lagrange equations \eqref{EL}. Since the latter provide necessary conditions for minimality, it follows that $y^*$ is not a minimizer. Moreover,  it follows from the previous calculations that if $\tau<T$ or if $z_k \neq 0$ for some $k\geq 2$ then $y^*$ cannot be a minimizer of $I.$		The proof is complete.\end{proof}

\noindent As mentioned above, the previous lemma implies that, for any minimizing path $y^*$, $\tau$ is in fact the first exit time from the open ball $\mathring{B}_\h(0,L),$ i.e. $\tau=\inf\{t\in[0,T]: y^*\notin \mathring{B}_\h(0,L) \}$  and furthermore $\tau=T$.

\begin{rem}\label{smallergaprem}
If the sampling time $T$ is large enough, the results of Lemma \ref{exitdirectionlem} as well as Theorems \ref{MDPthm}, \ref{Asymptoticthm} remain true under the weaker spectral gap assumption that $2a_1^f<a_k^f$ for all $k\geq 2.$  Since we are interested in schemes that perform well for large values of $T,$ this generalization comes at no cost. For more details on this relaxed condition see \cite{salins2017rare}, Theorem 3.9.
\end{rem}

Up to this point we have worked under Hypothesis \ref{A3c} to show that minimizers of the functional $I$ lie on the one-dimensional subspace where the change of measure $u$ acts. In the absence of a sufficiently large spectral gap the situation is more complicated. In particular, if the sampling time $T$ is large enough, the minimizers can be orthogonal to $u$. In other words, forcing the system towards its physical exit direction $e_1^f$ might actually lead to controlled trajectories that exit from a subspace that is orthogonal to $e_1^f$ under the change of measure. This is proved in the following lemma.

\begin{lem}\label{a2lem} Assume that the eigenvalues $\{a^f_k\}_{k\in\N}$ are strictly increasing, $a_2^f\leq 2a_1^f$ and let
	 $$T^*:=-\frac{1}{2a^f_2}\ln\bigg(1-\frac {a^f_2}{2a_1^f}\bigg).$$
	 If $T> T^*$ then any minimizer $y^*\in\arg\min\{I(y); y\in \mathcal{T} \} $ satisfies $\|y^*(t)\|_{\h}<L$ for all $t<T$  and $\|y^*(T)\|_{\h}=L.$ Moreover $y^*$ first exits $\mathring{B}_\h(0,L)$ at $\tau=T$ in the direction of the eigenvector $e_2^f$ (recall Remark \ref{minrem}) i.e. for all $k\neq 2,$ $$\langle y^*(\tau), e_k^f\rangle_\h=\langle y^*(T), e_k^f\rangle_\h=0.$$
\end{lem}
\begin{proof} As in the proof of Lemma \ref{exitdirectionlem} we have
	\begin{equation*}
	\begin{aligned}
	I(\phi^*_{z,\tau})&=a_1^f z^2_1 +\sum_{k=1}^{\infty}\frac{a^{f}_kz_k^2}{1-e^{-2a^{f}_k\tau}}\\&
	=: \sum_{k=1}^{\infty}\lambda_k^fz_k^2.
	\end{aligned}
	\end{equation*}
	We claim that, without loss of generality, we can consider $\tau\in(T^*,T].$ Assuming the latter for now, we can compare the weights $\lambda_k^f$ to conclude that
	\begin{equation}\label{lambda2}
	\begin{aligned}
	\lambda_2^f=\frac{a^{f}_2}{1-e^{-2a^{f}_2\tau}}< \frac{a^{f}_2}{1-e^{-2a^{f}_2T^*}}=\frac{a_2^f}{1-e^{\ln(1-a_2^f/2a_1^f)}}=2a_1^f\leq a^{f}_1\bigg(1+\frac{1}{1-e^{-2a^{f}_1\tau}}\bigg)=\lambda_1^f
	\end{aligned}
	\end{equation}
	and since $x\mapsto x/(1-e^{-2\tau x})$ is (strictly) increasing for all $\tau,$ it follows that
     \begin{equation*}
     \lambda_2^f< \lambda_k^f\;,\;\;\forall k> 2.
     \end{equation*}
     Therefore, the quadratic form is minimized for $z^*\in\partial B_\h(0,L)$ such that $z^*_k=0$ for all $k\neq 2$ and $z^*_2=\pm L$. Consequently
         \begin{equation}\label{ubound}
         \begin{aligned}
        \inf_{(z,\tau)\in\partial B_\h(0,L)\times[ T^*,T ] } I(\phi^*_{z,\tau})=\frac{a^{f}_2L^2}{1-e^{-2a^{f}_2T}} \geq   \inf_{(z,\tau)\in\partial B_\h(0,L)\times[0,T]}I(\phi^*_{z,\tau})
         \end{aligned}
         \end{equation}
         and
          \begin{equation*}
         \begin{aligned}
         \inf_{(z,\tau)\in\partial B_\h(0,L)\times[0,T]}I(\phi^*_{z,\tau})\geq  \inf_{(z,\tau)\in\partial B_\h(0,L)\times[ 0,T ] } \bigg(\inf_{k\in\N}\lambda_{k}^{f}(\tau)\bigg)\|z\|^2_{\h}= L^2\inf_{\tau\in[ 0,T ] } \bigg(\inf_{k\in\N}\lambda_{k}^{f}(\tau)\bigg).
         \end{aligned}
         \end{equation*}
         Since $\lambda_2^f\leq \lambda_k^f$ for all $k\geq 2$ it follows that
          \begin{equation}\label{lbound}
         \begin{aligned}
         \inf_{(z,\tau)\in\partial B_\h(0,L)\times[0,T]}I(\phi^*_{z,\tau})\geq L^2  \bigg(\lambda_{1}^{f}(T)\wedge \lambda_{2}^{f}(T)\bigg)=L^2\lambda_{2}^{f}(T)=\frac{a^{f}_2L^2}{1-e^{-2a^{f}_2T}},
         \end{aligned}
         \end{equation}
  which follows from \eqref{lambda2} by setting $\tau=T>T^*.$ Since the infimum is achieved  at $t=T$, the combination of \eqref{ubound} and \eqref{lbound} concludes the proof.
 \end{proof}

\begin{rem} Lemma \ref{a2lem} highlights the importance of sufficient spectral gaps for the design of efficient changes of measure. If Hypothesis \ref{A3c} fails, a scheme that forces the $e_1^f$ direction will be far from optimal and is expected to produce large errors for small values of $\epsilon.$ Under the assumptions of that lemma, one can repeat the arguments of the proof above to show
	\begin{equation*}
	\lim_{\epsilon\to 0}-\frac{1}{h^2(\epsilon)}\log Q^{\epsilon}(u_{1}^\epsilon)=\frac{a_2^fL^2}{1-e^{-2a^{f}_2T}} <2G_T(0,0).
	\end{equation*}
	If the ratio $2a_1^f/a_2^f$ is large, this bound translates to sub-optimal  performance as $\epsilon\to 0$ which does not improve as $T\to\infty.$ Moreover, as we will see in Section \ref{Sec5}, this ratio depends non-trivially on the interval length $\ell$ and is indeed large when $\ell$ is moderately small. This behavior is caused by the linearization of the dynamics and is completely absent when $f=0.$ For an example that satisfies the assumptions of Lemma \ref{a2lem} see Section \ref{Neumannsec}.
\end{rem}

\noindent Before we conclude this section we consider once again the situation where the eigenvalues $\{a_k^f\}_{k\in\N}$ do not satisfy Hypothesis \ref{A3c} but instead Hypothesis \ref{A3c'} holds. We show that the conclusions of Lemma \ref{exitdirectionlem} can be recovered by projecting to a higher dimensional eigenspace of $A+DF(x^*)$ consisting of the first $k_0$ eigenvalues.
\begin{prop}\label{kprop} Let $k_0$ as in Hypothesis \ref{A3c'},  $U$ as in \eqref{subsoleq} and $u_{k_0}$ as in \eqref{uchoice}.  Under Hypothesis \ref{A3c'} any minimizer $y^*\in\arg\min\{I^{k_0}(y); y\in\mathcal{T} \}$ satisfies the same properties as in Lemma \ref{exitdirectionlem}.
\end{prop}
\begin{proof} Following the computations in \eqref{Icomp}, which carry over verbatim, we see that
	\begin{equation}\label{Ikform}
	I^{k_0}(y)=\mathcal{S}_{x^*,\tau}(y)+\sum_{j=1}^{k_0}\bigg[a_j^f\bigg(\langle y(\tau), e_j^f\rangle^2_\h-\langle y(0), e_j^f\rangle^2_\h\bigg)\bigg].
	\end{equation}
Since the second term is constant for each fixed value of the exit point $y(\tau),$ the Euler-Lagrange equations and minimizers for this functional are then identical to those derived in Lemma \ref{Eulalem} for $I.$ Thus, for any minimizing path $\phi^*_{z,\tau}$ that hits the point $z=(z_k)_{k\in\N}\in\partial B_{\h}(0,L)$ at time $\tau\in[0,T]$ we have
\begin{equation*}
\begin{aligned}
I^{k_0}(\phi^*_{z,\tau})&=\sum_{j=1}^{k_0}a^f_j z^2_j+\sum_{j=1}^{\infty}\frac{a_j^fz_j^2}{1-e^{-2a_j^f\tau}}
=:\sum_{j=1}^{\infty}\lambda^f_{k_0,j}z_j^2.
\end{aligned}
\end{equation*}
Comparing the weights $\lambda^f_{k_0,j}$ we see that for all $1<j\leq k_0$
$$\lambda^f_{k_0,1}=a_1^f\bigg(1+\frac{1}{1-e^{-2a_1^f\tau}}\bigg)< a_j^f\bigg(1+\frac{1}{1-e^{-2a_j^f\tau}}\bigg)=\lambda_{k_0,j},$$
which holds since $x\mapsto x/(1-e^{-2\tau x})$ is (strictly) increasing for all $\tau$ and $a_1^f<a_2^f\leq a_j^f$ for any $j\geq 2$.
In order to show that minimizers point towards $z_1$ it remains to compare  $\lambda^f_{k_0,1}$ with $\lambda^f_{k_0,j}$ for $j\geq k_0+1.$ Since $\lambda_{k_0,k_0+1}\leq \lambda_{k_0,k_0+2}\leq \dots$ it suffices to consider $\lambda^f_{k_0,k_0+1}.$ In view of Hypothesis \ref{A3c'} and Theorem 3.4 of \cite{salins2017rare} we conclude that
$$\lambda^f_{k_0,1}=a_1^f\bigg(1+\frac{1}{1-e^{-2a_1^f\tau}}\bigg)<\frac{a^f_{k_0+1}}{1-e^{-2\tau a^f_{k_0+1}}}=\lambda^f_{k_0,k_0+1}$$
for all $\tau\in[0,T]$. The proof is complete.\end{proof}

\subsection{Tightness of $\hat{\eta}_{x^*}^{\epsilon, v^\epsilon}$  }\label{tightsec}
Let $v^\epsilon$ be a sequence in $\mathcal{A}$ satisfying the assumptions of Theorem \ref{Asymptoticthm},  $u_{k_0}$ as in \eqref{uchoice} and $u_{k_0}^\epsilon:[0,T]\times\h\rightarrow\h$ be a sequence that converges pointwise and uniformly over bounded subsets of $\h$ to $u_{k_0}.$ The goal of this section is to prove tightness estimates for the collection $\{ \hat{\eta}_{x^*}^{\epsilon, v^\epsilon}:\epsilon<\epsilon_0 \}$ of $C([0,T];\mathcal{X})-$valued random elements. Throughout the rest of this section we drop the index $k_0$ and write $u\equiv u_{k_0}, u^\epsilon\equiv u^{\epsilon}_{k_0} .$

 Recall that for each $\epsilon,$ $\hat{\eta}_{x^*}^{\epsilon, v^\epsilon}$ is the unique mild solution of the controlled equation \eqref{controleq} with $v=v^\epsilon, u=u^\epsilon.$ Existence and uniqueness is once again provided by Theorem 2.2 of \cite{cerrai2004large} (see also Theorem 7.1 of \cite{salins2020systems}). The following lemma guarantees that, for $\epsilon$ small, the sequence $v^{\epsilon}$ is bounded in $L^2$.
\begin{lem}\label{uniboundlem} There exists $\epsilon_0>0$ and a constant $C>0$ such that
	\begin{equation*}
	  \sup_{\epsilon<\epsilon_0}\ex\int_{0}^{\hat{\tau}^{\epsilon,v^\epsilon}_{x^*}}\big\|v^\epsilon(s)\big\|^2_{\h}ds\leq C.
	\end{equation*}
	
\end{lem}
\begin{proof}
  	In view of the variational representation \eqref{varrep} any approximate minimizer $v^\epsilon\in\mathcal{A}$ satisfies
	\begin{equation*}
	\begin{aligned}
\ex\bigg[\frac{1}{2}\int_{0}^{\hat{\tau}^{\epsilon,v^\epsilon}_{x^*}}\|v^\epsilon(s)\|^2_\h ds\bigg]&\leq  -\frac{1}{h^2(\epsilon)}\log Q^{\epsilon}(u)+\ex\int_{0}^{\hat{\tau}^{\epsilon,v^\epsilon}_{x^*}}\|u^\epsilon\big(s,\hat{\eta}^{\epsilon,v^\epsilon}_{x^*}(s)\big )\|^2_\h ds+\epsilon^2
\end{aligned}
	\end{equation*}
Now from the MDP for bounded functionals (see Definition \ref{MDPdef} as well as Remark \ref{MDPrem} below), along with Lemma \ref{Gupperlem}, there exists a constant $C>0$ such that, for $\epsilon$ sufficiently small,
$$ -\frac{1}{h^2(\epsilon)}\log Q^{\epsilon}(u)\leq C.$$
Hence, 
from the uniform convergence of $u^\epsilon$ to $u$ and the uniform boundedness of $u$ in bounded subsets of $\h$ and the fact that  $\hat{\tau}^{\epsilon,v^\epsilon}_{x^*}\leq T$ with probability $1$ the estimate follows.
\end{proof}
\begin{rem}\label{extensionrem} Without loss of generality, we can trivially extend the controls $v^{\epsilon}$ to $[0,T]$ by letting $v^{\epsilon}(t)=0 $ for $t\in[ \hat{\tau}^{\epsilon,v^\epsilon}_{x^*}, T   ].$ This convention will be in use for the rest of this section.
\end{rem}
\noindent We shall now proceed to the proof of tightness estimates.
\begin{lem}\label{tightlem}
	Let $p\geq 1$. For all $\epsilon,T>0$, there exist $\epsilon_0>0, \alpha,\beta>0$ such that
	
	\begin{equation}\label{etapriori}
	\begin{aligned}
	\sup_{\epsilon<\epsilon_0}\ex\sup_{t\in[0,T]}\big\|\hat{\eta}^{\epsilon,v^\epsilon}_{x^*}(t)\|^p_\e+\sup_{\epsilon<\epsilon_0}\ex\sup_{t\in[0,T]}\big\|\hat{\eta}^{\epsilon,v^{\epsilon}}_{x^*}(t)\big\|_{C^a}+     \sup_{\epsilon<\epsilon_0}\ex\big\|\hat{\eta}^{\epsilon,v^{\epsilon}}_{x^*}\big\|_{C^{\beta}([0,T];\e)}<C_{T,\ell,f}
	\end{aligned}
	\end{equation}
	
\end{lem}

\begin{proof} Using the mild formulation we have
\begin{equation}
\label{etamild}
\begin{aligned}
\hat{\eta}_{x^*}^{\epsilon, v^\epsilon}(t)&=\frac{1}{\sqrt{\epsilon}h(\epsilon)}\int_{0}^{t}S(t-s)\big[ F\big(x^*+\sqrt{\epsilon}h(\epsilon)\hat{\eta}^{\epsilon,v^{\epsilon}}_{x^*}(s)  \big)-F\big(x^*  \big) \big]ds\\&+\int_{0}^{t}S(t-s)\big[v^{\epsilon}(s)
-u^\epsilon\big(s,\hat{\eta}^{\epsilon,v^\epsilon}_{x^*}(s)\big )\big]ds+ \frac{1}{h(\epsilon)}\int_{0}^{t}S(t-s)dW(s)
\\&=: \Psi^{\epsilon,v^\epsilon}(t)+U^{\epsilon}(t)+\frac{1}{h(\epsilon)}W_A(t).
\end{aligned}
\end{equation}


\noindent We now fix a version of the process
$ \Psi^{\epsilon,v^\epsilon}(t,\xi)$
and work path-by-path. The paths of $\Psi^{\epsilon,v^\epsilon}$ are weakly differentiable with probability $1$ and
\begin{equation*}
\partial_t\Psi^{\epsilon,v^\epsilon}(t,\xi)=\mathcal{A}\Psi^{\epsilon,v^\epsilon}(t,\xi)+\frac{1}{\sqrt{\epsilon}h(\epsilon)}\big[ F\big(x^*+\sqrt{\epsilon}h(\epsilon)\hat{\eta}^{\epsilon,v^{\epsilon}}_{x^*}(t)  \big)- F\big(x^*  \big) \big](\xi),
\end{equation*}
\noindent with $\mathcal{A}$ as in \eqref{scriptA}. Next, let $t\in[0,T]$ and choose $\xi_t\in[0,L]$ to be such that
\begin{equation*}
\label{supopt}
\|\Psi^{\epsilon,v^\epsilon}(t)\|_\mathcal{E}=\Psi^{\epsilon,v^\epsilon}(t,\xi_t)\text{sign}\big(\Psi^{\epsilon,v^\epsilon}(t,\xi_t)\big)
\end{equation*}
 In view of Proposition A.1 in \cite{salins2020systems} (see also Proposition D.4 of \cite{da2014stochastic}) we can estimate the left derivative of the supremum norm $\| \Psi^{\epsilon,v^\epsilon}(t)\|_\mathcal{E}$ by
\begin{equation*}
\begin{aligned}
\frac{d^-}{dt}\|\Psi^{\epsilon,v^\epsilon}(t)\|_\mathcal{E}&\leq \mathcal{A}\Psi^{\epsilon,v^\epsilon}(t,\xi_t)
\text{sign}\big(\Psi^{\epsilon,v^\epsilon}(t,\xi_t)\big)\\&+\frac{1}{\sqrt{\epsilon}h(\epsilon)}\big[ f\big( x^*(\xi_t)+\sqrt{\epsilon}h(\epsilon)\hat{\eta}^{\epsilon,v^{\epsilon}}_{x^*}(t,\xi_t)  \big)-f\big(x^*(\xi_t)  \big) \big]\text{sign}\big(\Psi^{\epsilon,v^\epsilon}(t,\xi_t)\big).
\end{aligned}
\end{equation*}
From the uniform ellipticity of $\mathcal{A}$ we have for all $ t\in[0,T], \mathcal{A}\Psi^{\epsilon,v^\epsilon}(t,\xi_t)
\text{sign}\big(\Psi^{\epsilon,v^\epsilon}(t,\xi_t)\big)\leq 0.$ Thus, in view of Hypothesis \eqref{A2a}
\begin{equation*}
\begin{aligned}
&\frac{d^-}{dt}\|\Psi^{\epsilon,v^\epsilon}(t)\|_\mathcal{E}\leq \frac{1}{\sqrt{\epsilon}h(\epsilon)}\big[ f_1\big( x^*(\xi_t)+\sqrt{\epsilon}h(\epsilon)\hat{\eta}^{\epsilon,v^{\epsilon}}_{x^*}(t,\xi_t)  \big)-f_1\big(x^*(\xi_t)  \big) \big]\text{sign}\big(\Psi^{\epsilon,v^\epsilon}(t,\xi_t)\big)\\&
+\frac{1}{\sqrt{\epsilon}h(\epsilon)}\big[ f_2\big( x^*(\xi_t)+\sqrt{\epsilon}h(\epsilon)\hat{\eta}^{\epsilon,v^{\epsilon}}_{x^*}(t,\xi_t)  \big)-f_2\big(x^*(\xi_t)  \big) \big]\text{sign}\big(\sqrt{\epsilon}h(\epsilon)\Psi^{\epsilon,v^\epsilon}(t,\xi_t)\big)\\&
\leq M_{f_1}\big|\hat{\eta}^{\epsilon,v^{\epsilon}}_{x^*}(t,\xi_t)\big|+
\frac{1}{\sqrt{\epsilon}h(\epsilon)}\big[ f_2\big( x^*(\xi_t)+\sqrt{\epsilon}h(\epsilon)\Psi^{\epsilon,v^\epsilon}(t,\xi_t)+\sqrt{\epsilon}h(\epsilon)U^{\epsilon}(t,\xi_t)+\sqrt{\epsilon}W_A(t,\xi_t)  \big)-f_2\big(x^*(\xi_t)  \big) \big]\\&\cdot\text{sign}\big(\sqrt{\epsilon}h(\epsilon)\Psi^{\epsilon,v^\epsilon}(t,\xi_t)\big),
\end{aligned}
\end{equation*}
where $M_{f_1}$ is the Lipschitz constant of $f_1$. To proceed, we distinguish the following two cases:\\
\noindent \underline{Case 1:}
$$\text{sign}\big(\sqrt{\epsilon}h(\epsilon)\Psi^{\epsilon,v^\epsilon}(t,\xi_t)\big)=\text{sign}\big(\sqrt{\epsilon}h(\epsilon)\Psi^{\epsilon,v^\epsilon}(t,\xi_t)+\sqrt{\epsilon}h(\epsilon)U^{\epsilon}(t, \xi_t)+\sqrt{\epsilon}W_A(t,\xi_t) \big).$$
Since $f_2$ is non-increasing,
$$  f_2\big( x^*(\xi_t)+\sqrt{\epsilon}h(\epsilon)\Psi^{\epsilon,v^\epsilon}(t,\xi_t)+\sqrt{\epsilon}h(\epsilon)U^{\epsilon}(t)+\sqrt{\epsilon}W_A(t)  \big)-f_2\big(x^*(\xi_t)  \big) \text{sign}\big(\sqrt{\epsilon}h(\epsilon)\Psi^{\epsilon,v^\epsilon}(t,\xi_t)\big)\leq 0.$$
Hence,
\begin{equation}\label{case1}
\begin{aligned}
&\frac{d^-}{dt}\|\Psi^{\epsilon,v^\epsilon}(t)\|_\mathcal{E}\leq M_{f_1}\big|\hat{\eta}^{\epsilon,v^{\epsilon}}_{x^*}(t,\xi_t)\big|\leq M_{f_1}\big\|\hat{\eta}^{\epsilon,v^{\epsilon}}_{x^*}(t)\big\|_{\e} .
\end{aligned}
\end{equation}
\noindent \underline{Case 2:}
$$\text{sign}\big(\sqrt{\epsilon}h(\epsilon)\Psi^{\epsilon,v^\epsilon}(t,\xi_t)\big)\neq\text{sign}\big(\sqrt{\epsilon}h(\epsilon)\Psi^{\epsilon,v^\epsilon}(t,\xi_t)+\sqrt{\epsilon}h(\epsilon)U^{\epsilon}(t,\xi_t)+\sqrt{\epsilon}W_A(t,\xi_t) \big).$$
In this case it is straightforward to verify that
$$\sqrt{\epsilon}h(\epsilon)\big|\Psi^{\epsilon,v^\epsilon}(t,\xi_t)\big|\leq \big|\sqrt{\epsilon}h(\epsilon)U^{\epsilon}(t,\xi_t)+\sqrt{\epsilon}W_A(t,\xi_t)\big|.$$ The reader is referred to the proof of Theorem 6.1 of \cite{salins2020systems} for a similar argument. The latter, along with the optimality of $\xi_t$, yields
\begin{equation}\label{case2}
\begin{aligned}
\big\|\Psi^{\epsilon,v^\epsilon}(t)\big\|_{\e}=\big|\Psi^{\epsilon,v^\epsilon}(t,\xi_t)\big|\leq  \bigg|U^{\epsilon}(t,\xi_t)+\frac{1}{h(\epsilon)}W_A(t,\xi_t)\bigg|\leq\bigg\|U^{\epsilon}(t)+\frac{1}{h(\epsilon)}W_A(t)\bigg\|_{\e}.
\end{aligned}
\end{equation}
Setting $\Xi^{\epsilon,v^\epsilon}(t):=\max\{ \big\|U^{\epsilon}+W_A/h(\epsilon)\big\|_{C([0,T];\e)}, \big\|\Psi^{\epsilon,v^\epsilon}(t)\big\|_{\e} \}$, we can combine \eqref{case1}, \eqref{case2} and the mean value inequality to obtain
\begin{equation*}
\begin{aligned}
\Xi^{\epsilon,v^\epsilon}(t)- \big\|U^{\epsilon}+W_A/h(\epsilon)\big\|_{C([0,T];\e)}&= \Xi^{\epsilon,v^\epsilon}(t)-  \Xi^{\epsilon,v^\epsilon}(0)\\&
\leq \int_{0}^{t}\frac{d^-}{ds} \Xi^{\epsilon,v^\epsilon}(s)ds\\&
\leq M_{f_1}\int_{0}^{t}\big\|\hat{\eta}^{\epsilon,v^{\epsilon}}_{x^*}(s)\big\|_{\e}ds\\&\leq M_{f_1}\int_{0}^{t}\big[\big\|\Psi^{\epsilon,v^\epsilon}(s)\big\|_{\e}+\big\|U^{\epsilon}+W_A\big/h(\epsilon)\|_{C([0,T];\e)}      \big]ds\\&
\leq 2M_{f_1}\int_{0}^{t}\Xi^{\epsilon,v^\epsilon}(s)ds.
\end{aligned}
\end{equation*}
By Gr\"onwall's inequality,
\begin{equation*}
\begin{aligned}
\big\|\Psi^{\epsilon,v^\epsilon}(t)\big\|_{\e}\leq \Xi^{\epsilon,v^\epsilon}(t)\leq C_{T,\phi}\big\|U^{\epsilon}+W_A/h(\epsilon)\big\|_{C([0,T];\e)},
\end{aligned}
\end{equation*}
where $C_{T,\phi}=e^{2M_{f_1} T}$. Since the latter holds for all $t\in[0,T]$ we obtain
\begin{equation}\label{psibnd}
\big\|\Psi^{\epsilon,v^\epsilon}\big\|_{C([0,T];\e)}\leq C_{T,\phi}\big\|U^{\epsilon}+W_A/h(\epsilon)\big\|_{C([0,T];\e)}.
\end{equation}
Turning to the control term,
\begin{equation*}
\begin{aligned}
\big\|U^{\epsilon}(t)\big\|_{\e}\leq C\bigg\|\int_{0}^{t}S(t-s)v^{\epsilon}(s)
ds\bigg\|_{H^\theta(0,L)}  +\bigg\|\int_{0}^{t}S(t-s)u^\epsilon\big(s,\hat{\eta}^{\epsilon,v^\epsilon}_{x^*}(s)\big )
ds\bigg\|_\e
\end{aligned}
\end{equation*}
for any $\theta>1/2$. This is a consequence of the embedding $W^{\theta,p}(O)\hookrightarrow\mathcal{E}$ which holds for smooth domains $O\subset\R^d$ and all $\theta>d/p$. From the smoothing property \eqref{Sobosmoothing}, the Cauchy-Schwarz inequality, the uniform convergence of $u^\epsilon$ to $u$ and \eqref{uchoice} we have
\begin{equation}\label{Ubnd}
\begin{aligned}
\big\|U^{\epsilon}(t)\big\|_{\e}&\leq C_{T,\theta}\int_{0}^{t}(t-s)^{-\frac{\theta}{2}}\big\|v^{\epsilon}(s)
\big\|_\h ds+C_{T}\int_{0}^{t}\|u^\epsilon\big(s,\hat{\eta}^{\epsilon,v^\epsilon}_{x^*}(s)\big )
\big\|_\e ds\\&
\leq C\bigg(\int_{0}^{t}(t-s)^{-\theta}ds\bigg)^{\frac{1}{2}}\bigg(\int_{0}^{t}\big\|v^{\epsilon}(s)\big\|^2_{\h} ds \bigg)^{\frac{1}{2}}  + C_{T}\int_{0}^{t}\big(\|u\big(s,\hat{\eta}^{\epsilon,v^\epsilon}_{x^*}(s)\big )
\big\|_\e+\rho\big)ds
\\&
\leq C_{\theta} T^{(1-\theta)/2}\|v^\epsilon\|_{L^2([0,T];\h)}+\rho T C_{T}+2\lambda_1^f\|e_1^f\|^2_\e\int_{0}^{T}  \|\hat{\eta}^{\epsilon,v^\epsilon}_{x^*}(s)\|_\e ds
\end{aligned}
\end{equation}
which holds w.p. 1 for $\theta<1$, $\epsilon$ sufficiently small and $\rho>0$. As for the  stochastic convolution term we have $h(\epsilon)\to\infty$ and \eqref{convolutionbnd} yields
\begin{equation*}\label{Wbnd}
\ex\sup_{t\in[0,T]}\big\|  W_A(t)/h(\epsilon)\big\|^p_{\e}\leq C
\end{equation*}
for $\epsilon$ small and some $C>0$ independent of $\epsilon$. The estimate is a consequence of the Sobolev embedding theorem along with heat kernel estimates and the stochastic factorization formula. Combining \eqref{etamild}, \eqref{psibnd}, \eqref{Ubnd}, Lemma \ref{uniboundlem} and Remark \ref{extensionrem} we obtain
\begin{equation*}
\begin{aligned}
\ex \sup_{t\in[0,T]}\big\|\hat{\eta}^{\epsilon,v^\epsilon}_{x^*}(t)\|^p_\e&\leq  C\ex\sup_{t\in[0,T]}\big(\big\|\Psi^{\epsilon,v^\epsilon}(t)\big\|^p_{\e}+ \big\|U^{\epsilon}(t)\big\|^p_{\e}+\big\|W_A^{\epsilon}(t)/h(\epsilon)\big\|^p_{\e}\big)\\&
\leq C_{p,T}\bigg(1+\frac{1}{h^p(\epsilon)}+\int_{0}^{T}  \ex\sup_{s\in[0,t]}\|\hat{\eta}^{\epsilon,v^\epsilon}_{x^*}(s)\|^p_\e dt\bigg)
\end{aligned}
\end{equation*}
and $h(\epsilon)\to\infty$ as $\epsilon\to 0$.
Another application of Gr\"onwall's inequality leads to
\begin{equation}\label{etasup}
\sup_{\epsilon<\epsilon_0}\ex\sup_{t\in[0,T]}\big\|\hat{\eta}^{\epsilon,v^\epsilon}_{x^*}(t)\|^p_\e\leq C
\end{equation}
which is the first estimate in \eqref{etapriori}. Note here that $C$ does not depend on $x^*$. Turning to the spatial H\"older regularity, an application of Taylor's theorem for G\^ateaux derivatives yields
\begin{equation}\label{Taylor}
\begin{aligned}
\Psi^{\epsilon,v^\epsilon}(t)&=\frac{1}{\sqrt{\epsilon}h(\epsilon)}\int_{0}^{t}S(t-s)\big[ F\big(x^*+\sqrt{\epsilon}h(\epsilon)\hat{\eta}^{\epsilon,v^{\epsilon}}_{x^*}(s)  \big)-F\big(x^*  \big) \big]ds\\&
=\int_{0}^{t}S(t-s) \bigg[DF\big(x^*\big)\big(\hat{\eta}^{\epsilon,v^{\epsilon}}_{x^*}(s)\big)+  \frac{\sqrt{\epsilon}h(\epsilon)}{2} D^2F\big(x^*+\theta_0\sqrt{\epsilon}h(\epsilon)\hat{\eta}^{\epsilon,v^{\epsilon}}_{x^*}(s)  \big)\big(\hat{\eta}^{\epsilon,v^{\epsilon}}_{x^*}(s), \hat{\eta}^{\epsilon,v^{\epsilon}}_{x^*}(s)          \big) \bigg]ds
\end{aligned}
\end{equation}
for some $\theta_0\in(0,1)$. Let $\theta>1/2$ and $\alpha=(2\theta-1)/2$.       By virtue of the Sobolev embedding theorem (see e.g. Theorem 8.2 in \cite{di2012hitchhikers}) and Hypothesis \ref{A2b} we have
\begin{equation*}
\begin{aligned}
\big\|&\Psi^{\epsilon,v^\epsilon}(t)\|_{C^\alpha}\leq C \bigg\|\int_{0}^{t}S(t-s) \bigg[DF\big(x^*\big)\big(\hat{\eta}^{\epsilon,v^{\epsilon}}_{x^*}(s)\big)+  \sqrt{\epsilon}h(\epsilon) D^2F\big(x^*+\theta_0\sqrt{\epsilon}h(\epsilon)\hat{\eta}^{\epsilon,v^{\epsilon}}_{x^*}(s)  \big)\big(\hat{\eta}^{\epsilon,v^{\epsilon}}_{x^*}(s), \hat{\eta}^{\epsilon,v^{\epsilon}}_{x^*}(s)          \big) \bigg]ds\bigg\|_{H^\theta}\\&
\leq C\int_{0}^{t}(t-s)^{-\theta/2} \big\|DF\big(x^*\big)\big(\hat{\eta}^{\epsilon,v^{\epsilon}}_{x^*}(s)\big)+  \sqrt{\epsilon}h(\epsilon) D^2F\big(x^*+\theta_0\sqrt{\epsilon}h(\epsilon)\hat{\eta}^{\epsilon,v^{\epsilon}}_{x^*}(s)  \big)\big(\hat{\eta}^{\epsilon,v^{\epsilon}}_{x^*}(s), \hat{\eta}^{\epsilon,v^{\epsilon}}_{x^*}(s)          \big) \big\|_\h ds\\&
\leq C_f\int_{0}^{t}(t-s)^{-\theta/2}\bigg[\big(1+\|x^*\|^{p_0-1}_\e\big)\big\| \hat{\eta}^{\epsilon,v^{\epsilon}}_{x^*}(s) \big\|_\e+\sqrt{\epsilon}h(\epsilon)\big(1+\big\| x^*+\theta_0\sqrt{\epsilon}h(\epsilon)\hat{\eta}^{\epsilon,v^{\epsilon}}_{x^*}(s)    \big\|^{p_0-2}_\e\big)\big\| \hat{\eta}^{\epsilon,v^{\epsilon}}_{x^*}(s) \big\|^2_\e\bigg]ds\\&
\leq C_{f,\theta,p_0,x^*}\bigg[1+\sup_{t\in[0,T]}\big\| \hat{\eta}^{\epsilon,v^{\epsilon}}_{x^*}(s) \big\|^{p_0}_\e\bigg]t^{1-\theta/2}.
\end{aligned}
\end{equation*}
In view of \eqref{etasup},

\begin{equation}\label{psiholder}
\begin{aligned}
\sup_{\epsilon<\epsilon_0}\ex\sup_{t\in[0,T]}\big\|&\Psi^{\epsilon,v^\epsilon}(t)\|_{C^\alpha}\leq C_{T,f,\theta,p_0,x^*} \bigg[1+\sup_{\epsilon<\epsilon_0}\ex\sup_{t\in[0,T]}\big\| \hat{\eta}^{\epsilon,v^{\epsilon}}_{x^*}(s) \big\|^{p_0}_\e\bigg]<\infty.
\end{aligned}
\end{equation}
Repeating similar arguments to the ones used in \eqref{Ubnd} we see that
\begin{equation}\label{Uholder}
\begin{aligned}
\sup_{\epsilon<\epsilon_0}\ex\sup_{t\in[0,T]}\big\|U^{\epsilon}(t)\big\|_{C^\alpha}\leq C_{N,T,\theta,f}\bigg[1+\sup_{\epsilon<\epsilon_0}\ex\sup_{t\in[0,T]}\big\| \hat{\eta}^{\epsilon,v^{\epsilon}}_{x^*}(s) \big\|_\e\bigg]<\infty.
\end{aligned}
\end{equation}
Moreover, we have the following well-known spatial equicontinuity estimate for the stochastic convolution
\begin{equation}\label{Wholder}
\begin{aligned}
\ex\sup_{t\in[0,T]}\big\|  W_A(t)\big\|_{C^\alpha}\leq C.
\end{aligned}
\end{equation}
The reader is refered to \cite{da2014stochastic}, Theorems 5.16, 5.22 for the proof and a detailed discussion of regularity properties of stochastic convolutions. Combining the latter along with \eqref{psiholder} and \eqref{Uholder} we deduce that for each $\epsilon>0, t\in[0,T]$ $\hat{\eta}^{\epsilon,v^{\epsilon}}_{x^*}(t)\in C^a$ w.p. $1$ and furthermore
\begin{equation*}\label{etaholder}
\sup_{\epsilon<\epsilon_0}\ex\sup_{t\in[0,T]}\big\|\hat{\eta}^{\epsilon,v^{\epsilon}}_{x^*}(t)\big\|_{C^a}<\infty,
\end{equation*}
for some  sufficiently small $\epsilon_0$. It remains to study the temporal equicontinuity of $\hat{\eta}^{\epsilon,v^{\epsilon}}_{x^*}$. Letting $s<t\in[0,T]$ we have
\begin{equation*}
\begin{aligned}
\hat{\eta}^{\epsilon,v^{\epsilon}}_{x^*}(t)-\hat{\eta}^{\epsilon,v^{\epsilon}}_{x^*}(s)- [S(t-s)-I]\hat{\eta}^{\epsilon,v^{\epsilon}}_{x^*}(s)&=\frac{1}{\sqrt{\epsilon}h(\epsilon)}\int_{s}^{t}S(t-r)\big[ F\big(x^*+\sqrt{\epsilon}h(\epsilon)\hat{\eta}^{\epsilon,v^{\epsilon}}_{x^*}(r)  \big)-F\big(x^*  \big) \big]dr\\&+\int_{s}^{t}S(t-r)\big[v^{\epsilon}(r)
-u^\epsilon\big(r,\hat{\eta}^{\epsilon,v^\epsilon}_{x^*}(r)\big )\big]dr+ \frac{1}{h(\epsilon)}\int_{s}^{t}S(t-r)dW(r)
\\&=: \Psi^{\epsilon,v^\epsilon}(s,t)+U^{\epsilon}(s,t)+\frac{1}{h(\epsilon)}W_A(s,t).
\end{aligned}
\end{equation*}
Hence,
\begin{equation}\label{preequieta}
\begin{aligned}
\big\|\hat{\eta}^{\epsilon,v^{\epsilon}}_{x^*}(t)-\hat{\eta}^{\epsilon,v^{\epsilon}}_{x^*}(s)\big\|_\e\leq \big\|\Psi^{\epsilon,v^\epsilon}(s,t)\big\|_\e+\big\|U^{\epsilon}(s,t)\big\|_\e+\big\|W_A(s,t)\big\|_\e+\big\|[S(t-s)-I]\hat{\eta}^{\epsilon,v^{\epsilon}}_{x^*}(s)\big\|_\e.
\end{aligned}
\end{equation}
From the estimates preceding \eqref{psiholder} and the arguments in \eqref{Ubnd} we obtain
\begin{equation}\label{equipsi}
\sup_{\epsilon<\epsilon_0}\ex\sup_{s\neq t\in[0,T]}\frac{\big\|\Psi^{\epsilon,v^\epsilon}(s,t)\big\|_\e}{|t-s|^{1-\theta/2}}\leq C_{f,\theta,p_0,x^*}\bigg[1+\sup_{\epsilon<\epsilon_0}\ex\sup_{t\in[0,T]}\big\| \hat{\eta}^{\epsilon,v^{\epsilon}}_{x^*}(s) \big\|^{p_0}_\e\bigg]<\infty
\end{equation}
and
\begin{equation} \label{equiU}
\sup_{\epsilon<\epsilon_0}\ex\sup_{s\neq t\in[0,T]}\frac{\big\|U^{\epsilon}(s,t)\big\|_\e}{|t-s|^{1-\theta/2}}\leq   C_{\theta,N,T,f}\bigg[ 1+   \sup_{\epsilon<\epsilon_0}\ex\sup_{t\in[0,T]}\big\| \hat{\eta}^{\epsilon,v^{\epsilon}}_{x^*}(s) \big\|_\e   \bigg]<\infty
\end{equation}
respectively. As for the stochastic convolution, there exists $\beta\in(0,1)$ such that
\begin{equation}\label{equiW}
\ex\big[ W_{A}\big]_{C^{\beta}([0,T];\e)}\leq C
\end{equation}
(see e.g. \cite{da2014stochastic}, Theorem 5.22). Finally, let $\theta>0, \beta\in(0,1/2)$ such that $\beta+\theta/2<1$. From the Sobolev embedding theorem and \eqref{sobcont}
\begin{equation*}
\begin{aligned}
\big\|[S(t-s)-I]\hat{\eta}^{\epsilon,v^{\epsilon}}_{x^*}(s)\big\|_\e&\leq C \big\|[S(t-s)-I]\hat{\eta}^{\epsilon,v^{\epsilon}}_{x^*}(s)\big\|_{H^\theta}\\&\leq C \big\|[S(t-s)-I](-A)^{\frac{\theta}{2}}\hat{\eta}^{\epsilon,v^{\epsilon}}_{x^*}(s)\big\|_{\h}\\&\leq C
\big\|S(t-s)-I\big\|_{\mathscr{L}(H^\beta;\h)}\big\|(-A)^{\beta+\frac{\theta}{2}}\hat{\eta}^{\epsilon,v^{\epsilon}}_{x^*}(s)\big\|_{\h}\\&
\leq C(t-s)^\beta\big\|\hat{\eta}^{\epsilon,v^{\epsilon}}_{x^*}(s)\big\|_{H^{2\beta+\theta}}.
\end{aligned}
\end{equation*}
Following the derivation of the estimates \eqref{psiholder}, \eqref{Uholder}, \eqref{Wholder} (see also Lemma A.3 in \cite{gasteratos2022moderate}) we deduce that

\begin{equation*}
\begin{aligned}
\ex\sup_{s\neq t\in[0,T]}\frac{\big\|[S(t-s)-I]\hat{\eta}^{\epsilon,v^{\epsilon}}_{x^*}(s)\big\|_\e}{|t-s|^{\beta_0}}&\leq C \ex\sup_{s\in[0,T]}\big\|\hat{\eta}^{\epsilon,v^{\epsilon}}_{x^*}(s)\big\|_{H^{2\beta+\theta}}\\&
\leq C\bigg[  1+ \frac{1}{h(\epsilon)}+ \ex\sup_{t\in[0,T]}\big\| \hat{\eta}^{\epsilon,v^{\epsilon}}_{x^*}(s) \big\|^{p_0}_\e    \bigg]<\infty.
\end{aligned}
\end{equation*}
From the latter and \eqref{preequieta}-\eqref{equiW}, there exists a sufficiently small $\epsilon_0$ and $\beta>0$ such that
\begin{equation*}\label{equieta}
\sup_{\epsilon<\epsilon_0}\ex\sup_{t\in[0,T]}\big\|\hat{\eta}^{\epsilon,v^{\epsilon}}_{x^*}\big\|_{C^{\beta}([0,T];\e)}<\infty.
\end{equation*}
This proves the last estimate in \eqref{etapriori} and completes the proof.
\end{proof}

\noindent From Lemma \ref{tightlem}, along with an infinite-dimensional version of the Arzel\`a-Ascoli theorem, it follows that the family of laws of the controlled processes $\{\eta^{\epsilon,v^\epsilon}_{x^*}\}_{\epsilon}$ is concentrated on compact subsets of $C([0,T];\mathcal{E}),$ uniformly over sufficiently small values of $\epsilon$. Thus, in view of Prokhorov's theorem (Theorem \ref{Prokhorov} below), it forms a relatively compact set in the topology of weak convergence of measures in $C([0,T];\mathcal{E})$. In the next section we aim to characterize the limit points as $\epsilon\to 0$.

\subsection{Limiting behavior of  $\hat{\eta}_{x^*}^{\epsilon, v^\epsilon}$}\label{limsec}
 Before we proceed to the main body of this section let us recall the notion of a tight family of probability measures and the classical theorem of Prokhorov.

 \begin{dfn}\label{D:TightnessDef}
 	Let $\mathcal{Z}$ be a Polish space and $\Pi\subset\mathscr{P}(\mathcal{Z})$ be a set of Borel probability measures on $\mathcal{Z}$ and $\{P_n\}_{n\in\N}\subset\Pi.$
 	\noindent We say that
 	(i) $P_n$ converges weakly to a measure $P\in\mathscr{P}(\mathcal{Z})$ if for every $f\in C_b(\mathcal{Z})$
 	\begin{equation*}
 	\lim_{n\to\infty} \int_{\mathcal{Z}} f dP_n= \int_{\mathcal{Z}} f dP.
 	\end{equation*}
 	(ii) $\Pi$ is \textit{tight} if for each $\epsilon>0$ there exists a compact set $K_\epsilon\subset \mathcal{Z}$ such that for all $P\in\Pi$,
 	\begin{equation*}
 	P(\mathcal{Z}\setminus K_\epsilon)<\epsilon.
 	\end{equation*}
 \end{dfn}
 \noindent Prokhorov's theorem asserts that  the notions of tightness and relative weak sequential compactness are equivalent for Borel measures on Polish spaces.
 \begin{thm}(Prokhorov)\label{Prokhorov} Let $\mathcal{Z}$ be a Polish space and $\Pi\subset\mathscr{P}(\mathcal{Z})$ be a tight family of Borel probability measures. Then  every sequence in $\Pi$ contains a weakly convergent subsequence.
 \end{thm}

\begin{lem}\label{limlem} Let $\epsilon_0$ be sufficiently small, $v^\epsilon$ be a sequence in $\mathcal{A}$ satisfying the assumptions of Theorem \ref{Asymptoticthm},  $u$ as in \eqref{uchoice} and $u^\epsilon:[0,T]\times\h\rightarrow\h$ be a sequence that converges pointwise and uniformly over bounded subsets of $\h$ to $u.$ Any sequence in $ \{(\hat{\eta}^{\epsilon,v^\epsilon}_{x^*}, v^{\epsilon})\}_{\epsilon<\epsilon_0}$ has a further subsequence that converges in distribution in $C([0,T];\mathcal{E})\times L^2([0,T];\h)$ to a pair $(\hat{\eta}^{v^0}_{x^*}, v^0)$ in the product of uniform and weak topologies. Moreover:\\
$(i)$   $\hat{\eta}^{v^0}_{x^*}$ is equal in law to the (unique) solution of
\begin{equation}\label{limeq}
   \big\{ \dot{\phi}(t) = [A +DF(x^*)]\phi(t)-u(t,\phi(t))+v^0(t),\;\; \phi(0)=0\big\},
\end{equation}
$(ii)$ Any sequence in $\{\hat{\tau}^{\epsilon,v^\epsilon}_{x^*}\;;\epsilon<\epsilon_0\}$ converges in distribution to a $[0,T]$-valued random variable $\hat{\tau}^{v^0}$ such that $$\hat{\eta}^{v^0}_{x^*}(\hat{\tau}^{v^0})\in\partial B_\h(0,L)$$ and for all
$t<\hat{\tau}^{v^0},$ $\hat{\eta}^{v^0}_{x^*}(t)\in B_\h(0,L)$ with probability $1$ (recall that $B_\h(0,L)$ denotes a closed ball on $\h$).
\end{lem}
\begin{proof} Starting from the controls $v^\epsilon,$ Lemma \ref{uniboundlem} along with Remark \ref{extensionrem} yield
		\begin{equation*}
		\sup_{\epsilon>0}\ex\int_{0}^{T}\big\|v^\epsilon(t)\big\|^2_{\h}dt<\infty.
		\end{equation*}
Since any bounded subset of $L^2([0,T];\h)$ is relatively compact in the weak topology, we deduce from the discussion after Lemma \ref{tightlem} that the family of laws of the pairs  $\{(\hat{\eta}^{v^\epsilon}_{x^*}, v^{\epsilon})\}_{\epsilon<\epsilon_0}$ is tight. By virtue of Prokhorov's theorem any sequence of such elements contains a subsequence (denoted with the same notation) that converge in distribution to a pair $(\hat{\eta}_{x^*}, v^{0})$ of $C([0,T];\mathcal{E})\times L^2([0,T];\h)$-valued random elements. We remark here that $L^2([0,T];\h)$ with the weak topology is not globally metrizable, hence not a Polish space, and Prokhorov's theorem is not directly applicable. However the same conclusions can be drawn by a more general version of the theorem (e.g. Theorem 8.6.7 in \cite{bogachev2007measure}).  Invoking Skorokhod's theorem we can now assume that this convergence happens almost surely. This theorem involves the introduction of a new probability space with respect to which the convergence takes place. This will not be reflected in our notation for the sake of convenience. We will now characterize the law of  $\hat{\eta}_{x^*}$.

$(i)   $   Recall that for all $t\in[0,T]$
\begin{equation*}
\label{etamild2}
\begin{aligned}
\hat{\eta}_{x^*}^{\epsilon, v^\epsilon}(t)&=\frac{1}{\sqrt{\epsilon}h(\epsilon)}\int_{0}^{t}S(t-s)\big[ F\big(x^*+\sqrt{\epsilon}h(\epsilon)\hat{\eta}^{\epsilon,v^{\epsilon}}_{x^*}(s)  \big)-F\big(x^*  \big) \big]ds\\&+\int_{0}^{t}S(t-s)v^{\epsilon}(s)ds
-\int_{0}^{t}S(t-s)u^\epsilon\big(s,\hat{\eta}^{\epsilon,v^\epsilon}_{x^*}(s)\big )+ \frac{1}{h(\epsilon)}W_A(t)
\end{aligned}
\end{equation*}
with probability $1$. Starting from the last term, the estimate \eqref{convolutionbnd} yields  $\frac{1}{h(\epsilon)}W_A\longrightarrow 0$ in $L^p(\Omega; C([0,T];\mathcal{E}))$ for any $p\geq 1$. Next, from Lemma 4.7 in \cite{salins2017rare} we have
\begin{equation*}
\int_{0}^{\cdot}S(\cdot-s)v^\epsilon(s)ds\longrightarrow\int_{0}^{\cdot}S(\cdot-s)v^0(s)ds
\end{equation*}
almost surely in $C([0,T];\mathcal{E}).$ As for the term involving the changes of measure $u^\epsilon$
\begin{equation*}
\begin{aligned}
\ex\sup_{t\in[0,T]} \bigg\|\int_{0}^{t}S(t-s)\big[u^\epsilon\big(s,\hat{\eta}^{\epsilon,v^\epsilon}_{x^*}(s)\big )-u\big(s,\hat{\eta}_{x^*}(s)\big )\big]
ds\bigg\|_\e&\leq  c\ex\int_{0}^{T}\big\|u^\epsilon\big(s,\hat{\eta}^{\epsilon,v^\epsilon}_{x^*}(s)\big )-u\big(s,\hat{\eta}^{\epsilon,v^\epsilon}_{x^*}(s)\big )\big\|_\e
ds\\&
+c\ex\int_{0}^{T}\big\|u\big(s,\hat{\eta}^{\epsilon,v^\epsilon}_{x^*}(s)\big )-u\big(s,\hat{\eta}_{x^*}(s)\big )\big\|_\e
ds.
\end{aligned}
\end{equation*}
The first term on the right hand side converges to $0$ by our assumptions along with \eqref{etasup}. The almost sure convergence of $\hat{\eta}^{\epsilon,v^\epsilon}_{x^*}$ and the continuity of $u$  (see  \eqref{uchoice}) along with the dominated convergence theorem imply the convergence of the second term to $0$. Next, in view of \eqref{Taylor},  Hypothesis \ref{A2b} and the dominated convergence theorem we have
\begin{equation*}
\begin{aligned}
\ex\sup_{t\in[0,T]}\bigg\|\frac{1}{\sqrt{\epsilon}h(\epsilon)}\int_{0}^{t}S(t-s)\big[ F\big(x^*+\sqrt{\epsilon}h(\epsilon)\hat{\eta}^{\epsilon,v^{\epsilon}}_{x^*}(s)  \big)-F\big(x^*  \big)- DF\big(x^*\big)\big(\hat{\eta}_{x^*}(s)\big) \big]ds\bigg\|_\e\longrightarrow 0
\end{aligned}
\end{equation*}
as $\epsilon\to 0$. Uniqueness of \eqref{limeq} along with a subsequence argument complete the proof. \\
$(ii)$ Since $[0,T]$ is compact in the standard topology, the family of $[0,T]$-valued random variables $\{\hat{\tau}^{\epsilon,v^\epsilon}_{x^*}\}_{\epsilon<\epsilon_0}$ is tight. Invoking Prokhorov's and Skorokhod's theorems once again, any sequence in this family has a subsequence that converges almost surely to a $[0,T]$-valued random variable $\hat{\tau}^{v^0}_{x^*}.$ From the almost sure convergence of $\hat{\eta}^{\epsilon,v^{\epsilon}}_{x^*}$ and the definition of $\hat{\tau}^{\epsilon,v^\epsilon}_{x^*}$ (see Lemma \ref{varreplem}),  $\hat{\eta}^{\epsilon,v^{\epsilon}}_{x^*}\big(\hat{\tau}^{\epsilon,v^{\epsilon}}_{x^*}\big)\longrightarrow  \hat{\eta}_{x^*}^{ v^0}\big(\hat{\tau}^{v^0}_{x^*}\big)\in\partial B_{\h}(0,L)$ almost surely (the latter being a closed set ). Moreover, for any $t<\hat{\tau}^{v^0}_{x^*},$ there exists $\delta>0$ and $\epsilon_0>0$ sufficiently small such that $t\leq \hat{\tau}^{v^0}_{x^*}-\delta<\hat{\tau}^{\epsilon,v^{\epsilon}}_{x^*}$ for all $\epsilon\leq \epsilon_0$ on a set of probability $1$. Thus, for $\epsilon$ sufficiently small, $\{ \hat{\eta}^{\epsilon,v^{\epsilon}}_{x^*}\big(t\big)  \}_{\epsilon}\subset B_{\h}(0,L)$ and the pointwise limit $\hat{\eta}^{\epsilon,v^{0}}_{x^*}\big(t\big)\in B_\h(0,L)$ with probability $1.$ \end{proof}

\begin{rem}\label{triviallimlem} A simple consequence of Lemma \ref{limlem} is that the moderate deviation process $\eta^{\epsilon}_{x}$ \eqref{etadef} which results by setting $u=v^\epsilon=0$ in \eqref{controleq}, converges as $\epsilon\to 0$ to the solution of the linear deterministic PDE $\dot{\phi}(t) = [A +DF(X^0_x(t))]\phi(t)$ with zero initial condition, i.e. $\eta^{\epsilon}_{x}\rightarrow 0.$
	
\end{rem}

\subsection{Proof of Theorem \ref{MDPthm}} Before we move on to the proof we remind the reader that the index $k_0$ has been dropped.

Let $\epsilon>0$. Returning to \eqref{varrep}, choose a sequence $\{v^\epsilon\}\subset\mathcal{A}$ of approximate minimizers such that \eqref{approxmin} holds. Since $u^\epsilon$ converges uniformly to $u$ over bounded subsets, there exists $\epsilon_0$ sufficiently small such that for any $\delta>0$ and $\epsilon<\epsilon_0$
\begin{equation}\label{preupperbound}
\begin{aligned}
-\frac{1}{h^2(\epsilon)}\log Q^{\epsilon}(u^\epsilon)\geq \ex\bigg[\frac{1}{2}\int_{0}^{\hat{\tau}^{\epsilon,v^\epsilon}_{x^*}}\|v^\epsilon(s)\|^2_\h ds-\int_{0}^{\hat{\tau}^{\epsilon,v^\epsilon}_{x^*}}\|u\big(s,\hat{\eta}^{\epsilon,v^\epsilon}_{x^*}(s)\big )\|^2_\h ds\bigg]-\epsilon-\delta.
\end{aligned}
\end{equation}
From the variational representation \eqref{varrep}, Lemma \ref{uniboundlem} and the assumptions on $u^\epsilon$ and $u$ there exists $\epsilon_0$ sufficiently small such that
\begin{equation*}
\sup_{\epsilon<\epsilon_0}\bigg|\frac{1}{h^2(\epsilon)}\log Q^{\epsilon}(u^\epsilon)\bigg|\leq \sup_{\epsilon<\epsilon_0}\ex\frac{1}{2}\int_{0}^{\hat{\tau}^{\epsilon,v^\epsilon}_{x^*}}\|v^\epsilon(s)\|^2_\h ds+\sup_{\epsilon<\epsilon_0}\ex\int_{0}^{T}\|u^\epsilon\big(s,\hat{\eta}^{\epsilon}_{x^*}(s)\big )\|^2_\h ds<\infty.
\end{equation*}
Thus, there exists a sequence in $\epsilon$ over which the left hand side in \eqref{preupperbound} converges to $\liminf_{\epsilon\to0} -\log Q^{\epsilon}(u^\epsilon)/h^2(\epsilon)$. Since the functional $\mathcal{J}: C([0,T];\e)\times L^2([0,T];\h)\times[0, T]\rightarrow \R$,
\begin{equation*}
\label{J}
\mathcal{J}(\eta, v,\tau):=\frac{1}{2}\int_{0}^{\tau}\|v(s)\|^2_\h ds-\int_{0}^{\tau}\|u\big(\eta(s)\big )\|^2_\h ds
\end{equation*}
is lower semi-continuous in the product of uniform, weak and standard topologies, we can pass to a further subsequence and apply the Portmanteau lemma along with Lemma \ref{limlem} to obtain
\begin{equation}\label{upper}
\begin{aligned}
\liminf_{\epsilon\to0}-\frac{1}{h^2(\epsilon)}\log Q^{\epsilon}(u^\epsilon)&\geq \liminf_{\epsilon\to0}\ex  \big[\mathcal{J}\big( \hat{\eta}^{\epsilon,v^\epsilon}_{x^*}, v^\epsilon,\hat{\tau}^{\epsilon,v^\epsilon}_{x^*}\big)\big]-\delta\\&
\geq \ex  \big[\mathcal{J}\big(\hat{\eta}^{v^0}_{x^*}, v^0,\hat{\tau}^{v^0}_{x^*}\big)\big]-\delta\\&
= \ex\bigg[\frac{1}{2}\int_{0}^{\hat{\tau}^{v^0}_{x^*}}\|v^0(s)\|^2_\h ds-\int_{0}^{\hat{\tau}^{v^0}_{x^*}}\|u\big(\hat{\eta}^{v^0}_{x^*}(s)\big )\|^2_\h ds\bigg]-\delta\\&
\geq   \inf_{y\in\mathcal{T}}\inf_{v\in \mathcal{C}_{y,x^*}}\int_{0}^{\tau}\bigg(\frac{1}{2}\|v(s)\|^2_\h-\|u(y(s))\|_\h^2\bigg) ds-\delta,
\end{aligned}
\end{equation}
with $\mathcal{T}$ as in \eqref{scriptT}. Since $\delta$ is arbitrary, the upper bound is complete.
To obtain a lower bound we will use the conclusions of Proposition \ref{kprop} for the limiting variational problem. To this end let $y^*$ satisfy
\begin{equation*}
\begin{aligned}
  \inf_{v\in \mathcal{C}_{y^*,x^*}}\int_{0}^{\tau_{y^*}}\bigg(\frac{1}{2}\|v(s)\|^2_\h-\|u(y^*(s))\|_\h^2\bigg) ds  =\inf_{y\in\mathcal{T}}\inf_{v\in \mathcal{C}_{y,x^*}}\int_{0}^{\tau}\bigg(\frac{1}{2}\|v(s)\|^2_\h-\|u(y(s))\|_\h^2\bigg) ds.
\end{aligned}
\end{equation*}
As we mentioned in Section \ref{epsilon0sec}, the optimization problem on the left-hand side has an explicit solution attained by
\begin{equation}\label{vbar}
\bar{v}(t)=\dot{y}^*(t)-Ay^*(t) -DF(x^*)y^*(t)+u\big(y^*(t)\big), t\in[0,T]
\end{equation}
and from Proposition \ref{kprop}, $T=\inf\{ t>0: \|y^*(t)\|_{\h}=L\}=\tau_{y^*}$. Now consider the processes $\hat{\eta}^{\epsilon, \bar{v}}_{x^*}$
controlled by $\bar{v}$. From Lemmas \ref{tightlem}, \ref{limlem}, $\{\hat{\eta}^{\epsilon, \bar{v}}_{x^*};\epsilon>0      \}$ is tight and converges in distribution to a process $\hat{\eta}^{\bar{v}}_{x^*}.$ From the choice of $\bar{v}$ and uniqueness of solutions it follows that $\hat{\eta}^{\bar{v}}_{x^*}=y^*$ with probability $1$. Moreover, the exit times
$\hat{\tau}^{\epsilon, \bar{v}}_{x^*}$ converge in distribution to a random time $\hat{\tau}^{\bar{v}}$ which is no less than the first exit time of $y^*$ from $\mathring{B}_\h(0,L)$. Since the latter is equal to $T$ it follows that $\hat{\tau}^{\bar{v}}=T$ with probability $1$. Thus
\begin{equation}\label{lower}
\begin{aligned}
\limsup_{\epsilon\to0}-\frac{1}{h^2(\epsilon)}\log Q^{\epsilon}(u^\epsilon)&\leq  \frac{1}{2}\int_{0}^{\hat{\tau}^{\epsilon, \bar{v}}_{x^*}}\|\bar{v}(s)\|^2_\h ds+\limsup_{\epsilon\to0} -\ex\int_{0}^{\hat{\tau}^{\epsilon, \bar{v}}_{x^*}}\|u^\epsilon\big(\hat{\eta}^{\epsilon, \bar{v}}_{x^*}\big )\|^2_\h ds\\&
\leq \frac{1}{2}\int_{0}^{T}\|\bar{v}(s)\|^2_\h ds-\liminf_{\epsilon\to0} \ex\int_{0}^{\hat{\tau}^{\epsilon, \bar{v}}_{x^*}}\|u^\epsilon\big(\hat{\eta}^{\epsilon, \bar{v}}_{x^*}\big )\|^2_\h ds\\&
\leq \frac{1}{2}\int_{0}^{T}\|\bar{v}(s)\|^2_\h ds-\int_{0}^{T}\|u\big(\hat{\eta}^{ \bar{v}}_{x^*}\big )\|^2_\h ds \\&
=  \inf_{v\in \mathcal{C}_{y^*,x^*}}\int_{0}^{\tau_{y^*}}\bigg(\frac{1}{2}\|v(s)\|^2_\h-\|u(y^*(s))\|_\h^2\bigg) ds
\\&
=  \inf_{y\in\mathcal{T}}\inf_{v\in \mathcal{C}_{y,x^*}}\int_{0}^{\tau}\bigg(\frac{1}{2}\|v(s)\|^2_\h-\|u(y(s))\|_\h^2\bigg) ds,
\end{aligned}
\end{equation}
where the second inequality follows from lower semi-continuity. Combining \eqref{upper} and \eqref{lower} allows us to conclude.
\begin{rem}\label{MDPrem} Theorem \ref{MDPthm} is essentially equivalent to an MDP for the family $\{X^\epsilon\}_{\epsilon}$ of solutions of \eqref{evoeq}, in the space $C([0,T];\mathcal{E}).$ The latter is an asympotic statement for exponential functionals of $g(X^\epsilon),$ where $g: C([0,T];\mathcal{E})\rightarrow \R$ is continuous and bounded (see Definition \ref{MDPdef}), while the former covers exit probabilities and corresponds to the choice $g=\tilde{g}$ with
	\begin{equation*}
	\tilde{g}(\eta)=\begin{cases}
	&0, \;\;\eta: \sup_{t\in[0,T]}\|\eta(t)\|_\mathcal{\h}\geq L\\&
	\infty, \;\; \eta: \sup_{t\in[0,T]}\|\eta(t)\|_\mathcal{\h}< L.
	\end{cases}
	\end{equation*}
The case for bounded continuous test functions is in fact simpler, does not require analysis of the limiting variational problem and can be proved using very similar arguments to the ones used above. To be precise, for any continuous, bounded 	$g: C([0,T];\mathcal{E})\rightarrow \R$ the variational representation \eqref{varrep} takes the form
\begin{equation*}
-\frac{1}{h^2(\epsilon)}\log\;\ex\big[ e^{-h^2(\epsilon)g(\eta^\epsilon)}\big]=\inf_{v\in\mathcal{A}}\ex\bigg[ \frac{1}{2} \int_{0}^{T} \|v(t) \|^2_\h dt+ g\big(\eta^{\epsilon,v}  \big)  \bigg],
\end{equation*}
according to the classical results of \cite{budhiraja2008large}. The controlled process $\eta^{\epsilon,v}$ solves \eqref{controleq} with $u=0$ and $\mathcal{A}$ is a collection of square-integrable adapted controls. The tightness and limiting statements of Lemmas \ref{tightlem}, \ref{limlem} carry over verbatim after setting $u=0$ and \eqref{MDPLP} then follows with the same action functional \eqref{MDPaction} by proving an upper and a lower bound as above. In particular, the upper bound is a consequence of lower-semicontinuity and the lower bound follows by considering the minimizing control $\bar{v}$ in \eqref{vbar}. In fact, this simpler MDP is used to obtain Lemma \ref{uniboundlem} above, which is important for the case of unbounded functionals that we consider here.
\end{rem}

\subsection{Proof of Theorem \ref{Asymptoticthm}}  Let  $\{v^\epsilon\}\subset\mathcal{A}$ satisfy \eqref{approxmin}. From Lemma \ref{limlem}, Theorem \ref{MDPthm} and the lower semi-continuity argument in \eqref{upper} we know that the triples $(\hat{\eta}^{\epsilon,v^\epsilon}_{x^*}, v^{\epsilon}, \hat{\tau}^{\epsilon,v^\epsilon}_{x^*})$ converge in distribution to a triple  $(\hat{\eta}^{v^0}_{x^*}, v^{0}, \hat{\tau}^{v^0}_{x^*} )$ and
\begin{equation*}
\begin{aligned}
\inf_{y\in\mathcal{T}}\inf_{v\in \mathcal{C}_{y,x^*}}\int_{0}^{\tau}\bigg(\frac{1}{2}\|v(s)\|^2_\h&-\|u(y(s))\|_\h^2\bigg) ds=\limsup_{\epsilon\to0}-\frac{1}{h^2(\epsilon)}\log Q^{\epsilon}(u^\epsilon)\\&
\geq \limsup_{\epsilon\to0}\ex\bigg[\frac{1}{2}\int_{0}^{\hat{\tau}^{\epsilon,v^\epsilon}_{x^*}}\|v^\epsilon(s)\|^2_\h ds-\int_{0}^{\hat{\tau}^{\epsilon,v^\epsilon}_{x^*}}\|u^\epsilon\big(s,\hat{\eta}^{\epsilon,v^\epsilon}_{x^*}(s)\big )\|^2_\h ds\bigg]\\&
\geq \liminf_{\epsilon\to0}\ex\bigg[\frac{1}{2}\int_{0}^{\hat{\tau}^{\epsilon,v^\epsilon}_{x^*}}\|v^\epsilon(s)\|^2_\h ds-\int_{0}^{\hat{\tau}^{\epsilon,v^\epsilon}_{x^*}}\|u^\epsilon\big(s,\hat{\eta}^{\epsilon,v^\epsilon}_{x^*}(s)\big )\|^2_\h ds\bigg]\\&
=\ex\bigg[\frac{1}{2}\int_{0}^{\hat{\tau}^{v^0}_{x^*}}\|v^0(s)\|^2_\h ds-\int_{0}^{\hat{\tau}^{v^0}_{x^*}}\|u\big(\hat{\eta}^{v^0}_{x^*}(s)\big )\|^2_\h ds\bigg].
\end{aligned}
\end{equation*}

\noindent Invoking Lemma \ref{limlem} once again we have $\hat{\eta}^{v^0}_{x^*}\in\mathcal{T}$ and $v^0\in\mathcal{C}_{\hat{\eta}^{v^0}_{x^*},x^*}$ with probability $1$. Since the left-hand side is the infimum over all such paths and controls it follows that
\begin{equation*}
\begin{aligned}
\frac{1}{2}\int_{0}^{\hat{\tau}^{v^0}_{x^*}}\|v^0(s)\|^2_\h ds-\int_{0}^{\hat{\tau}^{v^0}_{x^*}}\|u\big(\hat{\eta}^{v^0}_{x^*}(s)\big )\|^2_\h ds = \inf_{y\in\mathcal{T}}\inf_{v\in \mathcal{C}_{y,x^*}}\int_{0}^{\tau}\bigg(\frac{1}{2}\|v(s)\|^2_\h&-\|u(y(s))\|_\h^2\bigg) ds
\end{aligned}
\end{equation*}
with probability $1$. Thus, from Proposition \ref{kprop} we can conclude that
 $\hat{\tau}^{\epsilon,v^\epsilon}_{x^*}\rightarrow T$ in probability as $\epsilon\to 0$,
 $\blangle   \hat{\eta}^{v^0}_{x^*}(T), e_1^f    \brangle^2_\h=L^2$ with probability $1$ and \eqref{asymptoticexitplace} follows.

 It remains to prove \eqref{varbounds}. We start from the upper bound which is a consequence of Lemma \ref{Gupperlem}, provided that $E=\{  \phi\in C([0,T];\h):  \tau_{\phi}\leq T\}$ is a $\mathcal{S}_{x^*,T}-$continuity set. This property can be verified from the analysis of Section \ref{epsilon0sec}. In particular, Lemmas \ref{Eulalem}, \ref{exitdirectionlem} and Proposition \ref{kprop} remain true after setting the second summand in \eqref{Ifinalform} or \eqref{Ikform} equal to $0$. Hence the infima of the action functional over $\{\tau_{\phi}\leq T\}, \{\tau_{\phi}<T\} $ and $\{\tau_{\phi}=T\}$ coincide and the estimate follows. As for the lower bound, we combine Theorem \ref{MDPthm}, \eqref{minvalue} and \eqref{subsoleq} to obtain
 $$G_T(0,0)=\frac{a_1^fL^2}{1-e^{-2a^{f}_1T}}\geq a_1^fL^2=U(0,0)\;,\;\;\lim_{T\to\infty}G_T(0,0)=U(0,0)$$ and
 \begin{equation*}
     \lim_{\epsilon\to 0}-\frac{1}{h^2(\epsilon)}\log Q^{\epsilon}(u^\epsilon)= a_1^fL^2\bigg(1+\frac{1}{1-e^{-2a^{f}_1T}}\bigg)=U(0,0)+G_T(0,0).
 \end{equation*}
The latter shows that the lower bound actually holds with equality, hence the proof is complete.

   \section{Implementation and pre-asymptotic analysis of the scheme}\label{Sec4}

   \subsection{Implementation issues and exponential mollification}

    In Section \ref{Sec3}, we demonstrated that, under fairly general spectral gap conditions, an importance sampling scheme using the change of measure $u_{k_0}$ \eqref{uchoice} achieves nearly optimal asymptotic behavior as the noise intensity $\epsilon\to0$. However, changes of measure based only on the quasipotential subsolution $U$ \eqref{subsoleq} can lead to poor pre-asymptotic performance. This issue is present even in finite dimensions and is related to the behavior of the controlled dynamics near the origin. In \cite{DupuisSpilZhou}, the authors demonstrated that, for certain choices of controls $v$, the second moment of the estimator degrades over time. 
   In these situations, the system tends to spend a large amount of time near the attractor thus accumulating a large running cost which affects the variance. As a result, for fixed $\epsilon>0$  the pre-exponential terms which are ignored by the asymptotic bounds \eqref{varbounds} dominate  and can even lead to errors that increase exponentially as $T$ grows. For more details the reader is referred to the discussion in \cite{DupuisSpilZhou} pp.2919-2921.

   In infinite dimensions, an additional challenge appears when the changes of measure act on the full space $\h.$ As we will see in Lemma \ref{preasymptoticlem} below, in order to prove that the second moment of a scheme behaves well for any fixed $\epsilon>0,$ one needs to have good control over the quantity
   \begin{equation}\label{Ddef1}
   \begin{aligned}
   \mathscr{D}_{x^*}^{\epsilon}(Z_{x^*})(t,\eta)&:=\partial_t Z_{x^*}(t,\eta)+\mathbb{H}_{x^*}\big(\eta,D_\eta Z_{x^*}(t,\eta)\big)+\frac{1}{2h^2(\epsilon)}\text{tr}\big[ D^2_\eta Z_{x^*}(t,\eta)\big],
   \end{aligned}
   \end{equation}
   where $Z_{x^*}$ denotes a subsolution used for the analysis of the scheme. However, any radial function $Z:\h\rightarrow\R$ such that $Z(\eta)=\bar{Z}(\|\eta\|_\h)$, with $\bar{Z}''<0,$ satisfies $\text{tr}\big[ D^2_\eta Z(\eta)\big]=-\infty$. Thus, apart from dealing with the difficulties related to unbounded operators (see Remark \ref{infdimrem}), changes of measure for SRDEs that effectively accomplish dimension reduction are necessary for provably efficient performance.

   In this section we construct a scheme under Hypothesis \ref{A3c}, i.e. our changes of measure only force the $e_1^f$ direction. From this point on it is understood that $u\equiv u_1$ and $u^\epsilon_{1}\equiv u^\epsilon.$ In order to deal with the aforementioned issues,  our changes of measure $u^\epsilon$ will meet the following criteria: 1) The projected-quasipotential subsolution (denoted below by $F_1$) will be used for regions of space that are sufficiently far from the origin.  2) A constant subsolution $F^{\epsilon}_2$  will dominate near zero. $F^{\epsilon}_2$ does not influence the dynamics until they enter the domain where $F_1$ dominates. 3) To avoid issues from lack of smoothness, the combination of $F_1,F_2$ should be appropriately mollified. 4) As $\epsilon\to 0$ the changes of measure $u^\epsilon$ converge to the asymptotically nearly optimal $u$. A suitable choice is provided by the exponential mollification of $F_1,F_2^\epsilon$.

   To be precise, we define for $a_1^f,e_1^f$ as in Hypothesis \ref{A3c}, $\kappa\in(0,1)$ and  $\delta=\delta(\epsilon)>0$

  \begin{equation*}\label{F12}
  F_1(\eta):=a_1^f(L^2-\langle \eta, e^f_1\rangle^2_\h),\; F^\epsilon_2:=a_1^f(L^2-h(\epsilon)^{-2\kappa}), \eta\in\h
  \end{equation*}
    and consider the exponential mollification
   \begin{equation}\label{Udef}
   U^\delta(t,\eta):=-\delta\log\bigg( e^{-\frac{F_1(\eta)}{\delta}} +    e^{-\frac{F^\epsilon_2}{\delta}}  \bigg).
   \end{equation}
 We implement our scheme using the change of measure
 \begin{equation}
 \label{uepsilonchoice}
 u^\epsilon(t,\eta):=-D_\eta U^\delta(t,\eta)=-2a_1^f\rho^\epsilon(\eta)\langle e_1^f, \eta\rangle_\h e_1^f,
 \end{equation}
   where
 \begin{equation*}\label{rho}
 \rho^\epsilon(\eta):=\frac{ e^{-\frac{F_1(\eta)}{\delta}}}{  e^{-\frac{F_1(\eta)}{\delta}} +    e^{-\frac{F^\epsilon_2}{\delta}}},
 \end{equation*}
 $\delta=2/h^2(\epsilon)$ is the mollification parameter and $\kappa$ is a parameter that controls the size of the neighborhood outside of which $F_1$ dominates.

       In order to derive non-asymptotic bounds for the second moment of the estimator, we will use the following min/max representation for the Hamiltonian
      \begin{equation*}
        \mathbb{H}_{x^*}(\eta, p)=\inf_{v}\sup_{u}\bigg[
        \langle p, A\eta+DF(x^*)\eta-u+v\rangle_\h-\frac{1}{2}\|u\|^2_\h+\frac{1}{4}\|v\|_\h^2 \bigg]
      \end{equation*}
      (see e.g. \cite{DupuisSpilZhou, dupuis2004importance}) and for any smooth functions $U_{x^*},Z_{x^*}:[0,T]\times\h\rightarrow\R$ we let $u(t,\eta)=-D_\eta U_{x^*}(t,\eta)$ and $p=D_\eta Z_{x^*}(t,\eta)$. Thus we obtain
      \begin{equation}\label{minmaxHam}
      \begin{aligned}
      \inf_{v}\bigg[\blangle D_\eta Z_{x^*}(t,\eta), A\eta&+DF(x^*)\eta-u(t,\eta)+v\brangle_\h-\frac{1}{2}\|u(t,\eta)\|^2_{\h}+\frac{1}{4}\|v\|^2_\h \bigg]\\&
      =\mathbb{H}_{x^*}\big(\eta, D_\eta Z_{x^*}(t,\eta) \big)-\frac{1}{2}\big\|  D_\eta Z_{x^*}(t,\eta)- D_\eta U_{x^*}(t,\eta) \big\|^2_\h.
      \end{aligned}
      \end{equation}
      A consequence of this expression is the following pre-asymptotic bound for the second moment:
      \begin{lem}\label{preasymptoticlem} For any smooth functions  $U_{x^*},Z_{x^*}:[0,T]\times\h\rightarrow\R,$ $\mathscr{D}_{x^*}$ as in \eqref{Ddef1} and some $\theta_0\in(0,1)$ let
      		\begin{equation}\label{Hdef}
      	\begin{aligned}
      	\mathscr{H}_{x^*}^{\epsilon}(Z_{x^*})(t,\eta)&:=\frac{\sqrt{\epsilon}h(\epsilon)}{2}\blangle D_\eta Z_{x^*}\big(t,\eta\big), D^2F\big(x^*+\theta_0\sqrt{\epsilon}h(\epsilon)\eta \big)\big(\eta,\eta\big) \brangle_\h
      	\end{aligned}
      	\end{equation}
      	and
      	\begin{equation}\label{frakhdef}
      	\begin{aligned}
      	\mathfrak{H}_{x^*}^{\epsilon}( U_{x^*}, Z_{x^*})(t,\eta)&:=\mathscr{H}_{x^*}^{\epsilon}( Z_{x^*})(t,\eta)+\mathscr{D}_{x^*}^{\epsilon}(Z_{x^*})(t,\eta)-\frac{1}{2}\|D_\eta Z_{x^*}(t,\eta)- D_\eta U_{x^*}(t,\eta)\|^2_\h.
      	\end{aligned}
      	\end{equation}
      	 For all $\epsilon>0$ we have
      	\begin{equation}\label{preasymptotic}
      	-\frac{1}{h^2(\epsilon)}\log Q^{\epsilon}(u^\epsilon)\geq\inf_{v\in\mathcal{A}}\bigg[  2Z_{x^*}\big(0,0\big)- 2\ex Z_{x^*}\big(\hat{\tau}^{\epsilon,v}_{x^*},\hat{\eta}_{x^*}^{\epsilon, v}(\hat{\tau}^{\epsilon,v}_{x^*})\big)+2\ex\int_{0}^{\hat{\tau}^{\epsilon,v}_{x^*}}\mathfrak{H}_{x^*}^{\epsilon}( U_{x^*}, Z_{x^*})\big(s,\hat{\eta}^{\epsilon,v}_{x^*}(s)\big) ds    \bigg].
      	\end{equation}
      \end{lem}
     \noindent The proof makes use of It\^o's formula and is deferred to Appendix \ref{AppA}.
     \begin{rem}
     	The term $\mathscr{H}_{x^*}^{\epsilon}$ accounts for the error coming from the local approximation of the nonlinear dynamics by their linearized version around the stable equilibrium $x^*.$ A significant part of this section is devoted to the pre-asymptotic control of this term.
     \end{rem}
  The rest of this section is devoted to the pre-asymptotic analysis of $Q^{\epsilon}(u^\epsilon)$ based on the lower bound \eqref{preasymptotic}  with $U_{x^*}=U^\delta(t,\eta),$
   \begin{equation}\label{Zdef}
  Z(t,\eta)=Z_{x^*}(t,\eta)=(1-\zeta)U^\delta(t,\eta), \zeta\in(0,1).
  \end{equation}

       \subsection{Performance analysis of the scheme} At this point we shall recall the definition of the random times $$\hat{\tau}^{\epsilon,v}_{x^*}=\inf\{ t>0 : \hat{\eta}^{\epsilon,v}_{x^*}(t)\notin \mathring{B}_\h(0, L)     \}$$  where $\hat{\eta}^{\epsilon,v}_{x^*}$ solves \eqref{controleq}. Before we state the main result of this section, we provide the definition of exponential negligibility; a concept which will be frequently used in the sequel.

         \begin{dfn}
       	A term will be called exponentially negligible (a) in the moderate deviations range if it can be bounded from above in absolute value by $C_1e^{-	c_2h^2(\epsilon)}/h^2(\epsilon)$ where $C_1<\infty, c_2>0$ (b) in the large deviations range if (a) holds with $1/h^2(\epsilon)$ replaced by $\epsilon$.
       \end{dfn}
       \begin{rem} Since $\sqrt{\epsilon}h(\epsilon)\to 0$ as $\epsilon\to 0,$ exponential negligibility in the large deviations range implies exponential negligibility in the moderate deviations range.
       \end{rem}	
       \noindent The analysis of this section is summarized in the following theorem. Its proof is postponed for the end of this section and is preceded by several auxiliary estimates.

       \begin{thm}\label{preasymptoticthm} Let $T,\alpha,\zeta_0,\epsilon>0$ and $u^\epsilon(t,\eta)=-D_\eta U^{\delta(\epsilon)}(t,\eta)$ with $U^\delta$ defined in \eqref{Udef}. Assume that $\delta=2/h^2(\epsilon), \kappa\in(0,1-\alpha), \zeta\in(\zeta_0,1/2)$ and $\epsilon$ is sufficiently small to have  $h^{2(\kappa+\alpha-1)}(\epsilon)\leq \frac{9a_1^f}{2}(\zeta_0-2\zeta_0^2)\wedge \frac{a^f_1}{2}.$ Then, up to exponentially negligible terms in the moderate deviations range,
       	\begin{equation}\label{preasymptoticfin1}
       	-\frac{1}{h^2(\epsilon)}\log Q^{\epsilon}(u^\epsilon)\geq\bigg[ (1-\zeta)a_1^f\bigg(L^2-\frac{1}{h(\epsilon)^{2\kappa}}\bigg)-\frac{2\log 2}{h^2(\epsilon)}\bigg]-CT\sqrt{\epsilon}h(\epsilon).
       	\end{equation}
       	Moreover, if $h(\epsilon)$ is such that $\sqrt{\epsilon}h^3(\epsilon)\longrightarrow0$ as $\epsilon\to0$ then for $\epsilon$ sufficiently small we have
       	\begin{equation}\label{preasymptoticfin2}
       	-\frac{1}{h^2(\epsilon)}\log Q^{\epsilon}(u^\epsilon)\geq\bigg[ (1-\zeta)a_1^f\bigg(L^2-\frac{1}{h(\epsilon)^{2\kappa}}\bigg)-\frac{2\log 2}{h^2(\epsilon)}\bigg].
       	\end{equation}    	
       \end{thm}

   \begin{rem} Note that for a small fixed  $\epsilon,$ \eqref{preasymptoticfin1} shows that, in theory, the second moment degrades as the sampling time $T$ grows. This degradation is caused by the linearization error \eqref{Hdef} and suggests that, in practice, good performance lies in the balance between $\epsilon$ and $T$. Fortunately, \eqref{preasymptoticfin2} shows that this theoretical degradation is no longer present if the scaling $h(\epsilon)$ does not grow too fast. Moreover, the simulation studies of Section \ref{Sec6} show that our scheme performs well for large $T$ even when this growth assumption is not satisfied.
   	\end{rem}

   \noindent  The following lemma collects a few straightforward computations that will be used below. Its proof can be found in Appendix \ref{AppA}.

     \begin{lem}\label{mainestimate1lem} For all $(t,\eta)\in[0,T]\times\h,\zeta\in(0,1),$ $U^\delta, Z$ as in \eqref{Udef}, \eqref{Zdef} and some $\theta_0\in(0,1)$ we have
   	\begin{equation}\label{mainestimate1}
   	\begin{aligned}
   	\mathfrak{H}_{x^*}^{\epsilon}( U^\delta, Z)(t,\eta)&=2(1-\zeta)(a_1^f)^2\rho^\epsilon(\eta)\langle e_1^f, \eta\rangle^2_\h\big[1-(1-\zeta)\rho^\epsilon(\eta)      \big]-2\zeta^2(a_1^f)^2\big(\rho^\epsilon(\eta)\big)^2\langle e_1^f, \eta\rangle^2_\h\\&
   	-\frac{(1-\zeta)a_1^f\rho^\epsilon(\eta)}{h^2(\epsilon)}\bigg[1+\frac{2}{\delta}\rho^\epsilon(\eta)\big(1-\rho^\epsilon(\eta)\big)\langle e_1^f, \eta\rangle_\h \bigg]\\&
   	-\sqrt{\epsilon}h(\epsilon)2(1-\zeta)a_1^f\rho^\epsilon(\eta)\langle e_1^f, \eta\rangle_\h \langle e_1^f, D^2F\big(x^*+\theta_0\sqrt{\epsilon}h(\epsilon)\eta \big)\big(\eta,\eta\big)\rangle_\h.
   	\end{aligned}
   	\end{equation}
   \end{lem}

\noindent Moving on to the main body of the analysis, let $B_{\infty}(x^*,1)$ denote an open $L^\infty-$ball of radius $1$ centered at $x^*$  and
         \begin{equation*}
         \tau_{\infty}^{\epsilon}=\inf\big\{t>0:  \|\hat{\eta}_{x^*}^{\epsilon,v}(t) \|_{L^\infty}\geq \tfrac{1}{\sqrt{\epsilon}h(\epsilon)}   \big\}=\inf\big\{t>0:  \hat{X}_{x^*}^{\epsilon,v}(t)\in B_{\infty}(x^*,1)^c  \big\}.
         \end{equation*}
         Returning to \eqref{preasymptotic} we have the following decomposition
         \begin{equation}\label{preasymptoticnew}
         \begin{aligned}
         -\frac{1}{h^2(\epsilon)}\log &Q^{\epsilon}(u^\epsilon)\geq\inf_{v\in\mathcal{A}}\bigg[  2Z_{x^*}\big(0,0\big)- 2\ex Z_{x^*}\big(\hat{\tau}^{\epsilon,v}_{x^*},\hat{\eta}_{x^*}^{\epsilon, v}(\hat{\tau}^{\epsilon,v}_{x^*})\big)+2\ex\int_{0}^{\hat{\tau}^{\epsilon,v}_{x^*}}\mathfrak{H}_{x^*}^{\epsilon}( U_{x^*}, Z_{x^*})\big(s,\hat{\eta}^{\epsilon,v}_{x^*}(s)\big) ds \bigg]\\&
         =\inf_{v\in\mathcal{A}}\bigg[  2Z_{x^*}\big(0,0\big)- 2\ex Z_{x^*}\big(\hat{\tau}^{\epsilon,v}_{x^*},\hat{\eta}_{x^*}^{\epsilon, v}(\hat{\tau}^{\epsilon,v}_{x^*})\big)
         +2\ex\int_{0}^{\hat{\tau}^{\epsilon,v}_{x^*}\wedge\tau_{\infty}^{\epsilon}}\mathfrak{H}_{x^*}^{\epsilon}( U_{x^*}, Z_{x^*})\big(s,\hat{\eta}^{\epsilon,v}_{x^*}(s)\big) ds\\&\quad\quad
         +2\ex\int_{\hat{\tau}^{\epsilon,v}_{x^*}\wedge\tau_{\infty}^{\epsilon}}^{\hat{\tau}^{\epsilon,v}_{x^*}}\mathfrak{H}_{x^*}^{\epsilon}( U_{x^*}, Z_{x^*})\big(s,\hat{\eta}^{\epsilon,v}_{x^*}(s)\big) ds  \bigg].
         \end{aligned}
         \end{equation}

         \begin{rem}
         	This decomposition allows us to deal with the cubic power of $\eta$ that appears in \eqref{Hdef}. Since we are only controlling the spatial $L^2-$norm of the moderate deviation process, this term is problematic. In particular, estimates based in the a-priori bound \eqref{etasup} will introduce $T$-dependent constants which are not desirable for the pre-asymptotic analysis.
         \end{rem}

     \noindent  The last term in \eqref{preasymptoticnew} concerns the behavior of the controlled process $\hat{\eta}^{\epsilon,v}$ in the event that it exits an $L^\infty-$ball of radius $1/\sqrt{\epsilon}h(\epsilon)$ before it exits $\mathring{B}_{\h}(0,L)$. Since the latter is a very rare event in the moderate deviations range, we expect that this term is exponentially negligible.  This claim is proved in the following proposition.
         \begin{prop}\label{linearizationLDPprop}
         	  The term
         	\begin{equation*}
         	2\ex\int_{\hat{\tau}^{\epsilon,v}_{x^*}\wedge\tau_{\infty}^{\epsilon}}^{\hat{\tau}^{\epsilon,v}_{x^*}}\mathfrak{H}_{x^*}^{\epsilon}( U^\delta, Z)\big(s,\hat{\eta}^{\epsilon,v}_{x^*}(s)\big) ds
         	\end{equation*}
         	is exponentially negligible in the moderate deviations range for $\epsilon$ sufficiently small.
         \end{prop}
          \begin{proof}
          	 Let $t\in[0,T], \eta\in L^\infty\cap B_\h(0,L)$, $\epsilon$ small enough to have $\sqrt{\epsilon}h(\epsilon)<1.$ In view of \eqref{fgrowth},
          	\begin{equation*}
          	\begin{aligned}
          	\frac{1}{\sqrt{\epsilon}h(\epsilon)}\big|\mathscr{H}_{x^*}^{\epsilon}( Z)\big(t,\eta \big)\big|&\leq 2(1-\zeta)a_1^f\rho^\epsilon(\eta)\big|\langle e_1^f, \eta\rangle_\h\big|\big|     \langle D^2F(x^*+\theta\sqrt{\epsilon}h(\epsilon)\eta)(\eta,\eta),e_1^f \rangle_\h      \big|\\&
          	\leq 2a_1^f\|\eta\|^3_\h\|e_1^f\|_{\h}\big\|\partial^2_xf\big(  x^*+\theta\sqrt{\epsilon}h(\epsilon)\eta    \big)e_1^f\big\|_{L^\infty}\\&
          	\leq 2C_{f,p}a_1^fL^3\|e_1^f\|_{L^\infty}\big(1+\|  x^* \|^{p_0-2}_{L^\infty}+\|\eta \|^{p_0-2}_{L^\infty}\big).
          	\end{aligned}
          	\end{equation*}
          Moreover, from \eqref{mainestimate1} we have
          	
          		\begin{equation*}
          	\begin{aligned}
          	\big|\mathfrak{H}_{x^*}^{\epsilon}( U^\delta, Z)\big(t,\eta \big)-\mathscr{H}_{x^*}^{\epsilon}(Z)\big(t,\eta \big)\big|&\leq
          		2(1-\zeta)(a_1^f)^2\rho^\epsilon(\eta)\langle e_1^f, \eta\rangle^2_\h\big[1-(1-\zeta)\rho^\epsilon(\eta)      \big]+2\zeta^2(a_1^f)^2\big(\rho^\epsilon(\eta)\big)^2\langle e_1^f, \eta\rangle^2_\h\\&+
          	 \frac{(1-\zeta)a_1^f\rho^\epsilon(\eta)}{h^2(\epsilon)}\bigg[1+\frac{2}{\delta}\rho^\epsilon(\eta)\big(1-\rho^\epsilon(\eta)\big)\langle e_1^f, \eta\rangle_\h \bigg]\\&\leq
          	 4(a_1^f)^2\|e_1^f\|^2_{\h}\|\eta\|^2_{\h}+ \frac{a_1^f}{h^2(\epsilon)}\bigg[1+h^2(\epsilon)\|e_1^f\|_{\h}\|\eta\|_{\h} \bigg]\\&
          	 \leq C_{\ell,f,L},
          	\end{aligned}
          	\end{equation*}
          	where we used that $\zeta,\rho\in(0,1),$ and $h(\epsilon)>1$. Combining the last two estimates we deduce that for any $v\in\mathcal{A},$
          	\begin{equation*}
          	\begin{aligned}
          	\bigg|\int_{\hat{\tau}^{\epsilon,v}_{x^*}\wedge\tau_{\infty}^{\epsilon}}^{\hat{\tau}^{\epsilon,v}_{x^*}}\mathfrak{H}_{x^*}^{\epsilon}( U^\delta, Z)ds\bigg|&\leq     	\mathds{1}_{\{\tau_{\infty}^{\epsilon}\leq\hat{\tau}^{\epsilon,v}_{x^*} \}}\int_{\tau_{\infty}^{\epsilon}}^{\hat{\tau}^{\epsilon,v}_{x^*}}\big|\mathfrak{H}_{x^*}^{\epsilon}( U^\delta, Z)\big(s,\hat{\eta}^{\epsilon,v}_{x^*}(s)\big)\big|ds\\&
          	\leq C_{f,p_0,L,\ell}\mathds{1}_{\{\tau_{\infty}^{\epsilon}\leq\hat{\tau}^{\epsilon,v}_{x^*} \}}\int_{0}^{T}\big(1+\|  x^* \|^{p_0-2}_{L^\infty}+\big\|\hat{\eta}^{\epsilon,v}_{x^*}(s) \big\|^{p_0-2}_{L^\infty}\big)ds\\&
          	\leq C_{f,p_0,L,T}\mathds{1}_{\{\tau_{\infty}^{\epsilon}\leq T \}}\bigg(  1+\|  x^* \|^{p_0-2}_{L^\infty}+\sup_{s\in[0,T]}\big\|\hat{\eta}^{\epsilon,v}_{x^*}(s) \big\|^{p_0-2}_{L^\infty} \bigg).
          	\end{aligned}
          	\end{equation*}
          	An application of H\"older's inequality along with \eqref{etapriori} yields
          	\begin{equation*}
          	\begin{aligned}
          	\bigg|\ex\int_{\hat{\tau}^{\epsilon,v}_{x^*}\wedge\tau_{\infty}^{\epsilon}}^{\hat{\tau}^{\epsilon,v}_{x^*}}\mathfrak{H}_{x^*}^{\epsilon}( U^\delta, Z)\big(s,\hat{\eta}^{\epsilon,v}_{x^*}(s)\big) ds\bigg|&\leq C_{f,p_0,L,\ell,T}\pr[\tau_{\infty}^{\epsilon}\leq T ]^{\frac{1}{2}}\bigg(  1+\|  x^* \|^{p_0-2}_{L^\infty}+\ex\bigg[\sup_{s\in[0,T]}\big\|\hat{\eta}^{\epsilon,v}_{x^*}(s) \big\|^{2p_0-4}_{L^\infty}\bigg]^{\frac{1}{2}} \bigg)\\&\leq
          	C_{f,p,L,T,x^*}\pr[\tau_{\infty}^{\epsilon}\leq T]^{\frac{1}{2}}.
          	\end{aligned}
          	\end{equation*}
          	
          	\noindent  Recall now that $\hat{X}_{x^*}^{\epsilon, v}$ solves
          	\begin{equation*}
          	\big\{ d\hat{X}^\epsilon(t)=[A     \hat{X}^\epsilon(t)+F\big(     \hat{X}^\epsilon(t)\big)]dt+\sqrt{\epsilon}h(\epsilon)\big[v(t)- u^\epsilon\big(\hat{\eta}^{\epsilon,v}_{x^*}(t)\big )]dt+\sqrt{\epsilon}dW(t)\;\;,\hat{X}^\epsilon(0)= x^*\big\}
          	\end{equation*}
          	and, as $\epsilon\to0$, $\{\hat{X}_{x^*}^{\epsilon, v}\}_{\epsilon>0}$ satisfies a large deviation principle in $C([0,T];L^\infty(0,\ell))$ with action functional
          	$\widetilde{\mathcal{S}}_{x^*,T}: C([0,T];L^\infty(0,\ell))\rightarrow [0,\infty] $ given by
          	\begin{equation*}
          	\widetilde{\mathcal{S}}_{x^*,T}(\phi)=\inf_{u\in\mathcal{P}_\phi}\frac{1}{2}\int_{0}^{T}\big\| u(t)\big\|^2_{\h}dt,\;\; \mathcal{P}_\phi=\bigg\{  u\in L^2([0,T];\h):\forall t\in[0,T]\; \phi(t)=S(t)x^*+\int_{0}^{t}S(t-s)\big[F\big(     \phi(s)\big)+ u(s)     \big]   ds   \bigg\},
          	\end{equation*}
          	where the convention $\inf\varnothing=+\infty$ is in use (see e.g. \cite{cerrai2004large}, Theorems 6.2, 6.3). Passing to a convergent subsequence if necessary, we deduce that 	
          	\begin{equation*}
          	\lim_{\epsilon\to0}\epsilon\log\pr[\tau_{\infty}^{\epsilon}\leq T  ]\leq -\inf_{\phi\in \mathcal{B}_{\infty}(x^*,L)^c}\widetilde{\mathcal{S}}_{x^*,T}(\phi),
          	\end{equation*}
          	where $\mathcal{B}_{\infty}(x^*,L):=\{\phi\in C([0,T];\h): \sup_{t\in[0,T]}\|\phi(t)-x^*\|_{L^\infty}<1\}.$
          	Hence, for $\epsilon$ sufficiently small
          	\begin{equation*}\label{ldpnegligibility}
          	\epsilon\log\pr[\tau_{\infty}^{\epsilon}\leq T  ]\leq -\inf_{\phi\in\mathcal{B}_{\infty}(x^*,L)^c}\widetilde{\mathcal{S}}_{x^*,T}(\phi)/2
       	\end{equation*}
          	or equivalently
          	\begin{equation*}\pr[\tau_{\infty}^{\epsilon}\leq T ]^{\frac{1}{2}}\leq e^{-\inf_{\phi\in \mathcal{B}_{\infty}(x^*,L)^c}\widetilde{\mathcal{S}}_{x^*,T}(\phi)/4\epsilon}.
          	\end{equation*}
          	Finally, we claim that $\inf_{\phi\in \mathcal{B}_{\infty}(x^*,L)^c}\widetilde{\mathcal{S}}_{x^*,T}(\phi)>0$. Indeed, since the action functional is lower semi-continuous (see Lemma 5.1, \cite{cerrai2004large}) and $\mathcal{B}_{\infty}(x^*,L)^c\subset C([0,T];L^\infty(0,\ell))$ is closed, there exists a minimizer $\phi^*\in  \mathcal{B}_{\infty}(x^*,L)^c.$ Furthermore, there exists $u^*\in\mathcal{P}_{\phi^*}$ such that
          	\begin{equation*}
          	2\inf_{\phi\in \mathcal{B}_{\infty}(x^*,L)^c}\widetilde{\mathcal{S}}_{x^*,T}(\phi)=2\widetilde{\mathcal{S}}_{x^*,T}(\phi^*)>\frac{1}{2}\int_{0}^{T}\big\| u^*(t)\big\|^2_{\h}dt=\frac{1}{2}\big\|u^*\big\|^2_{L^2([0,T];\h)}>0.
          	\end{equation*}
          	The last inequality fails if and only if $u^*=0$ almost everywhere in $[0,T]\times[0,\ell].$ Since $x^*$ is an equilibrium of the uncontrolled system, the latter implies that $\phi^*(t)=x^*$ for all $t\in[0,T],$ hence  $\phi^*\notin\mathcal{B}_{\infty}(x^*,L)^c.$ This contradicts the initial choice of $\phi^*$ and concludes the argument. 
          	Therefore, the term of interest is exponentially negligible in the large deviation range hence also in the moderate deviation range.
          \end{proof}
           \noindent Next, we turn our attention to the third term in \eqref{preasymptoticnew}. The linearization error in this term is easier to control, since the process $\sqrt{\epsilon}h(\epsilon)\hat{\eta}^{\epsilon,v}$ is uniformly bounded by $1$ in $L^\infty-$norm. This fact is used in the following lemma whose proof can be found in Appendix \ref{linearization2lemproof}.
          \begin{lem}\label{linearization2lem} For all $\eta\in B_{\infty}(0,1/\sqrt{\epsilon}h(\epsilon))$ there exists a constant $C=C_{x^*,\ell,f}>0$ such that for $\epsilon$ sufficiently small we have
          	  \begin{equation}\label{linerrorbnd2}
          	\begin{aligned}
          \mathscr{H}_{x^*}^{\epsilon}(Z)(t,\eta)
         	 \geq -C\sqrt{\epsilon}h(\epsilon)2(1-\zeta)a_1^f\rho^\epsilon(\eta)\big|\langle e_1^f, \eta\rangle_\h\big|.
         	 \end{aligned}
         	 \end{equation}
         	 \end{lem}
        \noindent As for $ \mathfrak{H}_{x^*}^{\epsilon}( U^\delta, Z)(t,\eta)- \mathscr{H}_{x^*}^{\epsilon}(Z_{x^*})(t,\eta),$ straightforward algebra along with the arguments of Lemma 4.2 of \cite{DupuisSpilZhou} yield
        \begin{equation*}
        \begin{aligned}
          \mathfrak{H}_{x^*}^{\epsilon}( U^\delta, Z)(t,\eta)- \mathscr{H}_{x^*}^{\epsilon}(Z)(t,\eta)&\geq  \mathscr{D}_{x^*}^{\epsilon}(Z)(t,\eta)-\frac{\zeta^2}{2}\big\| D_\eta U^\delta(t,\eta)\big\|^2_{\h}\\&
          \geq (1-\zeta)\mathscr{D}_{x^*}^{\epsilon}(Z)(t,\eta)-\frac{\zeta-2\zeta^2}{2}\big\| D_\eta U^\delta(t,\eta)\big\|^2_{\h}\\&
          \geq \frac{1-\zeta}{2}\bigg(1-\frac{1}{h^2(\epsilon)\delta}\bigg)\beta^\epsilon_0(\eta)+(1-\zeta)\rho^\epsilon(\eta)\gamma_1\\&+\frac{\zeta-2\zeta^2}{2}\rho^\epsilon(\eta)^2\big\| D_\eta F_1(\eta)\big\|^2_{\h},
        \end{aligned}
        \end{equation*}
        where the quantity $\beta_0(\eta):=\rho^\epsilon(\eta)\big(1-\rho^\epsilon(\eta)\big)\big\| D_\eta F_1(\eta)\big\|^2_{\h}$
        is nonnegative since $\rho^\epsilon\in [0,1]$ and $\gamma_1:=\mathscr{D}_{x^*}(F_1)(\eta)=-a_1^f/h^2(\epsilon).$
        Combining the latter with \eqref{linerrorbnd2} and substituting $\delta=2/h^2(\epsilon)$ and $$\big\| D_\eta F_1(\eta)\big\|^2_{\h}=4(a_1^f)^2\langle \eta, e_1^f\rangle^2\big\|e_1^f\|^2_\h=4(a_1^f)^2\langle \eta, e_1^f\rangle^2,$$
        we obtain the lower bound
        \begin{equation}\label{mainestimate3}
        \begin{aligned}
        \mathfrak{H}_{x^*}^{\epsilon}(U^\delta,Z)(t,\eta)&
        \geq (1-\zeta)\rho^\epsilon(\eta)\big(1-\rho^\epsilon(\eta)\big)(a_1^f)^2\langle \eta, e_1^f\rangle^2-\frac{1}{h^2(\epsilon)}(1-\zeta)\rho^\epsilon(\eta)a_1^f\\&+2(\zeta-2\zeta^2)\rho^\epsilon(\eta)^2(a_1^f)^2\langle \eta, e_1^f\rangle^2-2C\sqrt{\epsilon}h(\epsilon)(1-\zeta)a_1^f\rho^\epsilon(\eta)\big|\langle e_1^f, \eta\rangle_\h\big|.
        \end{aligned}
        \end{equation}
      At this point we partition $B_\h(0,L)=B_1^f\cup B_2^f\cup B_3^f,$ where
           \begin{equation}\label{regions}
         \begin{aligned}
         & B^f_1:=\big\{\eta\in\h:\; \langle \eta, e_1^f\rangle_\h^2\leq h(\epsilon)^{-2(\kappa+\alpha)} \big\},\\&
         B^f_2:=\big\{\eta\in\h:\;  h(\epsilon)^{-2(\kappa+\alpha)}<\langle \eta, e_1^f\rangle_\h^2\leq 2h(\epsilon)^{-2\kappa}-h(\epsilon)^{-2}K \big\},\\&
         B^f_3:=\big\{\eta\in\h:\;  2h(\epsilon)^{-2\kappa}-h(\epsilon)^{-2}K <\langle \eta, e_1^f\rangle_\h^2\leq L^2 \big\}.
         \end{aligned}
         \end{equation}
         and the constants $\alpha,\zeta, \kappa\in(0,1), K<0$ will be chosen later. The remaining part of this section is devoted to the study of the right-hand side of \eqref{mainestimate3} on each component separately.

          \begin{lem}\label{region1} Let $\epsilon>0$ small enough to have $\sqrt{\epsilon}h(\epsilon)<1$ and $ \zeta\in(0,1/2)$. For all $\eta\in B_1^f$ \eqref{regions}, $ t\in[0,T]$ we have
          	\begin{equation*}	
          	\mathfrak{H}_{x^*}^{\epsilon}(U^\delta,Z)(t,\eta)\geq 0,
          	 \end{equation*}
          	 up to terms that are exponentially negligible in the moderate deviations range.
          \end{lem}
  \noindent We address the region $B_3^f$ in the following lemma.
        \begin{lem}\label{region3} Let $\kappa\in(0,1),$ $K=-\ln3,\zeta_0>0,$	$\zeta\in[\zeta_0,1/2),$ $\epsilon>0$ small enough to have $\sqrt{\epsilon}h(\epsilon)<1$ and $h^{2(\kappa-1)}(\epsilon)\leq \frac{9a_1^f}{2}(\zeta_0-2\zeta_0^2).$
         	For all $\eta\in B_3^f$ \eqref{regions}, $t\in[0,T]$ we have either
         	\begin{equation*}	
         	(i)\quad\mathfrak{H}_{x^*}^{\epsilon}(U^\delta,Z)(t,\eta)\geq -C\sqrt{\epsilon}h(\epsilon),
         	\end{equation*}
         	or, if $h(\epsilon)$ is such that $\sqrt{\epsilon}h^3(\epsilon)\rightarrow 0,$ then for sufficiently small $\epsilon$ we have
         		\begin{equation*}	
         	(ii)\quad\mathfrak{H}_{x^*}^{\epsilon}(U^\delta,Z)(t,\eta)\geq 0.
         	\end{equation*}
        \end{lem}
        \noindent It remains to study the region $B_2^f.$ It is the most problematic region as there is no guarantee that the weight $\rho^\epsilon$ is exponentially negligible or of order one. The analysis is deferred to Appendix \ref{region2proof}.
        \begin{lem}\label{region2}
        	Let $\alpha\in(0,1), \kappa<1-\alpha, K=-\ln3,$	$\zeta\in(\zeta_0,1/2),$ $\epsilon>0$ small enough to have $\sqrt{\epsilon}h(\epsilon)<1$ and $h^{2(\kappa+\alpha-1)}(\epsilon)\leq \frac{a^f_1}{2}.$
        	For all $\eta\in B_2^f$ \eqref{regions}, $ t\in[0,T]$ we have either
        	\begin{equation*}	
        	(i)\quad\mathfrak{H}_{x^*}^{\epsilon}(U^\delta,Z)(t,\eta)\geq -C\sqrt{\epsilon}h(\epsilon),
        	\end{equation*}
             where  $C$ does not depend on $\epsilon$
        	or, if $\sqrt{\epsilon}h^3(\epsilon)\longrightarrow 0$, there exists $\epsilon$ sufficiently small such that
        		\begin{equation*}	
        	(ii)\quad\mathfrak{H}_{x^*}^{\epsilon}(U^\delta,Z)(t,\eta)\geq 0.
        	\end{equation*}
        \end{lem}
     \noindent Combining the three previous lemmas we arrive at the following regarding the third term in \eqref{preasymptoticnew}
     \begin{lem}\label{term3lem} There exists a constant $C$ independent of $T>0$ such that for $\epsilon$ sufficiently small, 	
     	\begin{equation*}
     	(i)\quad\quad\quad\ex\int_{0}^{\hat{\tau}^{\epsilon,v}_{x^*}\wedge\tau_{\infty}^{\epsilon}}\mathfrak{H}_{x^*}^{\epsilon}(U^\delta,Z)\big(s,\hat{\eta}^{\epsilon,v}_{x^*}(s)\big) ds\geq -CT\sqrt{\epsilon}h(\epsilon).
     	\end{equation*}
     	 up to exponentially negligible terms in the moderate deviations range. Moreover, if $h(\epsilon)$ is such that $\lim_{\epsilon\to 0}\sqrt{\epsilon}h^3(\epsilon)=0$ then, for $\epsilon$ sufficiently small,	
     	\begin{equation*}
     		(ii)\quad\quad\quad\quad\quad\quad\quad\ex\int_{0}^{\hat{\tau}^{\epsilon,v}_{x^*}\wedge\tau_{\infty}^{\epsilon}}\mathfrak{H}_{x^*}^{\epsilon}(U^\delta,Z)\big(s,\hat{\eta}^{\epsilon,v}_{x^*}(s)\big) ds\geq 0.
     	\end{equation*}
     	 up to exponentially negligible terms in the moderate deviations range.
     	  \end{lem}
     	\begin{proof}
     		$(i)$ From Lemmas \ref{region1}, \ref{region2}$(i)$, \ref{region3}$(i)$ we have
     		\begin{equation*}
     		\int_{0}^{\hat{\tau}^{\epsilon,v}_{x^*}\wedge\tau_{\infty}^{\epsilon}}\mathfrak{H}_{x^*}^{\epsilon}(U^\delta,Z)\big(s,\hat{\eta}^{\epsilon,v}_{x^*}(s)\big) ds\geq -C\big(\hat{\tau}^{\epsilon,v}_{x^*}\wedge\tau_{\infty}^{\epsilon}\big)\sqrt{\epsilon}h(\epsilon),
     		\end{equation*}
     		with probability $1$, up to exponentially negligible terms. Since $\hat{\tau}^{\epsilon,v}_{x^*}\wedge\tau_{\infty}^{\epsilon}\leq T$ with probability $1$ and the constant is deterministic, the estimate follows by taking expectation.\\
     		$(ii)$ The estimate follows from Lemmas \ref{region1}, \ref{region2}$(ii)$, \ref{region3}$(ii).$ 	
     	\end{proof}
\noindent We conclude this section with the proof of Theorem \ref{preasymptoticthm}.
     \begin{proof}[Proof of Theorem \ref{preasymptoticthm}]
     	In view of Lemmas \eqref{linearizationLDPprop} and \eqref{term3lem}$(i)$, \eqref{preasymptoticnew}
     	yields
     	\begin{equation*}
     	\begin{aligned}
     	-\frac{1}{h^2(\epsilon)}\log Q^{\epsilon}(u^\epsilon)\geq\inf_{v\in\mathcal{A}}\bigg[&  2Z\big(0,0\big)- 2\ex Z\big(\hat{\tau}^{\epsilon,v}_{x^*},\hat{\eta}_{x^*}^{\epsilon, v}(\hat{\tau}^{\epsilon,v}_{x^*})\big)-CT\sqrt{\epsilon}h(\epsilon)
     	  \bigg].
     	\end{aligned}
     	\end{equation*}
     	up to exponentially negligible terms. In view of \eqref{Zdef}, we have $Z\big(t,\eta\big)=(1-\zeta)U^\delta(t,\eta).$ From Theorem \ref{Asymptoticthm} we have   $\lim_{\epsilon\to 0}\ex[Z\big(\hat{\tau}^{\epsilon,v}_{x^*},\hat{\eta}_{x^*}^{\epsilon, v}(\hat{\tau}^{\epsilon,v}_{x^*})\big)]=0$. Thus for $\epsilon$ sufficiently small we may write
     		\begin{equation*}
     	\begin{aligned}
     	-\frac{1}{h^2(\epsilon)}\log Q^{\epsilon}(u^\epsilon)\geq\inf_{v\in\mathcal{A}}\bigg[ Z\big(0,0\big)-CT\sqrt{\epsilon}h(\epsilon)\bigg].
     	\end{aligned}
     	\end{equation*}
     	As for the first term, since $U^\delta$ is the exponential mollification of two functions, Lemma 4.1 of \cite{DupuisSpilZhou} gives that $$U^\delta(0,0)\geq F_1(0)\wedge F_2^{\epsilon}-\delta\log 2=F_2^{\epsilon}-\delta\log 2.$$
     	Finally, the improved bound \eqref{preasymptoticfin2}   follows by invoking Lemma \eqref{term3lem}$(ii).$   	
     \end{proof}
  \section{The case of a double-well potential}\label{Sec5} In this section we specialize our results to SRDEs in which the differential operator $\mathcal{A}=\Delta$ (i.e. the second derivative operator in one spatial dimension) and the reaction term takes the form $f=-V'_f,$ where $V_f$ is a double-well potential as the one depicted below. This choice is possible in view of Hypotheses \ref{A2a}, \ref{A2b} which allow arbitrary polynomial growth.  Thus, we assume that $V_f$ has two global minima and a local maximum which, for simplicity, is assumed to lie in the origin. Without loss of generality, we take $f'(0)=-V''_f(0)=1$. Such SRDEs arise as scaling limits of particle systems with nearest-neighbor coupling that evolve in the inverted potential $-V_f$ (see e.g. \cite{berglund2019introduction}, Chapter 1) and provide one of the simplest examples of non-trivial dynamical behavior.
   \begin{center}
 	\begin{tikzpicture}
 \begin{axis}[
 axis lines = center,
 xlabel = $x$,
 ylabel = {$y$},
 ytick=\empty,
 yticklabels={,,},
 xtick=\empty,
 xticklabels={,,} 	
 ]
 \addplot [
 domain=-3.1:2.5,
 samples=100,
 color=red,
 ]
 {  -1.4*x^2 + 0.2*x^3 + 0.25*x^4};
 \addlegendentry{$V_f(x)$}

 \end{axis}
 \end{tikzpicture}
 \end{center}
    The deterministic reaction-diffusion equation posed on the interval $(0, \ell)$ has two stable equilibria $x^*_{-}, x^*_{+}$, corresponding to the global minima, and a saddle point $x^*_{0}$, corresponding to the global maximum, that is identically equal to $0$. The equilibria $x^*_{\pm}$ only exist if $\ell>\pi$ for Dirichlet boundary conditions and for all $\ell>0$ for Neumann and periodic conditions. Moreover, every time the interval length $\ell$ crosses the value  $k\pi$ (for Neumann or Dirichlet b.c.) or $2k\pi$ (for periodic b.c.), for some $k\in\N,$ two (resp. one) non-constant saddle points $\pm x^*_{k,\ell}$ (resp. $x^*_{p,k,\ell}$) bifurcate from $x^*_{0}.$ The $k-$th non-constant saddle points feature $k$ kink-antikink pairs in the periodic and Dirichlet cases and $k$ kinks in the Neumann case. The interested reader is referred to \cite{berglund2019introduction}, Section 2.1 and \cite{faris1982large} for the bifurcation analysis of the problem with different boundary conditions.

    For most of the sequel, we specialize the discussion to the potential $$V_f(x)=\frac{x^4}{4}-\frac{x^2}{2}\;,\;\; x\in\R.$$
    The corresponding bistable stochastic dynamics are governed by the (stochastic) Allen-Cahn equation
    \begin{equation}\label{AllenCahn}
      \partial_tX^\epsilon=\partial^2_{\xi}X^{\epsilon}+ X^\epsilon-\big(X^\epsilon\big)^3+\sqrt{\epsilon}\dot{W}.
    \end{equation}
  The noiseless ($\epsilon=0$) equation was proposed in \cite{ALLEN19791085} as a simple model of phase separation of two-component alloy systems. It is also known in the literature as real Ginzburg-Landau \cite{Landau:486430} (due to its connections with the physical superconductivity theory bearing the same name) or Chafe-Infante problem \cite{doi:10.1080/00036817408839081}. Transitions between the stable states $x^*_{\pm}$ that correspond to the absolute minima $\pm1$ are enabled by the stochastic forcing and have been studied as models of quantum tunneling phenomena \cite{faris1982large} and thermally induced magnetization reversal of micromagnets \cite{PhysRevLett.87.270601}. For studies of transition times the interested reader is referred to \cite{berglund2017eyring,berglund2013sharp} and \cite{PhysRevLett.87.270601,maier2003effects} in the mathematical and physical literature respectively.
   \subsection{Stochastic Allen-Cahn with Neumann and periodic boundary conditions}\label{Neumannsec}  The eigenvalues of the Neumann and periodic (negative) Laplacian on the interval $(0,\ell)$ are respectively given by
   \begin{equation}\label{Neuperevalues}
   a^{Ne}_n=\bigg(\frac{n\pi}{\ell}\bigg)^2\;, a^{per}_{\pm n}=\bigg(\frac{2 n\pi}{\ell}\bigg)^2\;,\;\;n=0,1,2,\dots.
   \end{equation}
   For $\xi\in(0,\ell), n=1,2,\dots,$ the corresponding eigenfunctions are
        \begin{equation}\label{Neuperefunctions}e^{Ne}_0(\xi)=\ell^{-1/2},\; e^{Ne}_n(\xi)=\sqrt{\frac{2}{\ell}}\cos\bigg(\frac{n\pi\xi}{\ell}\bigg)
      \;,e^{per}_0(\xi)=\ell^{-1/2}\;, e^{per}_{\pm n}(\xi)=\sqrt{\frac{2}{\ell}}\bigg[\pm\sin\bigg(\frac{ 2n\pi\xi}{\ell}\bigg)+   \cos\bigg(\frac{ 2n\pi\xi}{\ell}\bigg)\bigg].
       \end{equation}
  For both cases, the stable equilibria are the constant functions $x^{*}_{\pm}(\xi)=\pm 1, \xi\in(0,\ell).$ The reaction term is $f(x)=x-x^3$ and the linearized operators $\Delta+DF(x^*_{\pm})$ acting on a function $y$ are given by
   $$[\Delta+DF(x^*_{\pm})]y(t)=y''(t)+\big[(1-3x^2)|_{x=x^*_{\pm}}\big]y(t)=y''(t)-2y(t)\;,\;\;t\in[0,T].$$

 Both cases can be treated simultaneously after indexing the eigenpairs by the natural numbers. In particular, in the Neumann case, the eigenvalues $\{a^f_n\}$ from Hypothesis \ref{A3b} are shifted eigenvalues of the (negative) Laplacian, i.e.
 $$a^f_n=2+\bigg(\frac{(n-1)\pi}{\ell}\bigg)^2,\; n=1,2,\dots$$
 and the sequence of eigenvectors $\{e_n^f\}$ coincides with $\{e^{Ne}_{n-1}\}.$ In the periodic case we set  $a_1^f=2,$ $a_{2n}^f=2+a^{per}_{n}, a^f_{2n+1}=2+a^{per}_{-n}$ and $e_1^f=e^{per}_{0},$ $e_{2n}^f=e^{per}_{n}, e^f_{2n+1}=e^{per}_{-n}$ for $n=1,2,\dots$.

   Turning to the spectral gap conditions, Theorems \ref{MDPthm} and \ref{Asymptoticthm} hold for any value of $\ell$ provided that the change of measure $u_{k_0}$ acts on a finite-dimensional eigenspace of sufficiently high dimension. For example, consider the Neumann problem with $\ell=4\pi/3.$ For this value of $\ell$, \eqref{AllenCahn} has $3$ saddle points (see e.g. \cite{schneider2017nonlinear}, Chapter 5.3.4) and it is easy to check that Hypothesis \ref{A3c} is violated. However, the weak spectral gap of Hypothesis \ref{A3c'} is satisfied for $k_0=3.$ Indeed, we have
   $$3a^f_1=6<2+\frac{9\pi^2}{(4\pi/3)^2}=2+\frac{81}{16}=7.0265=a^f_4.$$
   Thus, the asymptotic results hold with the change of measure $$ u_{3}(t,\eta)=2\big(a_1^f\langle \eta, e_1^f\rangle_\h e_1^f  +a_2^f\langle \eta, e_2^f\rangle_\h e_2^f+a_3^f\langle \eta, e_3^f\rangle_\h e_3^f       \big).$$
   As for the pre-asymptotic analysis of Section \ref{Sec4} and the numerical studies of the following section we work under the stronger spectral gap of Hypothesis \ref{A3c}. For the Neumann problem, this places the restriction
   $$3a_1^f=6<2+\frac{\pi^2}{\ell^2}= a_2^f\iff \ell<\frac{\pi}{2} $$
   which can be weakened to  $\ell<\pi/\sqrt{2}$ in view of Remark \ref{smallergaprem}. For the periodic problem, Hypothesis \ref{A3c} gives
   $$3a_1^f=6<2+\frac{4\pi^2}{\ell^2}= a_2^f\iff \ell<\pi.$$
   Finally, an example where the assumptions of Lemma \ref{a2lem} are satisfied is given by the Neumann problem with $\ell\geq \pi/\sqrt{2}.$ In this case it is straightforward to verify that $a_2^f=2+\pi^2/\ell^2\leq 4=2a_1^f.$
   \subsection{Stochastic Allen-Cahn with Dirichlet boundary conditions}\label{SS:DirichletAllenCahn} The eigenpairs of the Dirchlet Laplacian on the interval $(0,\ell)$ are explicitly given by
   \begin{equation}\label{eigenDirichlet}
  a^{Dir}_n=\bigg(\frac{n\pi}{\ell}\bigg)^2\;,\;   e^{Dir}_n(\xi)=\sqrt{\frac{2}{\ell}}\sin\bigg(\frac{n\pi\xi}{\ell}\bigg)\;,\;\;         n=1,2,\dots
   \end{equation}
    However, exact spectral analysis and numerical simulation of the linearized operators is more involved than the periodic and Neumann cases.  This is due to the fact that the stable equilibria $x^*_{\pm}$ are non-constant functions with absolute value less than or equal to $1$  that vanish at the endpoints $0,\ell.$ They can be determined by solving the Sturm-Liouville problem
    \begin{equation*}\label{Sturm}
   \left \{
   \begin{aligned}
   &x''(\xi)=V'_f(x(\xi))=x^3(\xi)-x(\xi) ,\; \xi\in(0,\ell)\\&
   x(0)=x(\ell)=0.
   \end{aligned}\right.
   \end{equation*}
    Following \cite{wakasa2006exact} (see also \cite{carrillo2000jacobi}), we can parametrize $x^*_{\pm}$ with respect to their minimum pointwise distance from the constant solutions $\pm 1$. The latter is in one-to-one correspondence with the bifurcation parameter $\ell$ (see \eqref{aellrelation} below).

  First, note that the scaling  $y(\xi)=x(\ell\xi)$ leads to the equivalent problem
   \begin{equation*}\label{equiDirichlet}
   \left \{
   \begin{aligned}
   &y''(\xi)=\ell^2\big(y^3(\xi)-y(\xi)\big),\; \xi\in(0,1)\\&
   y(0)=y(1)=0.
   \end{aligned}\right.
   \end{equation*}
    For any $a\in(0,1),$ the stable equilibria $y^*_{\pm}$ of the latter  are then given by $\pm y^*,$
   \begin{equation*}
   y^*(\xi)\equiv y^*_a(\xi)=a\;sn\bigg( 2 K\bigg(\frac{a^2}{2-a^2}\bigg)\xi, \frac{a^2}{2-a^2}   \bigg), \xi\in(0,1),
   \end{equation*}
   where for any $m\in(0,1),$  \begin{equation*}
   K(m):=\int_{0}^{1}\frac{dx}{\sqrt{(1-x^2)(1-m x^2)}}\;, m\in(0,1)
   \end{equation*}
   is the complete elliptic integral of the first kind and $ sn(\cdot,m)$ is the Jacobi elliptic sine function defined by
   \begin{equation*}
   x=\int_{0}^{sn(x,m)}\frac{dy}{\sqrt{(1-y^2)(1-m y^2)}}\;, x\in[0, K(m)].
   \end{equation*}
   The function $sn(\cdot,m)$ can be periodically extended to all of $\R$ so that $K(m)$ is its quarter-period.  We remark that there are several different parameterizations of $K$ in the literature (e.g. in \cite{wakasa2006exact} $K(\xi)$ corresponds to $K(\sqrt{\xi})$ in our notation). The definition above was chosen in agreement with \cite{abramowitz1964handbook} and the corresponding built-in Matlab function.

    The parameter $a$ is the maximum value of $y_a^*$ i.e.
    $$y^*_a\big(\inf\{\xi\in(0,1): y'(\xi)=0\}\big)=a.$$
    In order to convert to a parameterization in terms of $\ell,$ we first define a scaled quarter-period map $\mathcal{M}:(0,1)\rightarrow\R$ with
    \begin{equation*}
    \mathcal{M}(a):=\frac{1}{\sqrt{2}}\int_{0}^{a}\frac{dx}{\sqrt{V_f(a)-V_f(x)}}=\sqrt{2}\int_{0}^{1}\frac{dx}{\sqrt{1-x^2}\sqrt{2-a^2(1+x^2)}}=\frac{\sqrt{2}}{\sqrt{2-a^2}}K\bigg(\frac{a^2}{2-a^2}\bigg).
    \end{equation*}
 The correspondence of the interval length $\ell$ and $a$ is then given by
   \begin{equation}\label{aellrelation}
      \ell=2\mathcal{M}(a).
   \end{equation}
  As seen in the Figure 1, $\mathcal{M}$ is continuous, strictly increasing and $\lim_{a\to1}\mathcal{M}(a)=\infty$. Thus $\mathcal{M}$ is continuously invertible. Furthermore it is straightforward to verify that $\lim_{a\to 0}\mathcal{M}(a)=\pi/2$.   Putting the previous facts together we deduce that
  \begin{equation}\label{xstarDir}
  \begin{aligned}
  x^*_{+}(\xi)=-x^*_{-}(\xi)&=y^*_{+}(\xi/\ell)\\&
  = a\;sn\bigg( 2 K\bigg(\frac{a^2}{2-a^2}\bigg)\frac{\xi}{\ell}, \frac{a^2}{2-a^2}   \bigg)\\&
  =a\;sn\bigg( 2 K\bigg(\frac{a^2}{2-a^2}\bigg)\frac{\xi}{2\mathcal{M}(a)}, \frac{a^2}{2-a^2}   \bigg)\\&
  =a\; sn\bigg( \xi\sqrt{1-\frac{a^2}{2}}, \frac{a^2}{2-a^2}\bigg)\;,\xi\in(0,\ell)\;\;,\; a=\mathcal{M}^{-1}(\ell/2).
  \end{aligned}
  \end{equation}

   \begin{figure}[H]\label{fig1}
  	\centering
  	\includegraphics[width=0.3\linewidth]{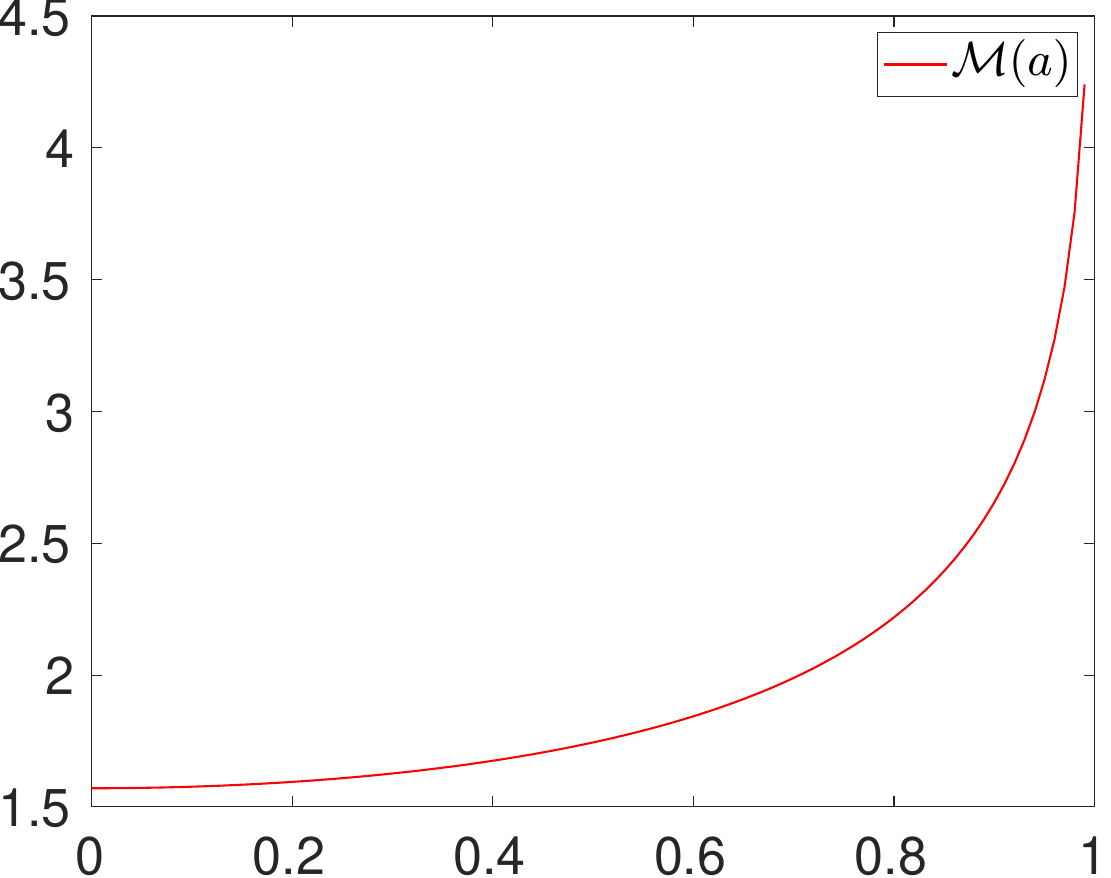}
  	\caption{The map $\mathcal{M}$.}
  \end{figure}
Turning to the spectral properties of the linearized operators $\Delta+DF(x^*_{\pm}),$ they have a countable sequence of eigenvalues-eigenvectors $\{(a_n^f, e_n^f)\}_{n\in\N},$ hence they satisfy Hypothesis \ref{A3b}. The first two pairs have been computed explicitly in \cite{wakasa2006exact} and are given by
  \begin{equation}\label{e1fdir}
  a_1^f=\frac{3}{2}a^2=\frac{3}{2}\mathcal{M}^{-1}(\ell/2)\;, e_1^f(\xi)=e_{1,a}^f(\xi)=sn\bigg( \xi\sqrt{1-\frac{a^2}{2}}, \frac{a^2}{2-a^2}\bigg)dn\bigg( \xi\sqrt{1-\frac{a^2}{2}}, \frac{a^2}{2-a^2}\bigg)
  \end{equation}
   and
   \begin{equation*}
   a_2^f=\frac{3}{2}(2-a^2)=\frac{3}{2}\big(2-\mathcal{M}^{-1}(\ell/2)\big)\;, e_2^f(\xi)= e^f_{2,a}(\xi)=sn\bigg( \xi\sqrt{1-\frac{a^2}{2}}, \frac{a^2}{2-a^2}\bigg)cn\bigg( \xi\sqrt{1-\frac{a^2}{2}}, \frac{a^2}{2-a^2}\bigg)
   \end{equation*}
    where $dn,cn$ denote the Jacobi delta amplitude and elliptic cosine functions
    \begin{equation*}
    dn(x,m):=\sqrt{1-m^2sn^2(x,m)}\;,\;\;cn^2(x,m):=1-sn^2(x,m), cn(0,m)=1.
    \end{equation*}
    The spectral gap of Hypothesis \ref{A3c} is then satisfied if
    $$\frac{3a_1^f}{a_2^f}=\frac{3a^2}{2-a^2}<1\iff a<\frac{\sqrt{2}}{2}\iff \ell< 2\mathcal{M}(\sqrt{2}/2)$$
    where we used the monotonicity of $\mathcal{M}$ and $2\mathcal{M}(\sqrt{2}/2)\approx 4.0043.$ Plots of the equilibria $y^*_{a}$ and eigenfunctions $e^f_{1,a}$ for $a=0.65,0.95$ are given in Figures 2, 3 below.


 \begin{figure}[H]
 	\centering
 	\includegraphics[width=0.4\linewidth]{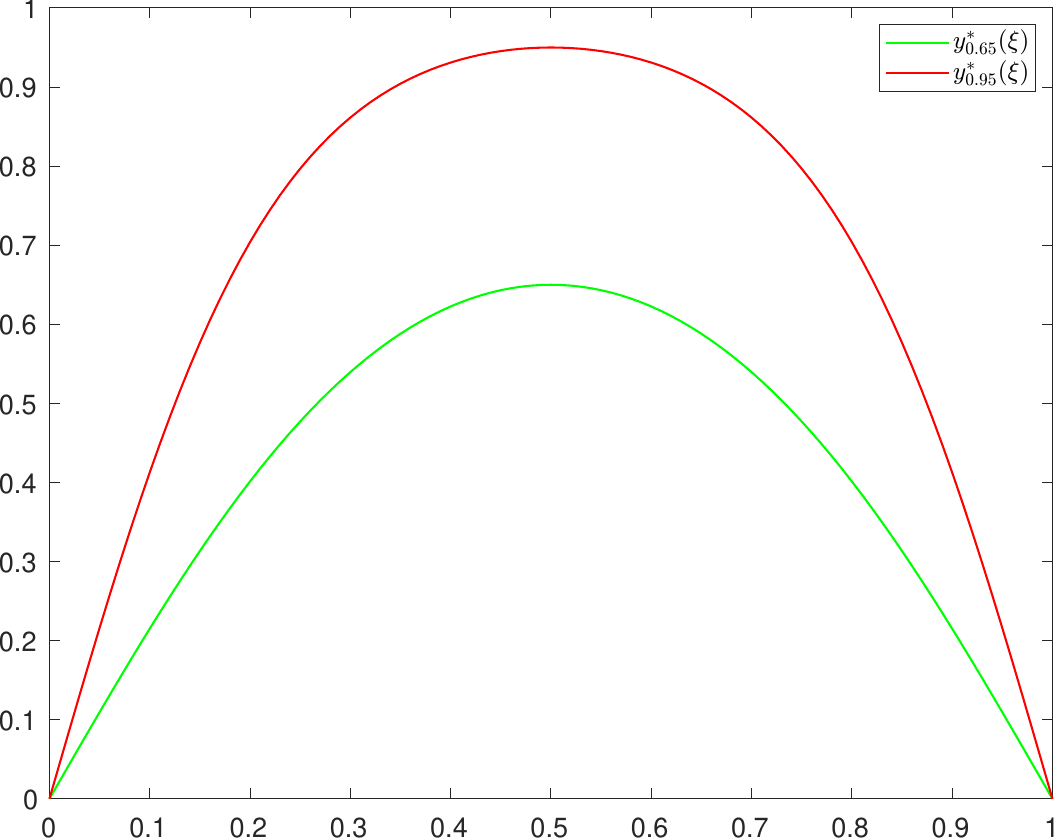}
 	\caption{Instances of the Dirichlet stable equilibrium $y_a^*$. }
 	\label{fig:dirichlet1}
 \end{figure}
  \begin{figure}[H]
 	\centering
 	\includegraphics[width=0.4\linewidth]{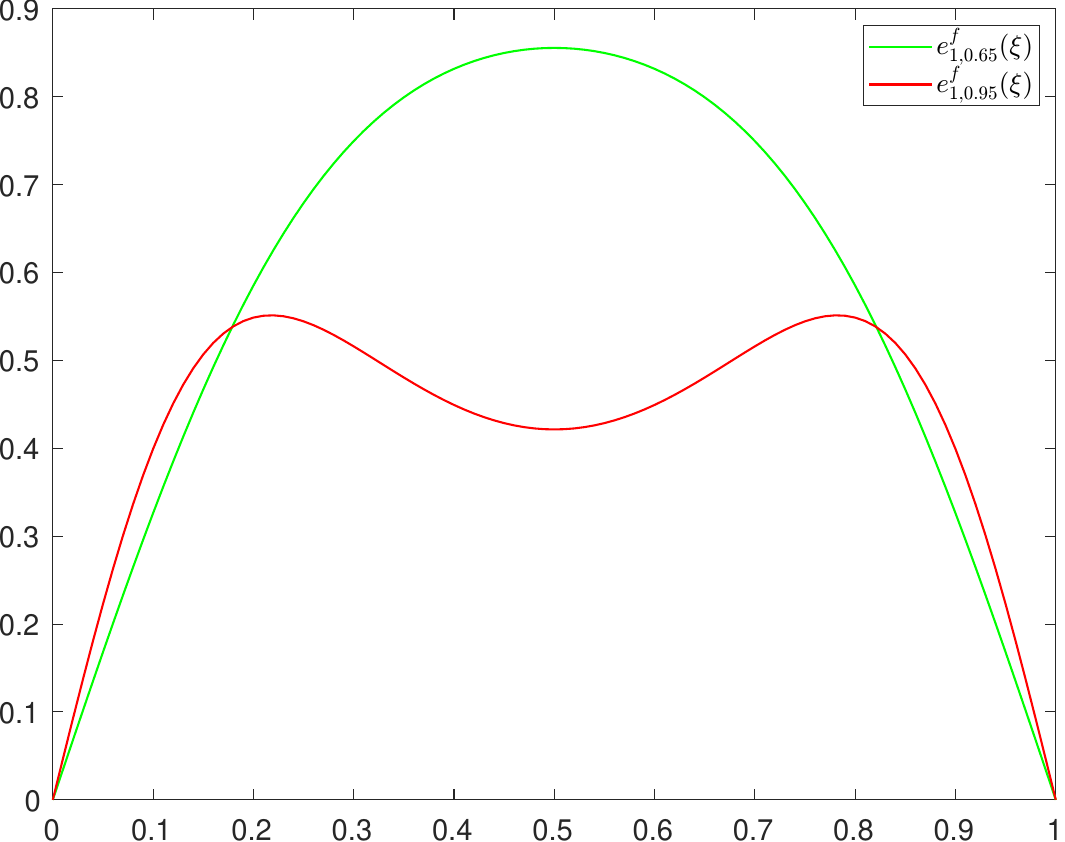}
 	\caption{Instances of the Dirichlet eigenfunctions $e^f_{1,a}$. }
 	\label{fig:psiplot}
 \end{figure}

\subsection{A higher-order Ginzburg-Landau SRDE}\label{SS:Ginzburg-Landau} We conclude this section with an example of an SRDE with a higher-order polynomial nonlinearity. This time we consider a potential given by
$$V_f(x)=\frac{4+\mu}{12}-\frac{1}{2}x^2-\frac{\mu}{4}x^4+\frac{\mu+1}{6}x^6\;,\;\;x\in\R. $$
If $\mu>-1$ then $V_f$ is a double-well potential with steeper walls than the fourth-order case. Such potentials have been considered in the physical literature as higher order quantum mechanical models, see e.g. \cite{kuznetsov1997periodic}. The nonlinear reaction term is given by $f(x)=-V'_f(x)= x+\mu x^3-(\mu+1)x^5$ and $f'(x) = 1+3\mu x^2-5(\mu+1)x^4$. For $\mu\in(-1,0],$ Hypothesis \ref{A2a} is satisfied with $f_1(x)=x$ and $f_2(x)=\mu x^3-(\mu+1)x^5$. The corresponding SRDE is given by
\begin{equation}
\label{quinticsrde}
\partial_tX^\epsilon=\partial^2_{\xi}X^{\epsilon}+ X^\epsilon+\mu\big(X^\epsilon\big)^3-(\mu+1)\big(X^\epsilon\big)^5+\sqrt{\epsilon}\dot{W}.
\end{equation}
The noiseless dynamics with Neumann or periodic boundary conditions are bistable for any $\ell>0$ with stable equilibria $x^*_{\pm}=\pm1$ and a saddle point $x_0^*=0$. The linearized operator $$\Delta+DF(\pm1)=\Delta+(1+3\mu x^2-5(\mu+1)x^4)|_{x=\pm1}=\Delta-2\mu-4$$
has the same eigenfunctions as the Laplacian and eigenvalues shifted by $-2\mu-4$.

As in the Allen-Cahn case, Theorems \ref{MDPthm} and \ref{Asymptoticthm} hold for any value of $\ell$ provided that the change of measure $u_{k_0}$ acts on a finite-dimensional eigenspace of sufficiently high dimension. The pre-asymptotic analysis of Section \ref{Sec4} holds under the spectral gap of Hypothesis \ref{A3c}. In the Neumann case the spectral gap holds if
$$3a_1^f=6\mu+12<2\mu+4+\frac{\pi^2}{\ell^2}=a_2^f\iff  \ell<\frac{\pi}{2\sqrt{\mu+2}}$$
and in the periodic case
$$3a_1^f=6\mu+12<2\mu+4+\frac{4\pi^2}{\ell^2}=a_2^f\iff  \ell<\frac{\pi}{\sqrt{\mu+2}}.$$
 \section{Numerical Simulations}\label{Sec6}

 In this section we demonstrate the theoretical results of this paper by a series of simulation studies for \eqref{evoeq}. As explained in Section \ref{Sec3}, we start the process $X_x^\epsilon$ at a stable equilibrium $x=x^*$ and develop a scheme that computes exit probabilities of the form
 \begin{equation*}
 P(\epsilon)= P(\epsilon,T)=\pr[ \tau_{x^*}^\epsilon\leq T]
 \end{equation*}
 for $\epsilon\ll1, T>0,$ where $
 \tau_{x^*}^\epsilon=\inf\{ t>0:  X_{x^*}^{\epsilon}\notin D\}$ and  $
 D=\mathring{B}_\h(x^*, L\sqrt\epsilon h(\epsilon)).$ For the simulations that follow we fix $L=1$ and set $R=R(\epsilon)=\sqrt\epsilon h(\epsilon)$. In view of Remark \ref{triviallimlem} we have
 $$P(\epsilon)=\pr\bigg[\sup_{t\in[0,T]}\big\|\eta_{x^*}^{\epsilon}(t)\big\|_{\h}\geq 1\bigg]$$
 and the process $\eta_{x^*}^{\epsilon}$ \eqref{etadef} converges in distribution to $0$ as $\epsilon\to 0$. Hence, for $\epsilon$ small, we are dealing with rare events. We will apply the scheme of Section \ref{Sec4} to the examples of Section \ref{Sec5} and compare its performance to the standard Monte Carlo, which corresponds to no change of measure at all.
 It is clear that in order to simulate the mild solutions in \eqref{etamild0}, \eqref{etamild} we need to discretize the equation in time and space. In the simulations below we used the exponential Euler scheme finite-dimensional Galerkin projection as it is described in \cite{jentzen2009overcoming}. In particular, with $\hat{\eta}^{\epsilon}=(\hat{X}^\epsilon-x^*)/\sqrt{\epsilon}h(\epsilon),$ $u^\epsilon$ as in \eqref{uchoice} and $\hat{X}^\epsilon$ solving
 \begin{equation}\label{evoeq2}
 \left\{\begin{aligned}
 &d\hat{X}^\epsilon(t)=[A\hat{X}^{\epsilon}(t)+F(\hat{X}^\epsilon(t))]dt+\sqrt{\epsilon}h(\epsilon)u^\epsilon(\hat{\eta}^{\epsilon}(t))dt+\sqrt{\epsilon}dW(t)\\&\hat{X}^{\epsilon}(0)=x^*
 \end{aligned}\right.
\end{equation}
we simulate the mild solution $\hat{X}^\epsilon$ on the sampling window $t\in[0,T]$ until it hits $\partial D$. Its $N$-th Galerkin projection is given in mild formulation by
$$X^\epsilon_N(t)=e^{A_Nt}P_Nx^*+\int_{0}^{t}e^{A_N(t-s)}\big[P_NF(\hat{X}^\epsilon_N(s))+\sqrt{\epsilon}h(\epsilon)P_Nu^\epsilon(\hat{\eta}_N^{\epsilon}(s))\big]ds+\sqrt{\epsilon}W_{A_N}(t) ,     $$
where $P_N$ denotes a projection to the $N$-dimensional subspace of $\h$ spanned by the eigenvectors $e_1,\dots,e_N$ of $A$ (not to be confused with the linearization eigenvectors $e^f_n$ of Hypothesis \ref{A3b}).
 Turning to the time discretization, we consider a time-step $h=T/\Delta t$ for some $\Delta t\in\N,$ discretization times $t_k=kh$, $k=0,\dots, \Delta t$ and set $\Theta_0^{N}:=P_Nx^*.$
 The exponential Euler scheme is then given by
 \begin{equation*}
  \begin{aligned}
 \Theta_{k+1}^{N}=e^{A_Nh}\Theta_k^{N}+A_N^{-1}\big[ e^{A_Nh}-I\big]^{-1}\big[P_NF\big( \Theta_{k}^{N}\big)+\sqrt{\epsilon}h(\epsilon)P_Nu^\epsilon\big(\tilde{\Theta}_{k}^{N}\big)  \big]+ \sqrt{\epsilon}\int_{t_{k}}^{t_{k+1}}e^{A_N(t_{k+1}-s) }P_NdW(s),
 \end{aligned}
 \end{equation*}
$k=0,\dots,\Delta t-1$    where $\tilde{\Theta}_{k}^{N}=(\Theta_{k}^{N}-\Theta_{0}^{N})/\sqrt{\epsilon}h(\epsilon).$ Letting $\Theta_{k,j}^N=\langle \Theta_{k}^N, e_j \rangle_\h, f_{k,j}^N=   \langle P_NF\big( \Theta_{k}^{N}\big), e_j\rangle_\h,  u_{k,j}^N=\langle P_Nu^\epsilon\big( \tilde{\Theta}_{k}^{N}\big), e_j\rangle_\h $ the numerical scheme for the approximation of \eqref{evoeq2} is then given by

\begin{equation}\label{exposcheme}
\begin{aligned}
\Theta_{k+1,j}^{N}=e^{-a_j h}\Theta_{k,j}^{N}+\frac{1-e^{-a_j h}}{a_j}\big( f_{k,j}^N+\sqrt{\epsilon}h(\epsilon) u_{k,j}^N  \big)+ \sqrt{\epsilon}\sqrt{\frac{1-e^{-2a_j h}}{2a_j}}w_{k,j}
\end{aligned}
\end{equation}
where for $k=0,\dots \Delta t-1, j=1,\dots, N$ $\xi_{k,j}$ are independent standard normal random variables.

For Neumann and periodic boundary conditions, the pairs $(a_j, e_j)$ are given by \eqref{Neuperevalues}, \eqref{Neuperefunctions}. Since the changes of measure $u^\epsilon$ act only in the direction of $e_1^f$ and the latter coincides with $e_0$ (i.e. a constant function) we have that $u_{k,j}^N=0$ when $j\neq 0.$ However the eigenvalue $a_0$ is in both cases equal to $0.$ Hence the exponential  Euler scheme is not well-defined for $j=0$. For this reason we simulate $\Theta_{k+1,0}^{N}$ via an explicit Euler scheme i.e.
\begin{equation*}
\Theta_{k+1,0}^N= \Theta_{k,0}^{N}+
h\big[f_{k,0}^N+\sqrt{\epsilon}h(\epsilon)u_{k,0}^N  \big]+\sqrt{\epsilon}\sqrt{h}w_{k,0}\;,\;\; k=0,\dots,\Delta t-1,
\end{equation*}
where $\Theta_{0,0}^{N}:=\langle x^*, e_0\rangle_\h$, $f_{k,0}^N= \langle P_NF\big( \Theta_{k}^{N}\big), e_0\rangle_\h,  u_{k,0}^N=\langle P_Nu^\epsilon\big( \tilde{\Theta}_{k}^{N}\big), e_0\rangle_\h$ and $w_{k,0}$ are once again independent standard normal random variables. The computation of the coefficients $\{ f^{N}_{k,j}\}$ in the Neumann (respectively periodic) case can be efficiently performed by applying a forward-backward odd (resp. periodic or Hartley-type) Fast Fourier Transform (FFT) in an iterative fashion. For more details on the discrete Fourier transform and the FFT algorithm the reader is referred to \cite{brigham1988fast} (Chapters 6 and 8 respectively).

Turning to the stochastic Allen-Cahn with Dirichlet boundary conditions the simulations require an additional step. As discussed in Section \ref{Sec5}, the stable equilibrium $x^*$ \eqref{xstarDir} is no longer a constant function and the changes of measure $u^\epsilon$ push towards $e_1^f$ \eqref{e1fdir} which no longer coincides with a single eigenvector $e^{Dir}_{k}$ \eqref{eigenDirichlet}. Thus, one needs to express $x^*$ and $e_1^f$ in terms of the eigenbasis $\{e^{Dir}_{k}\}_{k\in\N}$ of the Laplacian and then perform the exponential Euler scheme \eqref{exposcheme}. If the changes of measure acted on a higher dimensional eigenspace, this step essentially reduces to a change of basis which can be computed with numerical linear-algebraic methods. Regarding the coefficients $\{ f^{N}_{k,j}\},$ these can be computed by applying a forward-backward even Fast Fourier Transform iteratively.

\begin{rem} In the examples of Section \ref{Sec5} the stable equilibrium $x^*$ and the spectra of the linearized operators could be found explicitly. We remark here that our scheme does not depend on explicit formulas for eigenvalues and eigenvectors as long as those can be approximated numerically and the approximated eigenvalues satisfy Hypothesis \ref{A3c}.
	
\end{rem}

All the simulations below were done using a parallel MPI $C$ code with $M=5\times 10^4$ Monte Carlo trajectories. The FFTs were performed with the aid of the $C$ library FFTW. As it is standard in the related literature (see e.g. \cite{asmussen2007stochastic}, Chapter VI,1), the measure of performance is relative error per sample, defined as
\begin{equation*}
\text{relative error per sample} =\sqrt{M}\frac{\text{st.deviation}(\hat{P}^\epsilon) }{\text{expectation} [\hat{P}(\epsilon)]}.
\end{equation*}

The smaller the relative error per sample, the more efficient the algorithm and the more accurate the estimator. However, in practice both the standard deviation and the expected value of an estimator are typically unknown, which implies that empirical relative error is often used for measurement. This means that the expected value of the estimator will be replaced by the empirical sample mean
   and the standard deviation of the estimator will be replaced by the empirical sample standard error. A dash line in the simulation tables indicates that no trajectory exited $D$ before time $T$. Before presenting the simulation tables,     let us make a few comments on the parameter values and the end conclusions of the numerical studies.\\

\noindent \textbf{1)} (\textit{Simulations for the Neumann stochastic Allen-Cahn}) We estimate exit probabilities $P(\epsilon)$ for the solution  $X^\epsilon$ of \eqref{AllenCahn} driven by additive space-time white noise on the interval $(0,\ell)$ and Neumann boundary conditions. For the simulations we set $\ell=1, x^*=x^*_{+}=1,$
   $h(\epsilon)=\epsilon^{-0.1}$ and Galerkin projection level $N=50$. The numerical results can be found in Tables \ref{Neuprobimportnew}-\ref{Neuerrorsmcnew}.\\

   \noindent \textbf{2) }(\textit{Simulations for the periodic stochastic Allen-Cahn}) We estimate $P(\epsilon)$ for the solution  $X^\epsilon$ of \eqref{AllenCahn} driven by additive space-time white noise on the interval $(0,\ell)$ and periodic boundary conditions. For the simulations we set
   $x^*_{+}=1, \ell=1, R=R(\epsilon):=\sqrt{\epsilon}h(\epsilon)\in(0,1),   \;h(\epsilon)=\epsilon^{-0.1}, \eta^{\epsilon}=(X^{\epsilon}-x^*_{+})/R(\epsilon)$, Galerkin projection level $N=50$. The numerical results can be found in Tables \ref{perprobimportancenew}-\ref{pererrorbsmcnew}.\\

\noindent \textbf{3) }(\textit{Simulations for the Dirichlet stochastic Allen-Cahn}) We estimate
$P(\epsilon)$ for the solution  $X^\epsilon$ of \eqref{AllenCahn} driven by additive space-time white noise on the interval $(0,\ell)$ and Dirichlet boundary conditions. For the simulations we set $\ell=3.81828$,   $x^*=x^*_{+}=a\quad sn\bigg( \xi\sqrt{1-\frac{a^2}{2}}, \frac{a^2}{2-a^2}\bigg)$ with $a=\mathcal{M}^{-1}(\ell/2)=0.65,$
$h(\epsilon)=\epsilon^{-0.1}$ and Galerkin projection level $N=50$. Note that $\|x^*_+\|_{L^2}\approx 0.33$. The numerical results can be found in Tables \ref{Dirprobimportancenew}-\ref{Direrrorsmcnew}.   \\

   \noindent \textbf{4) }(\textit{Simulations for the quintic SRDE \eqref{quinticsrde}}) We estimate $P(\epsilon)$ for the solution  $X^\epsilon$ of \eqref{quinticsrde} driven by additive space-time white noise on the interval $(0,\ell)$ and Neumann boundary conditions. For the simulations we set
  $\mu=-0.5, x^*_{+}=1, \ell=1, R=R(\epsilon):=\sqrt{\epsilon}h(\epsilon)\in(0,1),   \;h(\epsilon)=\epsilon^{-0.1}, \eta^{\epsilon}=(X^{\epsilon}-x^*_{+})/R(\epsilon)$, Galerkin projection level $N=50$. The numerical results can be found in Tables \ref{quinticimportanceprobnew}-\ref{quinticSMCerrornew}.\\

   \noindent \textbf{5) } Standard Monte Carlo (sMC) estimation, i.e. with no change of measure does not perform well  for small values of $\epsilon,$ as indicated in Tables \ref{Neuerrorsmcnew},\ref{Direrrorsmcnew},\ref{pererrorbsmcnew}. A dash line indicates that there was no successful trajectory in the simulations and thus no estimate could be provided. The relative errors per sample are getting increasingly large making most of the reported probability values of Tables \ref{Neuprobsmcnew},\ref{Dirprobsmcnew},\ref{perprobsmcnew} to be of no value.\\

   \noindent \textbf{6) } The importance sampling scheme for the Allen-Cahn equation outperforms sMC and performs well for all boundary conditions and probabilities ranging from $10^{-1}$ to $10^{-11}$ (see Tables \ref{Neuprobimportnew}, \ref{Dirprobimportancenew}, \ref{perprobimportancenew}). In particular, the relative errors for the former are way lower than those of sMC. As expected, the estimated probabilities resulting from sMC and importance sampling scheme agree when the relative errors are below $10.0$. The relative errors per sample for the importance sampling scheme lie mostly below $1.2.$ The relative-error trends indicate that the accuracy improves as the sampling time grows from $T=1$ to $T=8.$  The relative errors per sample as reported in Tables \ref{Neuerrorimportnew},\ref{Direrrorimportancenew}, \ref{pererrorimportancenew} support the theoretical findings in that the scheme performs optimally as the theory predicts.\\ 
    \noindent \textbf{7) } The performance of the importance sampling scheme for the quintic SRDE \eqref{quinticsrde} experiences a slight degradation after $T=3$ (see Table \ref{quinticimportanceerrornew}). Nevertheless, it remains superior to that of the sMC (compare to Table \ref{quinticSMCerrornew}) while the relative errors remain mostly below $2.5$ and decrease with $\epsilon$.\\

    \noindent \textbf{8) } Table \ref{hdifference} provides a comparison between the importance sampling relative errors for the Neumann Allen-Cahn with $h(\epsilon)=\epsilon^{-0.2}$, $h(\epsilon)=\epsilon^{-0.1}.$ The sampling time is fixed to $T=2.$ We observe that the scaling $h(\epsilon)=\epsilon^{-0.1}$ leads to significantly lower relative errors than $ h(\epsilon)=\epsilon^{-0.2}$. This behavior is correctly predicted by \eqref{preasymptoticfin2} since the first satisfies $\sqrt{\epsilon}h^3(\epsilon)\rightarrow 0$ while the second does not. We remark that, despite the higher relative errors, the importance sampling scheme with $ h(\epsilon)=\epsilon^{-0.2}$ still outperforms sMC. Complete simulation tables for $h(\epsilon)=\epsilon^{-0.2}$ are available upon request.\\

    \noindent \textbf{9) } In Table \ref{Ntable}, we work with the Neumann Allen-Cahn and compare relative errors for different levels $N$ of the Galerkin approximation with $N = 50,100,150$. The sampling time is fixed to $T=3$. We notice that the relative errors are practically of the same order. This indicates that the first mode really dominates the rare event. Another observation we made is that the total simulation time increased significantly as we increased $N.$    These considerations led us to conclude that $N =50$ is an efficient and sufficiently good lower dimensional approximation to the corresponding SPDE.\\


 \subsection{Numerical results for Stochastic Allen-Cahn with Neumann boundary conditions}
 
 In this section, we provide numerical simulation results validating our theory for the stochastic Allen-Cahn equation with Neumann  boundary conditions studied in Subsection \ref{Neumannsec}.

\begin{table}[H]
	\begin{tabular}{ |c|c | c | c | c | c | c|c|}
		\hline
		$\epsilon$&$R=\sqrt{\epsilon}h(\epsilon)$&$T=1$&$T=2$ &$T=3$&$T=4$&$T=6$&$T=8$\\ \hline
		$0.01$ &$0.158489$&$1.93e-02$ &$5.43e-02$ &$8.73e-02$ &$1.20e-01$&$1.80e-01$&$2.35e-01$
		\\ \hline
		$0.004$ &$0.109856$&$6.65e-03$&$2.01e-02$&$3.37e-02$&$4.66e-02$&$7.19e-02$&$9.75e-02$
		\\ \hline
		$0.002$ &$0.083255$&$2.60e-03$&$8.35e-03$&$1.41e-02$&$1.98e-02$&$3.13e-02$&$4.24e-02$\\
		\hline
		$0.0008$&$0.057708$&$6.22e-04$&$2.09e-03$&$3.64e-03$&$5.13e-03$&$8.24e-03$&$1.12e-02$\\
		\hline
		$0.0004$&$0.043734$&$1.72e-04$&$6.21e-04$&$1.08e-03$&$1.53e-03$&$2.46e-03$&$3.36e-03$\\
		\hline
		$0.0001$&$0.025119$&$6.95e-06$&$2.92e-05$&$5.28e-05$&$7.62e-05$&$1.24e-04$&$1.70e-04$\\
		\hline
		$0.00006$&$0.020477$&$1.77e-06$&$7.63e-06$&$1.38e-05$&$2.00e-05$&$3.26e-05$&$4.45e-05$\\
		\hline
		$0.000008$&$0.009146$&$1.31e-09$&$7.45e-09$&$1.39e-08$&$2.05e-08$&$3.39e-08$&$4.72e-08$\\
		\hline
		$0.000004$&$0.006931$ &$4.96e-11$&$3.30e-10$&$6.21e-10$&$9.31e-10$&$1.52e-09$&$2.13e-09$\\
		\hline
	\end{tabular}
	\vspace{0.2cm}
 \caption{Estimated probability values $P(\epsilon,T)$ for the stochastic Allen-Cahn equation with Neumann boundary conditions using the developed importance sampling scheme with $\kappa=0.9$ and mollification parameter $\delta:=2/h^2(\epsilon)$.}
\label{Neuprobimportnew}
\end{table}

\begin{table}[H]
	\begin{tabular}{ |c| c | c | c | c | c|c|}
		\hline
		$\epsilon$&$T=1$&$T=2$ &$T=3$&$T=4$&$T=6$&$T=8$\\ \hline
		$0.01$ & $2.1$&$1.2$ & $1.0$&$0.9$ &$1.0$&$1.3$
		\\ \hline
		$0.004$ &$2.4$&$1.3$&$1.0$&$1.0$&$1.0$&$1.3$
		\\ \hline
		$0.002$ &$2.5$&$1.4$&$1.1$&$1.0$&$1.0$&$1.2$\\ \hline
		$0.0008$ &$2.7$&$1.5$&$1.1$&$1.0$&$1.0$&$1.2$\\ \hline
		$0.0004$&$2.9$&$1.5$&$1.1$&$1.0$&$0.9$&$1.1$\\
		\hline
		$0.0001$&$3.3$&$1.6$&$1.2$&$1.0$&$0.9$&$1.1$\\
		\hline
		$0.00006$&$3.4$&$1.6$&$1.2$&$1.0$&$0.9$&$1.1$\\
		\hline
		$0.000008$&$4.2$&$1.7$&$1.2$&$1.0$&$0.9$&$1.0$\\
		\hline
		$0.000004$&$4.6$&$1.7$&$1.3$&$1.0$&$0.9$&$1.0$\\
		\hline
	\end{tabular}
	\vspace{0.2cm}
	\caption{Estimated relative errors per sample for the stochastic Allen-Cahn equation with Neumann boundary conditions using the developed importance sampling scheme with $\kappa=0.9$ and mollification parameter $\delta:=2/h^2(\epsilon)$. }
	\label{Neuerrorimportnew}
\end{table}

\begin{table}[H]
	\begin{tabular}{ |c|c | c | c | c | c | c|c|}
		\hline
		$\epsilon$&$R=\sqrt{\epsilon}h(\epsilon)$&$T=1$&$T=2$ &$T=3$&$T=4$&$T=6$&$T=8$\\ \hline
		$0.01$ &$0.158489$&$2.07e-02$ &$5.40e-02$ &$8.66e-02$ &$1.20e-01$&$1.80e-01$&$2.40e-01$
		\\ \hline
		$0.004$ &$0.109856$&$6.80e-03$&$2.02e-02$&$3.36e-02$&$4.65e-02$&$7.15e-02$&$9.82e-02$
		\\ \hline
		$0.002$ &$0.083255$&$2.60e-03$&$8.58e-03$&$1.41e-02$&$1.90e-02$&$3.05e-02$&$4.39e-02$\\
		\hline
		$0.0008$&$0.057708$&$3.80e-04$&$2.18e-03$&$3.22e-03$&$4.82e-03$&$7.32e-03$&$1.13e-02$\\
		\hline
		$0.0004$&$0.043734$&$1.80e-04$&$5.60e-04$&$9.40e-04$&$1.60e-03$&$2.50e-04$&$3.10e-03$\\
		\hline
		$0.0001$&$0.025119$&--&$4.00e-05$&$8.00e-05$&$1.40e-04$&$8.00e-05$&$1.40e-04$\\
		\hline
		$0.00006$&$0.020477$&$2.00e-05$&$2.00e-05$&--&$2.00e-05$&$6.00e-05$&$4.00e-05$\\
		\hline
		$0.000008$&$0.009146$&--&--&--&--&--&--\\
		\hline
		$0.000004$&$0.006931$ &--&--&--&--&--&--\\
		\hline
	\end{tabular}
	\vspace{0.2cm}
	\caption{Estimated probability values $P(\epsilon,T)$ for the stochastic Allen-Cahn equation with Neumann boundary conditions. The values reported are based on standard Monte Carlo simulation without employing some change of measure.}
	\label{Neuprobsmcnew}
\end{table}

\begin{table}[H]
	\begin{tabular}{ |c| c | c | c | c | c|c|}
		\hline
		$\epsilon$&$T=1$&$T=2$ &$T=3$&$T=4$&$T=6$&$T=8$\\ \hline
		$0.01$ & $6.9$&$4.2$ & $3.2$&$2.7$ &$2.1$&$1.8$
		\\ \hline
		$0.004$ &$12.1$&$7.0$&$5.4$&$4.5$&$3.6$&$3.0$
		\\ \hline
		$0.002$ &$19.6$&$10.7$&$8.4$&$7.2$&$5.6$&$4.7$\\ \hline
		$0.0008$ &$51.3$&$21.4$&$17.6$&$14.4$&$11.6$&$9.3$\\ \hline
		$0.0004$&$74.5$&$42.2$&$32.6$&$25.0$&$20.0$&$17.9$\\
		\hline
		$0.0001$&--&$158.1$&$111.8$&$84.5$&$111.8$&$84.5$\\
		\hline
		$0.00006$&$223.6$&$223.6$&--&$223.6$&$129.1$&$158.1$\\
		\hline
		$0.000008$&--&--&--&--&--&--\\
		\hline
		$0.000004$&--&--&--&--&--&--\\
		\hline
	\end{tabular}
	\vspace{0.2cm}
	\caption{Estimated relative errors per sample for the stochastic Allen-Cahn equation with Neumann boundary conditions. The values reported are based on standard Monte Carlo simulation without employing some change of measure. A probability of $2\times 10^{-5}$ means that only one out of the $5\times 10^4$ trajectories exited the domain. The relative error in that case is $223.6$. }
	\label{Neuerrorsmcnew}
\end{table}

\subsection{Numerical results for Stochastic Allen-Cahn with periodic boundary conditions}
 In this section, we provide numerical simulation results validating our theory for the stochastic Allen-Cahn equation with periodic  boundary conditions studied in Subsection \ref{Neumannsec}.

\begin{table}[H]
	\begin{tabular}{ |c|c | c | c | c | c | c|c|}
		\hline
		$\epsilon$&$R=\sqrt{\epsilon}h(\epsilon)$&$T=1$&$T=2$ &$T=3$&$T=4$&$T=6$&$T=8$\\ \hline
		$0.01$ &$0.158489$&$1.69e-02$ &$4.82e-02$ &$7.73e-02$ &$1.07e-01$&$1.62e-01$&$2.14e-01$
		\\ \hline
		$0.004$ &$0.109856$&$5.91e-03$&$1.82e-02$&$3.02e-02$&$4.23e-02$&$6.60e-02$&$8.87e-02$
		\\ \hline
		$0.002$ &$0.083255$&$2.39e-03$&$7.68e-03$&$1.29e-02$&$1.82e-02$&$2.88e-02$&$3.94e-02$\\
		\hline
		$0.0008$&$0.057708$&$5.62e-04$&$1.97e-03$&$3.44e-03$&$4.84e-03$&$7.71e-03$&$1.04e-02$\\
		\hline
		$0.0004$&$0.043734$&$1.61e-04$&$5.92e-04$&$1.03e-03$&$1.47e-03$&$2.34e-03$&$3.24e-03$\\
		\hline
		$0.0001$&$0.025119$&$6.90e-06$&$2.92e-05$&$5.27e-05$&$7.57e-05$&$1.22e-04$&$1.70e-04$\\
		\hline
		$0.00006$&$0.020477$&$1.73e-06$&$7.68e-06$&$1.38e-05$&$2.00e-05$&$3.25e-05$&$4.52e-05$\\
		\hline
		$0.000008$&$0.009146$&$1.34e-09$&$7.90e-09$&$1.52e-08$&$2.23e-08$&$3.66e-08$&$5.13e-08$\\
		\hline
		$0.000004$&$0.006931$ &$5.79e-11$&$3.65e-10$&$7.00e-10$&$1.05e-09$&$1.70e-09$&$2.37e-09$\\
		\hline
	\end{tabular}
	\vspace{0.2cm}
	\caption{Estimated probabilities for the stochastic Allen-Cahn equation with periodic boundary conditions using the developed importance sampling scheme with $\kappa=0.9$ and mollification parameter $\delta:=2/h^2(\epsilon)$.}
	\label{perprobimportancenew}
\end{table}

\begin{table}[H]
	\begin{tabular}{ |c| c | c | c | c | c|c|}
		\hline
		$\epsilon$&$T=1$&$T=2$ &$T=3$&$T=4$&$T=6$&$T=8$\\ \hline
		$0.01$ &$2.1$ &$1.2$ &$1.0$ &$0.9$&$1.0$&$1.3$
		\\ \hline
		$0.004$ &$2.4$&$1.3$&$1.0$&$0.9$&$1.0$&$1.2$
		\\ \hline
		$0.002$ &$2.5$&$1.4$&$1.1$&$0.9$&$1.0$&$1.2$\\
		\hline
		$0.0008$&$2.7$&$1.4$&$1.1$&$0.9$&$1.0$&$1.1$\\
		\hline
		$0.0004$&$2.9$&$1.4$&$1.1$&$1.0$&$0.9$&$1.1$\\
		\hline
		$0.0001$&$3.2$&$1.5$&$1.1$&$1.0$&$0.9$&$1.1$\\
		\hline
		$0.00006$&$3.4$&$1.6$&$1.1$&$1.0$&$0.9$&$1.0$\\
		\hline
		$0.000008$&$4.2$&$1.7$&$1.2$&$1.0$&$0.9$&$1.0$\\
		\hline
		$0.000004$ &$4.4$&$1.7$&$1.2$&$1.0$&$0.9$&$1.0$\\
		\hline
	\end{tabular}
	\vspace{0.2cm}
	\caption{Estimated relative errors per sample for the stochastic Allen-Cahn equation with periodic boundary conditions using the developed importance sampling scheme with $\kappa=0.9$ and mollification parameter $\delta:=2/h^2(\epsilon)$.}
	\label{pererrorimportancenew}
\end{table}

\begin{table}[H]
	\begin{tabular}{ |c|c | c | c | c | c | c|c|}
		\hline
		$\epsilon$&$R=\sqrt{\epsilon}h(\epsilon)$&$T=1$&$T=2$ &$T=3$&$T=4$&$T=6$&$T=8$\\ \hline
		$0.01$ &$0.158489$&$1.73e-02$ &$4.96e-02$ &$7.72e-02$ &$1.07e-01$&$1.62e-01$&$2.13e-01$
		\\ \hline
		$0.004$ &$0.109856$&$6.44e-03$&$1.80e-02$&$3.00e-02$&$4.27e-02$&$6.45e-02$&$8.98e-02$
		\\ \hline
		$0.002$ &$0.083255$&$2.54e-03$&$8.24e-03$&$1.28e-02$&$1.81e-02$&$2.77e-02$&$3.89e-02$\\
		\hline
		$0.0008$&$0.057708$&$6.20e-04$&$1.34e-03$&$3.01e-03$&$4.52e-03$&$7.94e-03$&$1.02e-02$\\
		\hline
		$0.0004$&$0.043734$&$1.40e-04$&$5.20e-04$&$1.04e-03$&$1.42e-03$&$2.24e-03$&$3.30e-03$\\
		\hline
		$0.0001$&$0.025119$&$2.00e-05$&--&$1.00e-04$&$1.00e-04$&$1.20e-04$&$1.40e-04$\\
		\hline
		$0.00006$&$0.020477$&--&$2.00e-05$&--&$6.00e-05$&$2.00e-05$&--\\
		\hline
		$0.000008$&$0.009146$&--&--&--&--&--&--\\
		\hline
		$0.000004$&$0.006931$ &--&--&--&--&--&--\\
		\hline
	\end{tabular}
	\vspace{0.2cm}
	\caption{Estimated probabilities for the stochastic Allen-Cahn equation with periodic boundary conditions. The values reported are based on standard Monte Carlo simulation without employing some change of measure.}
	\label{perprobsmcnew}
\end{table}

\begin{table}[H]
	\begin{tabular}{ |c| c | c | c | c | c|c|}
		\hline
		$\epsilon$&$T=1$&$T=2$ &$T=3$&$T=4$&$T=6$&$T=8$\\ \hline
		$0.01$ & $7.5$&$4.4$ & $3.5$&$2.9$ &$2.3$&$1.9$
		\\ \hline
		$0.004$ &$12.4$&$7.4$&$5.7$&$4.7$&$3.8$&$3.2$
		\\ \hline
		$0.002$ &$19.8$&$11.0$&$8.8$&$7.4$&$5.9$&$5.0$\\ \hline
		$0.0008$ &$40.1$&$27.3$&$17.9$&$14.8$&$11.2$&$9.9$\\ \hline
		$0.0004$&$84.5$&$43.8$&$30.1$&$26.5$&$21.1$&$17.4$\\
		\hline
		$0.0001$&$223.6$&--&$100.0$&$100.0$&$91.3$&$84.5$\\
		\hline
		$0.00006$&--&$223.6$&--&$129.1$&$223.6$&--\\
		\hline
		$0.000008$&--&--&--&--&--&--\\
		\hline
		$0.000004$&--&--&--&--&--&--\\
		\hline
	\end{tabular}
	\vspace{0.2cm}
	\caption{Estimated relative error per sample for the stochastic Allen-Cahn equation with periodic boundary conditions. The values reported are based on standard Monte Carlo simulation without employing some change of measure.}
	\label{pererrorbsmcnew}
\end{table}

\subsection{Numerical results for Stochastic Allen-Cahn with Dirichlet boundary conditions}

In this section, we provide numerical simulation results validating our theory for the stochastic Allen-Cahn equation with Dirichlet  boundary conditions studied in Subsection \ref{SS:DirichletAllenCahn}.

\begin{table}[H]
	\begin{tabular}{ |c|c | c | c | c | c | c|c|}
		\hline
		$\epsilon$&$R=\sqrt{\epsilon}h(\epsilon)$&$T=1$&$T=2$ &$T=3$&$T=4$&$T=6$&$T=8$\\ \hline
		$0.00008$ &$0.022974$&$3.96e-03$ &$2.47e-02$ &$5.06e-02$ &$7.69e-01$&$1.26e-01$&$1.73e-01$
		\\ \hline
		$0.00001$ &$0.010000$&$1.55e-04$&$2.19e-03$&$5.55e-03$&$9.23e-03$&$1.69e-02$&$2.43e-02$
		\\ \hline
		$0.000004$ &$0.006931$&$2.40e-05$&$5.18e-04$&$1.47e-03$&$2.61e-03$&$5.00e-03$&$7.37e-03$\\
		\hline
		$0.000001$&$0.003981$&$5.56e-07$&$3.44e-05$&$1.23e-04$&$2.32e-04$&$4.69e-04$&$7.13e-04$\\
		\hline
		$0.0000008$&$0.003641$&$4.82e-07$&$2.07e-05$&$7.71e-05$&$1.47e-04$&$3.00e-04$&$4.52e-04$\\
		\hline
		$0.0000004$&$0.002759$&$4.35e-08$&$3.68e-06$&$1.55e-05$&$3.06e-05$&$6.54e-05$&$1.00e-04$\\
		\hline
		$0.0000002$&$0.002091$&$2.56e-09$&$4.82e-07$&$2.38e-06$&$5.12e-06$&$1.12e-05$&$1.75e-05$\\
		\hline
		$0.0000001$&$0.001585$&$1.24e-10$&$5.23e-08$&$2.84e-07$&$6.34e-07$&$1.48e-06$&$2.33e-06$\\
		\hline
		$0.00000008$&$0.001450$ &$4.32e-11$&$2.27e-08$&$1.33e-07$&$3.08e-07$&$7.20e-07$&$1.14e-06$\\
		\hline
		\end{tabular}
		\vspace{0.2cm}
		\caption{Estimated probability values $P(\epsilon,T)$ for the stochastic Allen-Cahn equation with Dirichlet boundary conditions using the developed importance sampling scheme with $\kappa=0.9$ and mollification parameter $\delta:=2/h^2(\epsilon)$.}
		\label{Dirprobimportancenew}
		\end{table}
		
		\begin{table}[H]
		\begin{tabular}{ |c| c | c | c | c | c|c|}
	\hline
	$\epsilon$&$T=1$&$T=2$ &$T=3$&$T=4$&$T=6$&$T=8$\\ \hline
	$0.00008$&$5.3$ &$2.1$ &$1.4$ &$1.1$&$0.9$&$0.8$
	\\ \hline
	$0.00001$ &$10.1$&$2.8$&$1.7$&$1.3$&$1.0$&$0.9$
	\\ \hline
	$0.000004$&$14.5$&$3.3$&$1.9$&$1.4$&$1.0$&$0.9$\\
	\hline
	$0.000001$&$29.1$&$4.0$&$2.1$&$1.5$&$1.0$&$0.9$\\
	\hline
	$0.0000008$&$26.1$&$4.1$&$2.2$&$1.5$&$1.1$&$0.9$\\
	\hline
	$0.0000004$&$36.8$&$4.6$&$2.3$&$1.6$&$1.1$&$0.9$\\
	\hline
	$0.0000002$&$65.9$&$5.2$&$2.4$&$1.6$&$1.1$&$0.9$\\
	\hline
	$0.0000001$&$100.4$&$6.0$&$2.6$&$1.7$&$1.1$&$0.9$\\
	\hline
	$0.00000008$ &$112.4$&$6.2$&$2.7$&$1.8$&$1.1$&$0.9$\\
	\hline
\end{tabular}
\vspace{0.2cm}
\caption{Estimated relative errors per sample for the stochastic Allen-Cahn equation with Dirichlet boundary conditions using the developed importance sampling scheme with $\kappa=0.9$ and mollification parameter $\delta:=2/h^2(\epsilon)$.}
\label{Direrrorimportancenew}
\end{table}

\begin{table}[H]
	\begin{tabular}{ |c|c | c | c | c | c | c|c|}
	\hline
	$\epsilon$&$R=\sqrt{\epsilon}h(\epsilon)$&$T=1$&$T=2$ &$T=3$&$T=4$&$T=6$&$T=8$\\ \hline
	$0.00008$ &$0.022974$&$3.64e-03$ &$2.48e-02$ &$5.00e-02$ &$7.68e-02$&$1.27e-01$&$1.73e-01$
	\\ \hline
	$0.00001$ &$0.010000$&$1.00e-04$&$2.22e-03$&$6.12e-03$&$9.16e-03$&$1.63e-02$&$2.33e-02$
	\\ \hline
	$0.000004$ &$0.006931$&$2.00e-05$&$4.40e-04$&$1.84e-03$&$2.76e-03$&$4.80e-03$&$7.24e-03$\\
	\hline
	$0.000001$&$0.003981$&$2.00e-05$&$4.00e-05$&$1.20e-04$&$2.20e-04$&$4.40e-04$&$6.40e-04$\\
	\hline
	$0.0000008$&$0.003641$&--&$4.00e-05$&$4.00e-05$&$1.40e-04$&$3.20e-04$&$5.80e-04$\\
	\hline
	$0.0000004$&$0.002759$&--&--&$2.00e-05$&--&$4.00e-05$&$8.00e-05$\\
	\hline
	$0.0000002$&$0.002091$&--&--&--&$2.00e-05$&--&$2.00e-05$\\
	\hline
	$0.0000001$&$0.001585$&--&--&--&--&--&--\\
	\hline
	$0.00000008$&$0.001450$ &--&--&--&--&--&--\\
	\hline
	\end{tabular}
		\vspace{0.2cm}
		\caption{Estimated probability values $P(\epsilon,T)$ for the stochastic Allen-Cahn equation with Dirichlet boundary conditions. The values reported are based on standard Monte Carlo simulation without employing some change of measure.}
		\label{Dirprobsmcnew}
   	\end{table}
   	
   	\begin{table}[H]
  	\begin{tabular}{ |c| c | c | c | c | c|c|}
  \hline
  $\epsilon$&$T=1$&$T=2$ &$T=3$&$T=4$&$T=6$&$T=8$\\ \hline
  $0.00008$&$16.5$ &$6.3$ &$4.4$ &$3.5$&$2.6$&$2.2$
  \\ \hline
  $0.00001$ &$100.0$&$21.2$&$12.7$&$10.4$&$7.8$&$6.5$
  \\ \hline
  $0.000004$&$223.6$&$47.7$&$23.3$&$19.0$&$14.4$&$11.7$\\
  \hline
  $0.000001$&$223.6$&$158.1$&$91.3$&$67.4$&$47.7$&$39.5$\\
  \hline
  $0.0000008$&--&$158.1$&$158.1$&$84.5$&$55.9$&$41.5$\\
  \hline
  $0.0000004$&--&--&$223.6$&--&$158.1$&$111.8$\\
  \hline
  $0.0000002$&--&--&--&$223.6$&--&$223.6$\\
  \hline
  $0.0000001$&--&--&--&--&--&--\\
  \hline
  $0.00000008$ &--&--&--&--&--&--\\
  \hline
\end{tabular}
\vspace{0.2cm}
\caption{Estimated relative errors per sample for the stochastic Allen-Cahn equation with Dirichlet boundary conditions. The values reported are based on standard Monte Carlo simulation without employing some change of measure.}
\label{Direrrorsmcnew}
\end{table}

\subsection{Numerical results for the quintic SRDE \eqref{quinticsrde} with Neumann boundary conditions}

In this section, we provide numerical simulation results validating our theory for the for the quintic SRDE \eqref{quinticsrde} with Neumann boundary conditions  studied in Subsection \ref{SS:Ginzburg-Landau}.

\begin{table}[H]
	\begin{tabular}{ |c|c | c | c | c | c | c|c|}
		\hline
		$\epsilon$&$R=\sqrt{\epsilon}h(\epsilon)$&$T=1$&$T=2$ &$T=3$&$T=4$&$T=6$&$T=8$\\ \hline
		$0.008$ &$0.144956$&$3.87e-03$ &$1.08e-02$ &$1.77e-03$ &$2.24e-02$&$3.84e-02$&$5.27e-02$
		\\ \hline
		$0.003$ &$0.097915$&$6.22e-04$&$1.81e-03$&$2.98e-03$&$4.18e-03$&$6.43e-03$&$9.04e-03$
		\\ \hline
		$0.002$ &$0.083255$&$2.62e-04$&$7.66e-04$&$1.30e-03$&$1.79e-03$&$2.86e-03$&$3.85e-03$\\
		\hline
		$0.0008$&$0.057708$&$2.96e-05$&$8.99e-05$&$1.47e-04$&$2.09e-04$&$3.30e-04$&$4.47e-04$\\
		\hline
		$0.0006$&$0.051435$&$1.38e-05$&$4.16e-05$&$6.97e-05$&$9.90e-05$&$1.55e-04$&$2.09e-04$\\
		\hline
		$0.0002$&$0.033145$&$4.68e-07$&$1.50e-06$&$2.56e-06$&$3.57e-06$&$5.71e-06$&$7.62e-06$\\
		\hline
		$0.00006$&$0.020477$&$4.71e-09$&$1.60e-08$&$2.73e-08$&$3.89e-08$&$6.09e-08$&$8.41e-08$\\
		\hline
		$0.00002$&$0.013195$&$2.42e-11$&$8.84e-11$&$1.52e-10$&$2.16e-10$&$3.41e-10$&$4.82e-10$\\
		\hline
		$0.000008$&$0.009146$ &$1.16e-13$&$4.47e-13$&$7.72e-13$&$1.10e-12$&$1.80e-12$&$2.45e-12$\\
		\hline
	\end{tabular}
\vspace{0.2cm}
\caption{Estimated probabilities for the quintic SRDE \eqref{quinticsrde} with Neumann boundary conditions and $\mu=-0.5$ using the developed importance sampling scheme with $\kappa=0.999$ and mollification parameter $\delta:=2/h^2(\epsilon).$ The other parameters are $h(\epsilon)=\epsilon^{-0.1}, \ell=1, x^*=1.$}
\label{quinticimportanceprobnew}
\end{table}

\begin{table}[H]
	\begin{tabular}{ |c| c | c | c | c | c|c|}
		\hline
		$\epsilon$&$T=1$&$T=2$ &$T=3$&$T=4$&$T=6$&$T=8$\\ \hline
		$0.008$ &$2.1$ &$1.6$ &$1.6$ &$2.1$&$2.5$&$4.4$
		\\ \hline
		$0.003$ &$2.3$&$1.7$&$1.6$&$1.7$&$2.4$&$5.1$
		\\ \hline
		$0.002$ &$2.3$&$1.6$&$1.5$&$1.6$&$2.3$&$3.4$\\
		\hline
		$0.0008$&$2.4$&$1.6$&$1.4$&$1.5$&$2.1$&$3.0$\\
		\hline
		$0.0006$&$2.4$&$1.6$&$1.4$&$1.5$&$2.3$&$2.9$\\
		\hline
		$0.0002$&$2.4$&$1.5$&$1.3$&$1.4$&$1.9$&$2.7$\\
		\hline
		$0.00006$&$2.4$&$1.4$&$1.2$&$1.3$&$1.8$&$2.5$\\
		\hline
		$0.00002$&$2.4$&$1.3$&$1.2$&$1.2$&$1.7$&$2.7$\\
		\hline
		$0.000008$ &$2.4$&$1.3$&$1.1$&$1.2$&$1.6$&$2.5$\\
		\hline
	\end{tabular}
	\vspace{0.2cm}
\caption{Estimated relative errors per sample for the quintic SRDE \eqref{quinticsrde} with Neumann boundary conditions and $\mu=-0.5$ using the developed importance sampling scheme with $\kappa=0.999$ and mollification parameter $\delta:=2/h^2(\epsilon).$ The rest of the parameters are $h(\epsilon)=\epsilon^{-0.1}, \ell=1, x^*=1.$}
\label{quinticimportanceerrornew}
\end{table}

\begin{table}[H]
		\begin{tabular}{ |c| c|c | c | c | c | c|c|}
			\hline
			$\epsilon$&$R=\sqrt{\epsilon}h(\epsilon)$&$T=1$&$T=2$ &$T=3$&$T=4$&$T=6$&$T=8$\\ \hline
			$0.008$ &$0.144956$&$3.60e-03$ &$1.13e-02$ &$1.83e-02$ &$2.44e-02$&$3.83e-02$&$5.11e-02$\\ \hline
			$0.003$ &$0.097915$&$5.80e-04$&$1.72e-03$&$2.38e-03$&$3.76e-03$&$6.76e-03$&$8.08e-03$
			\\ \hline
			$0.002$ &$0.083255$&$2.40e-04$&$7.60e-04$&$1.26e-03$&$1.90e-03$&$3.02e-03$&$3.56e-03$\\
			\hline
			$0.0008$&$0.057708$&$8.00e-05$&$1.20e-04$&$1.60e-04$&$2.80e-04$&$1.80e-04$&$5.60e-04$\\
			\hline
			$0.0006$&$0.051435$&$2.00e-05$&$8.00e-05$&$4.00e-05$&$1.60e-04$&$1.40e-04$&$1.60e-04$\\
			\hline
			$0.0002$&$0.033145$&--&--&--&$2.00e-05$&--&--\\
			\hline
			$0.00006$&$0.020477$&--&--&--&--&--&--\\
			\hline
			$0.00002$&$0.013195$&--&--&--&--&--&--\\
			\hline
			$0.000008$ &$0.009146$&--&--&--&--&--&--\\
			\hline
		\end{tabular}
	\vspace{0.2cm}
	\caption{Estimated probability values for the quintic SRDE \eqref{quinticsrde} with Neumann boundary conditions and $\mu=-0.5$. The values reported are based on standard Monte Carlo simulation without employing some change of measure.}
		\label{quinticSMCprobnew}
	\end{table}

\begin{table}[H]
		\begin{tabular}{ |c| c | c | c | c | c|c|}
			\hline
			$\epsilon$&$T=1$&$T=2$ &$T=3$&$T=4$&$T=6$&$T=8$\\ \hline
			$0.008$ &$16.6$ &$9.4$ &$7.3$ &$6.3$&$5.0$&$4.3$
			\\ \hline
			$0.003$ &$41.5$&$24.1$&$20.5$&$16.3$&$12.1$&$11.1$
			\\ \hline
			$0.002$ &$64.5$&$36.3$&$28.2$&$22.9$&$18.2$&$16.7$\\
			\hline
			$0.0008$&$111.8$&$91.3$&$79.1$&$59.8$&$74.5$&$42.2$\\
			\hline
			$0.0006$&$223.6$&$111.8$&$158.1$&$79.1$&$84.5$&$79.1$\\
			\hline
			$0.0002$&--&--&--&$223.6$&--&--\\
			\hline
			$0.00006$&--&--&--&--&--&--\\
			\hline
			$0.00002$&--&--&--&--&--&--\\
			\hline
			$0.000008$ &--&--&--&--&--&--\\
			\hline
		\end{tabular}
	\vspace{0.2cm}
	\caption{Estimated relative errors per sample for the quintic SRDE \eqref{quinticsrde} with Neumann boundary conditions and $\mu=-0.5$. The values reported are based on standard Monte Carlo simulation without employing some change of measure.}
		\label{quinticSMCerrornew}
	\end{table}

\subsection{Numerical comparisons of relative errors and probabilities for different parameter values}

 In this section, we provide numerical simulation results validating our theory for the stochastic Allen-Cahn equation with Neumann  boundary conditions studied in Subsection \ref{Neumannsec}. In particular, we now explore the effect of  different moderate deviation scalings $h(\epsilon)$ and of different Galerkin projection levels $N$.

\begin{table}[H]
	\begin{tabular}{ |c | c | c |c|c|c|}
		\hline
		$\epsilon/T=2,h(\epsilon)=\epsilon^{-0.1}$&Prob.&Rel. error/sample &$\epsilon/T=2,h(\epsilon)=\epsilon^{-0.2}$&Prob.&Rel. error/sample \\ \hline
		$0.01$ &$5.43e-02$& $1.2$ &$0.08$ &$8.11e-02$& $1.7$
		\\ \hline
		
		$0.004$ &$2.01e-02$&$1.3$	&$0.05$ &$3.35e-02$&$2.1$
		\\ \hline
		
		$0.002$ &$8.35e-03$& $1.4$&	$0.03$ &$9.85e-03$& $2.5$\\
		\hline
		$0.0008$&$2.09e-03$&$1.5$	&$0.01$&$2.01e-04$&$4.3$\\
		\hline
		$0.0004$&$6.21e-04$&$1.5$	&	$0.008$&$6.84e-05$&$3.8$\\
		\hline
		$0.0001$&$2.92e-05$&$1.6$	&	$0.006$&$1.44e-05$&$5.1$\\
		\hline
		$0.00006$&$7.63e-06$&$1.6$&$0.004$&$1.23e-06$&$9.2$\\
		\hline
		$0.000008$&$7.45e-09$&$ 1.7$&	$0.002$&$4.64e-09$&$5.0$\\
		\hline
		$0.000004$&$3.30e-10$& $1.7$&$0.001$&$2.97e-12$&$5.4$\\ \hline
	\end{tabular}
	\vspace{0.2cm}
	\caption{Comparison of relative errors and probabilities produced by the importance sampling scheme for the Neumann Allen-Cahn equation with different moderate deviation scalings $h(\epsilon)$. The rest of the parameters are $x^*=1,\ell=1,\kappa=0.9, T=2.$  }
	\label{hdifference}
\end{table}

\begin{table}[H]
	\begin{center}
		\begin{tabular}{ |c | c | c |c|}
			\hline
			$\epsilon/T=3$&$N=50$&$N=100$ &$N=150$ \\ \hline
			
			$0.01$ &$1.0$&$1.0$	&$1.0$
			\\ \hline
			
			$0.004$ &$1.0$& $1.0$&$1.0$	\\
			\hline
			$0.002$&$1.1$&$1.1$	&$1.1$\\
			\hline
			$0.0008$&$1.1$&$1.1$	&$1.1$	\\
			\hline
			$0.0004$&$1.1$&$1.1$	&$1.1$	\\
			\hline
			$0.0001$&$1.2$&$1.1$&$1.2$\\
			\hline
			$0.00006$&$1.2$&$1.2$&$1.2$	\\
			\hline
			$0.000008$&$1.2$&$1.2$	&$1.2$	\\
			\hline
			$0.000004$ &$1.3$ & $1.2$ &$1.2$
			\\ \hline
		\end{tabular}
	\end{center}
	\vspace{0.2cm}
	\caption{Comparison of the relative errors produced by the importance sampling scheme for the Neumann Allen-Cahn equation with different Galerkin projection levels $N$. The rest of the parameters are $x^*=1,\ell=1,\kappa=0.9,h(\epsilon)=\epsilon^{-0.1}, T=3.$       }
	\label{Ntable}
\end{table}

\section{Conclusions and Future Work}
In this paper we studied the problem of rare event simulation for small-noise SRDEs via moderate deviation-based importance sampling. Taking advantage of the linearized limiting dynamics of the process $\hat{\eta}^{\epsilon,v}$ \eqref{etamild}, we constructed changes of measure that behave optimally in the limit as $\epsilon\to 0$ under the fairly general spectral gap condition of Hypothesis \ref{A3c'}. Working under the more restrictive Hypothesis \ref{A3c} we designed an importance sampling scheme with changes of measure that act on a one-dimensional eigenspace of the operator $A+DF(x^*).$ We were then able to show that this scheme performs well pre-asymptotically and supplemented the theoretical results with numerical simulations for gradient-type SRDEs corresponding to a double-well potential. Such systems have wide applicability and provided good examples to illustrate our theory. Nevertheless, there are other types of nonlinearities which satisfy our assumptions, e.g. $f=-V'_f,$ where the potential $V_f(x)=\sin x$ has more than two global minima.

The design and pre-asymptotic analysis of a scheme under the weaker spectral gap of Hypothesis \ref{A3c'} provides an interesting  direction for future work. This would allow for the simulation of rare events for SRDEs under bifurcation (e.g. when $\ell>\pi$ in the Neumann Allen-Cahn case). The asymptotic optimality of such a scheme is guaranteed by Theorem \ref{Asymptoticthm}. Even though the presence of non-constant saddle points with one unstable direction facilitates exits from $D$ \eqref{exitnhood}, the pre-asymptotic analysis of Section \ref{Sec4} is expected to be more complicated in this setting. This is due to the fact that the changes of measure $u_{k_0}$ \eqref{uchoice} act on $k_0-$ dimensional subspaces of $\h.$ One then has to show that the linearization error is negligible by considering the behavior of the system on carefully chosen partitions of a $k_0$-dimensional section of $D.$

Throughout this work we have considered SRDEs in one spatial dimension. In higher dimensions, equations like \eqref{AllenCahn} are singular and a-priori ill-posed. Thus, one has to consider SRDEs with a  spatially colored stochastic forcing or employ renormalization techniques. Metastability results for the renormalized two-dimensional Allen-Cahn can be found e.g. in \cite{berglund2017eyring}, \cite{tsatsoulis2020exponential} and references therein. Importance sampling for linear equations (i.e. $f=0$) with colored noise has been considered in \cite{salins2017rare}. In the latter, the spatial covariance operator $Q$ is assumed to be trace-class and diagonalizable with respect to the eigenbasis of the differential operator $A.$ Carrying the analysis of this paper over to higher spatial dimensions is challenging since $A$ and the linearized operator $A+DF(x^*)$ do not necessarily have the same eigenbasis (e.g. in the case of the Allen-Cahn with Dirichlet boundary conditions in spatial dimension $2$). In particular, the analysis of the exit direction in Section \ref{Sec4} would have to be generalized and take into account the non-commutativity of $Q$ and $A+DF(x^*)$.

Finally, we expect that the results of this paper can be used to design importance sampling schemes for simulating rare events  in slow-fast systems of SRDEs. Similar work for multiscale diffusions in finite dimensions has been done in \cite{spiliopoulos2020importance} and an MDP for multiscale SRDEs was recently proved in \cite{gasteratos2022moderate}.

\appendix
\section{}\label{AppA}

\subsection{Proof of Lemma \ref{preasymptoticlem}}

A formal application of It\^o's formula to $Z_{x^*}\big(t,\hat{\eta}_{x^*}^{\epsilon, v}(t)\big)$ and Taylor's theorem to $F$ yield
\begin{equation*}
\begin{aligned}
Z_{x^*}\big(\hat{\tau}^{\epsilon,v}_{x^*},\hat{\eta}_{x^*}^{\epsilon, v}&(\hat{\tau}^{\epsilon,v}_{x^*})\big)-Z_{x^*}\big(0,0\big)\\&=\int_{0}^{\hat{\tau}^{\epsilon,v}_{x^*}}\partial_t Z_{x^*}\big(s,\hat{\eta}_{x^*}^{\epsilon, v}(s)\big)ds+\int_{0}^{\hat{\tau}^{\epsilon,v}_{x^*}}\blangle D_\eta Z_{x^*}\big(s,\hat{\eta}_{x^*}^{\epsilon, v}(s)\big), [A+ DF(x^*)]\hat{\eta}_{x^*}^{\epsilon, v}(s)\brangle_\h ds\\&+\frac{\sqrt{\epsilon}h(\epsilon)}{2}\int_{0}^{\hat{\tau}^{\epsilon,v}_{x^*}}\blangle D_\eta Z_{x^*}\big(s,\hat{\eta}_{x^*}^{\epsilon, v}(s)\big), D^2F\big(x^*+\theta_0\sqrt{\epsilon}h(\epsilon)\hat{\eta}^{\epsilon,v}_{x^*}(s)  \big)\big(\hat{\eta}^{\epsilon,v}_{x^*}(s), \hat{\eta}^{\epsilon,v}_{x^*}(s)          \big) \brangle_\h ds\\&
+\int_{0}^{\hat{\tau}^{\epsilon,v}_{x^*}}\blangle D_\eta Z_{x^*}\big(s,\hat{\eta}_{x^*}^{\epsilon, v}(s)\big), v(s)
-u\big(s,\hat{\eta}^{\epsilon,v}_{x^*}(s)\big )\brangle_\h ds\\&
+\frac{1}{2h^2(\epsilon)}\int_{0}^{\hat{\tau}^{\epsilon,v}_{x^*}}\text{tr}\big[ D^2_\eta Z_{x^*}\big(s,\hat{\eta}_{x^*}^{\epsilon, v}(s)\big) \big]ds+\frac{1}{h(\epsilon)}\int_{0}^{\hat{\tau}^{\epsilon,v}_{x^*}}\blangle  D_\eta Z_{x^*}\big(s,\hat{\eta}_{x^*}^{\epsilon, v}(s)\big), dW(s) \brangle_\h\\&
\geq \int_{0}^{\hat{\tau}^{\epsilon,v}_{x^*}}\bigg[\frac{1}{2}\big\|u\big(s,\hat{\eta}^{\epsilon,v}_{x^*}(s)\big)\big\|^2_{\h}-\frac{1}{4}\|v(s)\|^2_\h\bigg]ds+\frac{1}{h(\epsilon)}\int_{0}^{\hat{\tau}^{\epsilon,v}_{x^*}}\blangle  D_\eta Z_{x^*}\big(s,\hat{\eta}_{x^*}^{\epsilon, v}(s)\big), dW(s) \brangle_\h
\\&+\frac{\sqrt{\epsilon}h(\epsilon)}{2}\int_{0}^{\hat{\tau}^{\epsilon,v}_{x^*}}\blangle D_\eta Z_{x^*}\big(s,\hat{\eta}_{x^*}^{\epsilon, v}(s)\big), D^2F\big(x^*+\theta_0\sqrt{\epsilon}h(\epsilon)\hat{\eta}^{\epsilon,v}_{x^*}(s)  \big)\big(\hat{\eta}^{\epsilon,v}_{x^*}(s), \hat{\eta}^{\epsilon,v}_{x^*}(s)          \big) \brangle_\h ds\\&
+\int_{0}^{\hat{\tau}^{\epsilon,v}_{x^*}}\bigg\{\partial_t Z_{x^*}\big(s,\hat{\eta}_{x^*}^{\epsilon, v}(s)\big)+ \mathbb{H}_{x^*}\big(\hat{\eta}_{x^*}^{\epsilon, v}(s), D_\eta Z_{x^*}(s,\hat{\eta}_{x^*}^{\epsilon, v}(s)) \big)+   \frac{1}{2h^2(\epsilon)}\text{tr}\big[ D^2_\eta Z_{x^*}\big(s,\hat{\eta}_{x^*}^{\epsilon, v}(s)\big) \big]\bigg\} ds\\&
-\frac{1}{2}\int_{0}^{\hat{\tau}^{\epsilon,v}_{x^*}}\big\|  D_\eta Z_{x^*}(s,\hat{\eta}_{x^*}^{\epsilon, v}(s))- D_\eta U_{x^*}(s,\hat{\eta}_{x^*}^{\epsilon, v}(s)) \big\|^2_\h ds,
\end{aligned}
\end{equation*}
where $\theta_0\in(0,1)$ and we used \eqref{minmaxHam} to obtain the inequality above. Taking expectation then yields
\begin{equation*}
\begin{aligned}
\ex\int_{0}^{\hat{\tau}^{\epsilon,v}_{x^*}}\bigg[\frac{1}{2}\|v(s)\|^2_\h-\big\|u\big(s,\hat{\eta}^{\epsilon,v}_{x^*}(s)\big)\big\|^2_{\h}\bigg]ds&\geq  2Z_{x^*}\big(0,0\big)- 2\ex Z_{x^*}\big(\hat{\tau}^{\epsilon,v}_{x^*},\hat{\eta}_{x^*}^{\epsilon, v}(\hat{\tau}^{\epsilon,v}_{x^*})\big)\\&+2\ex\int_{0}^{\hat{\tau}^{\epsilon,v}_{x^*}}\mathfrak{H}_{x^*}^{\epsilon}( U_{x^*}, Z_{x^*})\big(s,\hat{\eta}^{\epsilon,v}_{x^*}(s)\big) ds.
\end{aligned}
\end{equation*}
In view of \eqref{varrep} we conclude that
\begin{equation*}
\begin{aligned}
    -\frac{1}{h^2(\epsilon)}\log Q^{\epsilon}(u^\epsilon)&\geq\inf_{v\in\mathcal{A}}\bigg[  2Z_{x^*}\big(0,0\big)- 2\ex Z_{x^*}\big(\hat{\tau}^{\epsilon,v}_{x^*},\hat{\eta}_{x^*}^{\epsilon, v}(\hat{\tau}^{\epsilon,v}_{x^*})\big)+2\ex\int_{0}^{\hat{\tau}^{\epsilon,v}_{x^*}}\mathfrak{H}_{x^*}^{\epsilon}( U_{x^*}, Z_{x^*})\big(s,\hat{\eta}^{\epsilon,v}_{x^*}(s)\big) ds    \bigg].
\end{aligned}
\end{equation*}
The proof is complete.
\subsection{Proof of Lemma \ref{mainestimate1lem}} From the chain rule we compute
$$D_\eta Z(t,\eta)=(1-\zeta)D_\eta U^\delta(t,\eta)=\frac{ (1-\zeta)e^{-\frac{F_1(\eta)}{\delta}}}{  e^{-\frac{F_1(\eta)}{\delta}} +    e^{-\frac{F^\epsilon_2}{\delta}}   } D_\eta F_1(\eta)=-\frac{ a_1^f(1-\zeta)e^{-\frac{F_1(\eta)}{\delta}}}{  e^{-\frac{F_1(\eta)}{\delta}} +    e^{-\frac{F^\epsilon_2}{\delta}}}e_1^f=-(1-\zeta)a_1^f\rho^\epsilon(\eta)2e_1^f\langle e_1^f, \eta\rangle_\h.$$
Moreover,
\begin{equation*}
\begin{aligned}
D^2_\eta Z(t,\eta)&=-2(1-\zeta)a_1^f\bigg[\rho^\epsilon(\eta)e_1^f\otimes e^f_1+\frac{1}{\delta}\big(\rho^\epsilon(\eta)\big)^2\big(1-\rho^\epsilon(\eta)\big)2\langle e_1^f, \eta\rangle_\h e_1^f\otimes e_1^f\bigg]\\&
=-2(1-\zeta)a_1^f\rho^\epsilon(\eta)\bigg[1+\frac{2}{\delta}\rho^\epsilon(\eta)\big(1-\rho^\epsilon(\eta)\big)\langle e_1^f, \eta\rangle_\h \bigg]e_1^f\otimes e_1^f
\end{aligned}
\end{equation*}
and
\begin{equation*}
\begin{aligned}
\frac{1}{2h^2(\epsilon)}\text{tr}\big[D^2_\eta Z(t,\eta)\big]&=-\frac{(1-\zeta)a_1^f\rho^\epsilon(\eta)}{h^2(\epsilon)}\bigg[1+\frac{2}{\delta}\rho^\epsilon(\eta)\big(1-\rho^\epsilon(\eta)\big)\langle e_1^f, \eta\rangle_\h \bigg]\sum_{k\in\N_0}\blangle (e_1^f\otimes e_1^f)e_k^f, e_k^f\brangle_\h\\&
=-\frac{(1-\zeta)a_1^f\rho^\epsilon(\eta)}{h^2(\epsilon)}\bigg[1+\frac{2}{\delta}\rho^\epsilon(\eta)\big(1-\rho^\epsilon(\eta)\big)\langle e_1^f, \eta\rangle_\h \bigg].
\end{aligned}
\end{equation*}
In view of \eqref{Hamiltonian2} we have
\begin{equation*}
\begin{aligned}
\mathbb{H}_{x^*}\big(\eta,D_\eta Z(t,\eta)\big)&=-a_1^f\blangle \eta, D_\eta Z(t,\eta)\brangle_\h-\frac{1}{2}\big\|D_\eta Z(t,\eta)\big\|_\h^2\\&
=   2(1-\zeta)(a_1^f)^2\rho^\epsilon(\eta)\langle e_1^f, \eta\rangle^2_\h-\frac{4(1-\zeta)^2(a_1^f)^2\big(\rho^\epsilon(\eta)\big)^2}{2}\langle e_1^f, \eta\rangle^2_\h\|e_1^f\|_\h^2\\&
=2(1-\zeta)(a_1^f)^2\rho^\epsilon(\eta)\langle e_1^f, \eta\rangle^2_\h\big[1-(1-\zeta)\rho^\epsilon(\eta)      \big].
\end{aligned}
\end{equation*}
Finally,
\begin{equation*}
\begin{aligned}
-\frac{1}{2}\|D_\eta Z(t,\eta)- D_\eta U^\delta(t,\eta)\|^2_\h=-\frac{\zeta^2}{2}\|D_\eta U^\delta(t,\eta)\|^2_\h=-2\zeta^2(a_1^f)^2\big(\rho^\epsilon(\eta)\big)^2\langle e_1^f, \eta\rangle^2_\h.
\end{aligned}
\end{equation*}
In view of these computations we obtain
\begin{equation*}
\begin{aligned}
\mathfrak{H}_{x^*}^{\epsilon}( U^\delta, Z)&(t,\eta)=\mathscr{H}_{x^*}^{\epsilon}( Z)(t,\eta)+\mathscr{D}_{x^*}^{\epsilon}(Z)(t,\eta)-\frac{1}{2}\|D_\eta Z(t,\eta)- D_\eta U^\delta(t,\eta)\|^2_\h\\&
= \frac{\sqrt{\epsilon}h(\epsilon)}{2}\blangle D_\eta Z\big(t,\eta\big), D^2F\big(x^*+\theta_0\sqrt{\epsilon}h(\epsilon)\eta \big)\big(\eta,\eta\big)\brangle_\h+ \partial_t Z(t,\eta)+\mathbb{H}_{x^*}\big(\eta,D_\eta Z(t,\eta)\big)\\&+\frac{1}{2h^2(\epsilon)}\text{tr}\big[ D^2_\eta Z(t,\eta)\big]-\frac{1}{2}\|D_\eta Z(t,\eta)- D_\eta U(t,\eta)\|^2_\h\\&
=-\sqrt{\epsilon}h(\epsilon)(1-\zeta)a_1^f\rho^\epsilon(\eta)\langle e_1^f, \eta\rangle_\h\blangle  e_1^f , D^2F\big(x^*+\theta_0\sqrt{\epsilon}h(\epsilon)\eta \big)\big(\eta,\eta\big)\brangle_\h+0\\&
+2(1-\zeta)(a_1^f)^2\rho^\epsilon(\eta)\langle e_1^f, \eta\rangle^2_\h\big[1-(1-\zeta)\rho^\epsilon(\eta)      \big]\\&
-\frac{(1-\zeta)a_1^f\rho^\epsilon(\eta)}{h^2(\epsilon)}\bigg[1+\frac{2}{\delta}\rho^\epsilon(\eta)\big(1-\rho^\epsilon(\eta)\big)\langle e_1^f, \eta\rangle_\h \bigg]\\&
-2\zeta^2(a_1^f)^2\big(\rho^\epsilon(\eta)\big)^2\langle e_1^f, \eta\rangle^2_\h.
\end{aligned}
\end{equation*}
The proof is complete.

\subsection{Proof of Lemma \ref{linearization2lem}}\label{linearization2lemproof} In view of the last term in \eqref{mainestimate1} we have
	\begin{equation*}
	\begin{aligned}
	\blangle D_\eta Z\big(t,\eta\big)&, D^2F\big(x^*+\theta_0\sqrt{\epsilon}h(\epsilon)\eta \big)\big(\eta,\eta\big) \brangle_\h\\&
	\geq-\sqrt{\epsilon}h(\epsilon)2(1-\zeta)a_1^f\rho^\epsilon(\eta)\big|\langle e_1^f, \eta\rangle_\h\big| \langle e_1^f, D^2F\big(x^*+\theta_0\sqrt{\epsilon}h(\epsilon)\eta \big)\big(\eta,\eta\big)\rangle_\h\big|\\&
	\geq -\sqrt{\epsilon}h(\epsilon)2(1-\zeta)a_1^f\rho^\epsilon(\eta)\big|\langle e_1^f, \eta\rangle_\h\big|\|e_1^f\|_{\h}\|\eta^2\|_{L^\infty}\big\| \partial_x^2f\big(x^*+\theta_0\sqrt{\epsilon}h(\epsilon)\eta \big)\big\|_{\h} 	
	\\&	\geq-C_f\sqrt{\epsilon}h(\epsilon)2(1-\zeta)a_1^f\rho^\epsilon(\eta)\big|\langle e_1^f, \eta\rangle_\h\big|\bigg(1+\big\|(x^*+\theta_0\sqrt{\epsilon}h(\epsilon)\eta)^{p_0-2}\|_{\h})\bigg)
	\\&\geq -C_{f,\ell}\sqrt{\epsilon}h(\epsilon)2(1-\zeta)a_1^f\rho^\epsilon(\eta)\big|\langle e_1^f, \eta\rangle_\h\big|\bigg(1+\big\|x^*\big\|^{p_0-2}_{L^\infty}+(\theta_0)^{p_0-2}\|\sqrt{\epsilon}h(\epsilon)\eta\|^{p_0-2}_{L^\infty}\bigg),         	
	\end{aligned}
	\end{equation*}
	where we used the growth assumption \eqref{fgrowth} in the third inequality above. The estimate follows since $\|\eta\|_{L^\infty}\leq 1$ and $\theta_0\sqrt{\epsilon}h(\epsilon)<1$.

\subsection{Proof of Lemma \ref{region1}}
		In the region $B_1^f$ we have
		\begin{equation*}
		\begin{aligned}
		\rho^\epsilon\big(\eta\big)&
		=\frac{ e^{-\frac{F_1(\eta)-F_2^\epsilon}{\delta}}}{ e^{-\frac{F_1(\eta)-F_2^\epsilon}{\delta}} +  1}
		\end{aligned}
		\end{equation*}
		and $  e^{-\frac{F_1(\eta)-F_2^\epsilon}{\delta}}\leq  e^{-a_1^fh^{2}(\epsilon)h^{-2\kappa}(\epsilon)(1-h^{-2\alpha}(\epsilon))/2}.$
		For any reasonable choice of $h(\epsilon),$ (e.g. continuous, positive, strictly decreasing)  the latter implies that the second term in \eqref{mainestimate3} is exponentially negligible. Ignoring the nonnegative terms, note that the linearization error (last term) is bounded from above in absolute value by
		$$C_f\sqrt{\epsilon}h(\epsilon)\|\eta\|_{\h}\rho^\epsilon(\eta)\leq  C_{f,L}e^{-ch^2(\epsilon)}\leq Ch^{-2}(\epsilon)e^{-ch^2(\epsilon)}$$
		which holds by assuming that $\sqrt{\epsilon}h(\epsilon)<1$ and applying the inequality $e^{-cx}\leq (\beta/e)^\beta \gamma^{-\beta}x^{-\beta},$ which holds for all $\beta,\gamma,x>0,$ with $\gamma=c,x=h^2(\epsilon)$ and $\beta=1.$  The proof is complete.
		
	\subsection{Proof of Lemma \ref{region3}}
			$(i)$ Our choice of $K$ implies that $e^{-K}/(e^{-K}+1)=3/4.$ Hence  in this region we have $\rho^\epsilon\big(\eta\big)\geq 3/4.$ From \eqref{mainestimate3} and the fact that $\rho^\epsilon>(\rho^\epsilon)^2$ we have
			\begin{equation*}
			\begin{aligned}
			\mathfrak{H}_{x^*}^{\epsilon}(U^\delta,Z)(t,\eta)&
			\geq\rho^\epsilon(\eta)^2(a_1^f)^2\langle \eta, e_1^f\rangle^2\bigg( (1-\zeta)(1-\rho^\epsilon(\eta))+2(\zeta-2\zeta^2)     \bigg)
			-\frac{1}{h^2(\epsilon)}(1-\zeta)\rho^\epsilon(\eta)a_1^f
			\\&  - 2C\sqrt{\epsilon}h(\epsilon)(1-\zeta)a_1^f\rho^\epsilon(\eta)\big|\langle e_1^f, \eta\rangle_\h\big|
			\\&\geq 2(\zeta_0-2\zeta_0^2)\rho^\epsilon(\eta)^2(a_1^f)^2\langle \eta, e_1^f\rangle^2
			-\frac{1}{h^2(\epsilon)}(1-\zeta)\rho^\epsilon(\eta)a_1^f
			\\&- 2C_{L}\sqrt{\epsilon}h(\epsilon)(1-\zeta)a_1^f\rho^\epsilon(\eta)
			\\&
			\geq 2(\zeta_0-2\zeta_0^2)\frac{9}{16}(a_1^f)^2\langle \eta, e_1^f\rangle^2
			-\frac{1}{h^2(\epsilon)}(1-\zeta)a_1^f- 2C_{L}\sqrt{\epsilon}h(\epsilon)(1-\zeta)a_1^f
			\\&
			\geq\frac{18}{16}(\zeta_0-2\zeta_0^2)(a_1^f)^2\bigg( 2h(\epsilon)^{-2\kappa}-h(\epsilon)^{-2}K  \bigg) -\frac{1}{2h^2(\epsilon)}a_1^f- C_{L}\sqrt{\epsilon}h(\epsilon)a_1^f
			\\&\geq -\frac{18}{16}(\zeta_0-2\zeta_0^2)(a_1^f)^2h(\epsilon)^{-2}K  - C_{L}\sqrt{\epsilon}h(\epsilon)a_1^f\geq - C_{L}\sqrt{\epsilon}h(\epsilon)a_1^f,
			\end{aligned}
			\end{equation*}
			where we used that $\zeta\geq \zeta_0$ and the Cauchy-Schwarz inequality to obtain the second inequality, $\rho^{\epsilon}\in[3/4,1)$ for the third, $\zeta<1/2$ and $2h(\epsilon)^{-2\kappa}-h(\epsilon)^{-2}K <\langle \eta, e_1^f\rangle_\h^2$ in the fourth, and $h^{2(\kappa-1)}(\epsilon)<9a_1^f(\zeta_0-2\zeta_0^2)/2,$ $K<0$ in the last line.\\
			$(ii)$ If $h(\epsilon)$ is such that $\lim_{\epsilon\to 0}\sqrt{\epsilon}h^3(\epsilon)=0,$ the estimate follows from the first inequality in the last line of the previous display. In particular, we can write
			\begin{equation*}
			\begin{aligned}
			\mathfrak{H}_{x^*}^{\epsilon}(U^\delta,Z)(t,\eta)&
			\geq\sqrt{\epsilon}h(\epsilon)\bigg(-\frac{18}{16}(\zeta_0-2\zeta_0^2)(a_1^f)^2\frac{K}{\sqrt{\epsilon}h^3(\epsilon)}  - C_{L}a_1^f\bigg)
			\end{aligned}
			\end{equation*}
			and the estimate follows by letting $\epsilon$ be small enough to have  $\sqrt{\epsilon}h^3(\epsilon)<-K\frac{18}{16}(\zeta_0-2\zeta_0^2)(a_1^f)/C_L$ (recall that $-K>0$.)
			
				\subsection{Proof of Lemma \ref{region2}}\label{region2proof}
					\noindent $(i)$ \underline{Case 1}: Assume that $\eta$ is such that $\rho^\epsilon(\eta)\geq 1/2$. Then the arguments and estimates of Lemma \ref{region3} carry over verbatim.
					The only difference lies in the range of $\epsilon.$\\
					\noindent \underline{Case 2}: Assume that $\eta$ is such that $\rho^\epsilon(\eta)< 1/2$. Focusing on the first two terms in \eqref{mainestimate3} we see that
					\begin{equation*}
					\begin{aligned}
					\mathfrak{H}_{x^*}^{\epsilon}(U^\delta,Z)(t,\eta)&
					\geq (1-\zeta)\frac{\rho^\epsilon(\eta)}{2}(a_1^f)^2\langle \eta, e_1^f\rangle^2-\frac{1}{h^2(\epsilon)}(1-\zeta)\rho^\epsilon(\eta)a_1^f\\&+2(\zeta-2\zeta^2)(\rho^\epsilon(\eta))^2(a_1^f)^2\langle \eta, e_1^f\rangle^2-2C\sqrt{\epsilon}h(\epsilon)(1-\zeta)a_1^f\rho^\epsilon(\eta)\big|\langle e_1^f, \eta\rangle_\h\big|\\&
					\geq \rho^\epsilon(\eta)(1-\zeta)a_1^f\bigg(\frac{a_1^f}{2}\langle \eta, e_1^f\rangle^2-\frac{1}{h^2(\epsilon)}\bigg)\\&
					+2(\zeta-2\zeta^2)\rho^\epsilon(\eta)^2(a_1^f)^2\langle \eta, e_1^f\rangle^2-2C_{L}\sqrt{\epsilon}h(\epsilon)(1-\zeta)a_1^f\rho^\epsilon(\eta)^2\\&
					\geq \rho^\epsilon(\eta)(1-\zeta)a_1^f\bigg(   \frac{a_1^f}{2}h(\epsilon)^{-2(\kappa+\alpha)}  -\frac{1}{h^2(\epsilon)}  \bigg)\\&
					+2(\zeta-2\zeta^2)\rho^\epsilon(\eta)^2(a_1^f)^2\langle \eta, e_1^f\rangle^2-2C_{L}\sqrt{\epsilon}h(\epsilon)(1-\zeta)a_1^f\rho^\epsilon(\eta)^2\\&
					\geq 2(\zeta-2\zeta^2)\rho^\epsilon(\eta)^2(a_1^f)^2\langle \eta, e_1^f\rangle^2-2C_{L}\sqrt{\epsilon}h(\epsilon)(1-\zeta)a_1^f\rho^\epsilon(\eta)^2
					\end{aligned}
					\end{equation*}
					where we used that $h(\epsilon)^{2(\kappa-1)}\leq h(\epsilon)^{2(\kappa+\alpha-1)}\leq a_1^f/2$ to obtain the last line. The estimate then follows since the first term on the last line is nonnegative.
					
					\noindent $(ii)$ Continuing from the last inequality we have
					\begin{equation*}
					\begin{aligned}
					\mathfrak{H}_{x^*}^{\epsilon}(U^\delta,Z)(t,\eta)&
					\geq  2\rho^\epsilon(\eta)^2a_1^f\bigg((\zeta_0-2\zeta_0^2)a_1^fh(\epsilon)^{-2(\kappa+\alpha)}-C_{L}\sqrt{\epsilon}h(\epsilon)\bigg)
					\end{aligned}
					\end{equation*}
					and the estimate follows by letting $\epsilon$ small enough to have $\sqrt{\epsilon}h(\epsilon)^{1+2(\kappa+\alpha)}\leq \sqrt{\epsilon}h^3(\epsilon)\leq(\zeta_0-2\zeta_0^2)a_1^f/C_L. $



\newpage
\bibliographystyle{plain}
\nocite{*}
\bibliography{ImportancerefNew}

\end{document}